\documentclass[a4paper,  leqno]{amsart}

\let\oldtocsection=\tocsection
 \let\oldtocsubsection=\tocsubsection
 \let\oldtocsubsubsection=\tocsubsubsection
 
\renewcommand{\tocsection}[2]{\hspace{0em}\oldtocsection{#1}{#2}}
\renewcommand{\tocsubsection}[2]{\hspace{1.8em}\oldtocsubsection{#1}{#2}}
\renewcommand{\tocsubsubsection}[2]{\hspace{4em}\oldtocsubsubsection{#1}{#2}}

\usepackage[
pdfstartview=FitH, CJKbookmarks=true, bookmarksnumbered=true,
bookmarksopen=true, colorlinks=true, citecolor=black ]{hyperref}
\usepackage{graphicx}
\usepackage{enumitem}
\usepackage[dvipsnames]{xcolor}
\usepackage{scrextend}
\addtokomafont{labelinglabel}{\sffamily}

\usepackage[normalem]{ulem} 
\usepackage{soulutf8} 
\setulcolor{red}
\setstcolor{blue}
\sethlcolor{green}
\setlist{nolistsep}

\usepackage{graphicx}
\usepackage{amsmath,amsthm,amsfonts,amscd,amssymb,fixmath}
\usepackage[all,2cell]{xy}
\UseAllTwocells
\SilentMatrices
\numberwithin{equation}{subsection}

\usepackage{stix}

\usepackage{fancyhdr}
\newenvironment{enum2}{
	\begin{enumerate}[label={$\mathrm{(\arabic*)}$},leftmargin=1.8em]
  \setlength{\itemsep}{0pt}
  \setlength{\parskip}{0pt}
  \setlength{\parsep}{0pt}
  \setlength{\topsep}{0pt}
}{\end{enumerate}}
\newenvironment{enumi}
{\begin{enumerate}[itemsep=0pt,  parsep=0pt, topsep=0pt, label={$\mathrm{(\roman*)}$}, font=\normalfont,leftmargin=1cm]}
{\end{enumerate}}

\newlength{\mylen}
\setbox1=\hbox{$\bullet$}\setbox2=\hbox{\tiny$\bullet$}
\renewcommand\labelitemi{\raisebox{\mylen}{\tiny$\bullet$}}
\def\tinybullet{\labelitemi}

\theoremstyle{plain}
\newtheorem{thm}{Theorem}[section]
\newtheorem{prop}[thm]{Proposition}

\newtheorem{conj}[thm]{Conjecture}
\newtheorem{cor}[thm]{Corollary}
\newtheorem{assertion}[thm]{Assertion}
\newtheorem{lem}[thm]{Lemma}
\theoremstyle{definition}

\newtheorem{exm}[thm]{Example}
\theoremstyle{remark}
\newtheorem{rmk}[thm]{Remark}

\newcommand{\ra}{\rightarrow}

\newcommand{\lra}{\longrightarrow}
\newcommand{\lla}{\longleftarrow}
\newcommand{\pair}[1]{\langle #1\rangle}
\newcommand{\epic}{\twoheadrightarrow }

\newcommand{\monic}{\hookrightarrow}

\newcommand{\Ra}{\Rightarrow}

\def\cra{\rightsquigarrow}
\def\C{\mathcal{C}}

\newcommand{\CC}{\mathcal{C}}
\renewcommand{\AA}{\mathbb{A}}
\def\E{\mathcal{E}}

\newcommand{\X}{\mathbb{X}}
\newcommand{\XX}{\mathcal{X}}
\newcommand{\YY}{\mathcal{Y}}
\newcommand{\K}{\mathcal{K}}
\def\B{\mathcal{B}}
\newcommand{\A}{\mathcal{A}}
\newcommand{\D}{\mathcal{D}}
\newcommand{\F}{\mathcal{F}}
\def\cH{\mathcal{H}}

\def\T{\mathcal{T}}
\newcommand{\SSS}{\mathcal{S}}

\newcommand{\R}{\mathbb{R}}
\newcommand{\Q}{\mathbb{Q}}

\renewcommand{\SS}{\mathbb{S}}
\newcommand{\Z}{\mathbb{Z}}
\renewcommand{\P}{\mathbb{P}}
\renewcommand{\O}{\mathcal{O}}
\def\U{\mathcal{U}}
\def\V{\mathcal{V}}
\def\w{\mathbf{w}}

\newcommand{\id}{\textup{id}}

\newcommand{\iso}{\cong}
\renewcommand{\hom}{\textup{Hom}}
\newcommand{\End}{\textup{End}}
\newcommand{\ext}{\textup{Ext}}
\newcommand{\coh}{\textup{coh}}
\def\Qcoh{\textup{Qcoh}}
\newcommand{\cone}{\textup{cone}}
\newcommand{\cocone}{\textup{co-cone}}
\newcommand{\vect}{\textup{vect}}
\newcommand{\rk}{\textup{rk}}
\newcommand{\im}{\textup{im}}

\renewcommand{\ker}{\textup{ker}}
\renewcommand{\mod}{\textup{mod}}
\newcommand{\proj}{\textup{proj}}
\newcommand{\add}{\textup{add}}
\newcommand{\nilp}{\textup{nilp}}

\renewcommand{\deg}{\textup{deg}}
\newcommand{\pic}{\textup{Pic}}
\newcommand{\Aut}{\textup{Aut}}
\newcommand{\lcm}{\textup{lcm}}

\newcommand{\udp}{\underline{p}}
\newcommand{\udl}{\underline{\lambda}}
\newcommand{\nec}{$(\Rightarrow)\,$}
\newcommand{\suf}{$(\Leftarrow)\,$}
\renewcommand{\top}{\textup{top}}
\newcommand{\soc}{\textup{soc}}
\newcommand{\aut}{\textup{Aut}}

\newcommand{\ev}{\textup{ev}}
\newcommand{\tstr}{(\D^{\leq 0}, \D^{\geq 0})}

\def\DD{\mathbb{D}}
\def\chom{\mathcal{H}om}
\def\std{\textup{std}}
\def\op{\textup{op}}
\def\dim{\textup{dim}}
\def\gcd{\textup{gcd}}
\def\Tr{\textup{Tr}}
\title{Bounded t-structures on the bounded derived category of coherent sheaves over a weighted projective line}
\author{Chao Sun}
\address{School of Mathematical Sciences \\
         University of Science and Technology of China \\
        Hefei, Anhui 230026 \\
        P. R. China}
\email{maizisc@mail.ustc.edu.cn}
\date{\today}

\usepackage{lineno}

\begin{document}
\begin{abstract}
	We use recollement and HRS-tilt to describe bounded t-structures on the bounded derived category $\D^b(\X)$ of coherent sheaves over a weighted projective line $\X$ of domestic or tubular type. We will see from our description that the combinatorics in the classification of bounded t-structures on $\D^b(\X)$ can be reduced to that in the classification of bounded t-structures on the bounded derived categories of finite dimensional right modules over representation-finite  finite dimensional hereditary algebras.
\end{abstract}

\date{\today}
\maketitle

\section{Introduction}
\subsection{Background and aim}
In an attempt to give a geometric treatment of Ringel's canonical algebras \cite{Ringel}, Geigle and Lenzing introduced in  \cite{GL} a class of  noncommutative curves, called weighted projective lines, and 
each canonical algebra is realized as the endomorphism algebra of a  tilting bundle in the category of coherent sheaves over some weighted projective line. A stacky point of view to weighted projective lines is that for a weighted projective line $\X$ defined over a field $k$, there is a smooth algebraic $k$-stack $\XX$ with the projective line over $k$ as its coarse moduli space such that $\coh\XX\simeq \coh\X$ and $\Qcoh\XX\simeq \Qcoh\X$, where $\coh$ (resp. $\Qcoh$) denotes the category of coherent (resp. quasi-coherent) sheaves. 
As an indication of the importance of the notion of weighted projective lines, 
a famous theorem of Happel~\cite{H} states that if $\A$ is a 
connected hereditary category linear over an algebraically closed field $k$ with finite dimensional morphism and extension spaces such that its bounded derived category $\D^b(\A)$ admits a tilting object then $\D^b(\A)$  is triangle equivalent to the bounded derived category of finite dimensional modules over a finite dimensional hereditary algebra over $k$ or to the bounded derived category of coherent sheaves on a weighted projective line defined over $k$. 

The notion of t-structures is introduced by Beilinson, Bernstein and Deligne in \cite{BBD} to serve as a categorical framework for defining perverse sheaves in the derived category of constructible sheaves over a stratified space. Recently, there has been a growing interest in t-structures 
ever since  Bridgeland~\cite{Br} introduced the notion of stability conditions.  To give a stability condition on a triangulated category requires specifying a bounded t-structure. 
On the other hand, there are many works on bounded t-structures on the bounded derived category $\D^b(\Lambda)$ of finite dimensional  modules over a finite dimensional algebra $\Lambda$ in recent years.
Remarkably, K\"onig and Yang 
proved the existence of bijective correspondences, which we call K\"onig-Yang correspondences,  between several concepts among which are bounded t-structures with length heart on $\D^b(\Lambda)$,  simple-minded collections in $\D^b(\Lambda)$, silting objects in  $\K^b(\proj\Lambda)$,  and co-t-structures on $\K^b(\proj\Lambda)$, where $\K^b(\proj\Lambda)$ denotes the bounded homotopy category  of finite dimensional projective modules over $\Lambda$.

This article is devoted to describing bounded t-structures on the bounded derived category of coherent sheaves over a weighted projective line. 
We mainly combine two classical tools to describe t-structures: recollement and HRS-tilt. 
Recollement is introduced at the same time with t-structures in \cite{BBD}. 
A recollement stratifies a triangulated category into smaller ones and allows us to glue t-structures. 
HRS-tilt, introduced by Happel, Reiten and Smal{\o} in \cite{HRS}, constructs a new t-structure from an old one via a torsion pair in the heart of the old t-structure. 
We will see that a large class of t-structures are glued from recollements. Given a t-structure, to build a recollement from which the t-structure can be glued, we rely on Ext-projectives.  
This concept was introduced by Auslander and Smal{\o}  to investigate almost split sequences in subcategories \cite{AS}.  Assem, Salario and Trepode introduced a triangulated version in \cite{ASS} to study t-structures. 
Our small observation is that an exceptional Ext-projective object helps us to build a desired recollement under some condition (see Lemma~\ref{ext-proj recollement}).  
Almost all recollements in this article are built in this way (plus induction).  
There do exist  bounded t-structures without any available Ext-projective. Fortunately, in our situation, these are up to shift HRS-tilts with respect to some torsion pair in the standard heart  and they can be described explicitly.

\subsection{Main results}

Let $\X$ be a weighted projective line defined over an algebraic closed field $k$, and $\O$  its structure sheaf   (see \S \ref{sec: wpl def}). Depending on its weight function $\w: \P^1\ra \Z_{\geq 1}$, where $\P^1$ is (the set of closed points of) the projective line over $k$ and $\Z_{\geq 1}$ is the set of positive integers, $\X$ is of domestic type,  of tubular type, or of wild type.
Denote by $\vect\X$ resp. $\coh_0\X$ the category of vector bundles resp. torsion sheaves over $\X$, by $\A=\coh\X$ the category of coherent sheaves and  by $\D=\D^b(\X)$ the bounded derived category of $\coh\X$. 
$\coh_0\X$ consists exactly of finite length objects in $\coh\X$ and $\coh_0\X$ decomposes as a coproduct $\coh_0\X=\coprod_{\lambda\in \P^1}\coh_\lambda\X$, where $\coh_\lambda\X$ consists of those coherent sheaves supported at $\lambda$. 
For $P\subset \P^1$, denote by $(\T_P,\F_P)$ the torsion pair in $\coh\X$
\[(\add\{\coh_\lambda\X\mid \lambda\in P\}, \,\,\add\{\vect\X, \coh_\lambda\X\mid \lambda\in \P^1\backslash P\}).\]  
The number of isoclasses of simple sheaves in $\coh_\lambda\X$  is $\w(\lambda)$. A (possibly empty) collection $\SSS$ of simple sheaves over $\X$ is called proper if for each $\lambda\in \P^1$, $\SSS$ does not contain a complete set of simple sheaves in $\coh_\lambda\X$ and if simple sheaves in $\SSS$ are pairwise non-isomorphic. 
Two such collections are equivalent if they yield the same isoclasses of simple sheaves. A t-structure on $\D^b(\X)$ is said to be compatible with a given a recollement if it is glued from the recollement (see \S\ref{sec: glue t-str}). 
 See \S\ref{sec: notation} for the notation  $\pair{-}_\D$, $(-)^{\perp_\A}$, $(-)^{\perp_\D}$ and $\D^b(-)$.

 We are ready to state our theorem for a weighted projective line of domestic type.
\begin{thm}[Theorem~\ref{thm: domestic}]\label{thm1}
Suppose $\X$ is of domestic type  and let $\tstr$ be a bounded t-structure on $\D^b(\X)$ with heart $\B$. Then exactly one of the following holds: 

	\begin{enum2}
	\item up to the action of the Picard group $\pic\X$ of $\X$, $(\D^{\leq 0}, \D^{\geq 0})$ is compatible with the recollement \[ \xymatrix{\O^{\perp_\D}  \ar[rr]|{i_*}   & &\ar@/_1pc/[ll] \ar@/^1pc/[ll]\D=\D^b(\X) \ar[rr] & &\ar@/_1pc/[ll]|{j_!} \ar@/^1pc/[ll] \pair{\O}_\D,}\] 
	where $i_*,j_!$ are the inclusion functors, in which case $\B$ is of finite length;

\item for a unique (up to equivalence) proper collection $\SSS$ of  simple sheaves and a unique $P\subset \P^1$, $(\D^{\leq 0}, \D^{\geq 0})$ is compatible with the recollement
	\[ \xymatrix{\D^b(\SSS^{\perp_\A})=\SSS^{\perp_\D}  \ar[rr]|{i_*}   & &\ar@/_1pc/[ll] \ar@/^1pc/[ll]\D=\D^b(\X) \ar[rr] & &\ar@/_1pc/[ll]|{j_!} \ar@/^1pc/[ll] \pair{\SSS}_\D,}\] 
	where $i_*,j_!$ are the inclusion functors, 
	such that the corresponding t-structure on $\D^b(\SSS^{\perp_\A})$ is a shift of the HRS-tilt with respect to  the torsion pair $(\SSS^{\perp_\A}\cap \T_P, \SSS^{\perp_\A}\cap \F_P)$ in $\SSS^{\perp_\A}$, 
	in which case $\B$ is not of finite length and $\B$ is noetherian resp. artinian iff $P=\emptyset$ resp. $P=\P^1$.
	\end{enum2}
\end{thm}
To state our theorem for a weighted projective line  of tubular type, we need to introduce more notation  (see \S\ref{sec: wpl vect}).  Let $\R$ (resp. $\Q$) be the set of real (resp. rational) numbers and let $\bar{\R}=\R\cup \{\infty\}$, $\bar{\Q}=\Q\cup \{\infty\}$.  Let $\X$ be of tubular type. 
Denote by $\coh^\mu\X$ the category of semistable coherent sheaves over $\X$ with slope $\mu\in \bar{\Q}$ (we deem torsion sheaves to be semistable and thus $\coh^\infty\X=\coh_0\X$).
 $\D^b(\X)$  admits an exact autoequivalence $\Phi_{q',q}$ for each $q',q\in \Q\cup \{\infty\}$, which is called a telescopic functor, such that $\Phi_{q',q}(\coh^q\X)=\coh^{q'}\X$. 
For $\mu\in \Q$,  denote $\coh^\mu_\lambda \X=\Phi_{\mu,\infty}(\coh_\lambda\X)$. The category $\coh^\mu\X$  decomposes as $\coh^\mu\X=\coprod_{\lambda\in\P^1} \coh^\mu_\lambda\X$.  For $\mu\in\bar{\R}$, $\coh^{>\mu}\X$ (resp. $\coh^{<\mu}\X$) denotes the subcategory of $\coh\X$ consisting of those sheaves whose semistable factors have slope $>\mu$ (resp. $<\mu$).  
\begin{thm}[Theorem~\ref{thm: tubular}]\label{thm2}
Suppose $\X$ is of tubular type and let 
	$\tstr$ be a bounded t-structure on $\D^b(\X)$ with heart $\B$. Then exactly one of the following holds: 

	\begin{enum2}
	\item for a unique $\mu\in\R\backslash\Q$, $\tstr$ is a shift of the HRS-tilt with respect to the torsion pair $(\coh^{>\mu}\X, \coh^{<\mu}\X)$ in $\coh\X$, in which case $\B$ is neither noetherian nor artinian;

	\item for a unique $\mu\in\bar{\Q}$ and a unique $P\subset\P^1$, $\tstr$ is a shift of the HRS-tilt with respect to the torsion pair 
		\[(\add\{\coh^{>\mu}\X, \coh_\lambda^\mu\X\mid \lambda\in P\},\,\,\add\{\coh_\lambda^\mu\X,  \coh^{<\mu}\X\mid \lambda\in \P^1\backslash P\})\] in $\coh\X$, in which case $\B$ is not of finite length and $\B$ is noetherian resp. artinian iff $P=\emptyset$ resp. $P=\P^1$;
	
		\item for a unique $q\in \bar{\Q}$,  a unique (up to equivalence) nonempty proper collection $\SSS$ of  simple sheaves and a unique $P\subset \P^1$, $\Phi_{\infty,q}((\D^{\leq 0}, \D^{\geq 0}))$ is compatible with the recollement  
			\[ \xymatrix{ \D^b(\SSS^{\perp_\A})=\SSS^{\perp_\D}  \ar[rr]|{i_*}   & &\ar@/_1pc/[ll] \ar@/^1pc/[ll]\D=\D^b(\X) \ar[rr] & &\ar@/_1pc/[ll]|{j_!} \ar@/^1pc/[ll] \pair{\SSS}_\D,}\]
	where $i_*,j_!$ are the inclusion functors, such that the corresponding t-structure on $\D^b(\SSS^{\perp_\A})$ is a shift of the HRS-tilt with respect to the torsion pair $(\SSS^{\perp_\A}\cap \T_P,\SSS^{\perp_\A}\cap \F_P)$ in $\SSS^{\perp_\A}$, 
	in which case $\B$ is not of finite length and $\B$ is noetherian resp. artinian iff $P=\emptyset$ resp. $P=\P^1$; 
	
	\item for some $q\in \bar{\Q}$ and some exceptional simple sheaf $S$,  $\Phi_{\infty,q}((\D^{\leq 0},\D^{\geq 0}))$ is compatible with the recollement  
		\[ \xymatrix{ \D^b(S^{\perp_\A})=S^{\perp_\D}  \ar[rr]|{i_*}   & &\ar@/_1pc/[ll] \ar@/^1pc/[ll] \D=\D^b(\X) \ar[rr] & &\ar@/_1pc/[ll]|{j_!} \ar@/^1pc/[ll] \pair{S}_\D,}\]
	where $i_*,j_!$ are the inclusion functors, such that the corresponding t-structure on $\D^b(S^{\perp_\A})$ has length heart, in which case $\B$ is of finite  length.
	\end{enum2}
\end{thm}

We obtain from the  two theorems above certain bijective correspondence for those bounded t-structures whose heart is not of finite length. Note that any  group $G$ of exact autoequivalences of $\D^b(\X)$ acts on the set of bounded t-structures on $\D^b(\X)$ by $\Phi((\D^{\leq 0}, \D^{\geq 0})):=(\Phi(\D^{\leq 0}), \Phi(\D^{\geq 0}))$ for  $\Phi \in G$ and a bounded t-structure $(\D^{\leq 0}, \D^{\geq 0})$ on $\D^b(\X)$. In the following corollary, we deem $\Z$ as the group of exact autoequivalences generated by the translation functor of $\D^b(\X)$, which acts freely on the set of bounded t-structures on $\D^b(\X)$. 

\begin{cor}[{Corollary~\ref{bijection for not length heart}}]\label{cor}
	\begin{enum2}
	 \item If $\X$ is of domestic type then there is a bijection
	\begin{multline}\label{bijection for not length}
	\{\text{bounded t-structures on $\D^b(\X)$ whose heart is not of finite  length}\}/\Z \longleftrightarrow\\
	\bigsqcup_\SSS\left( \{P\mid P\subset \P^1\}\times \{\text{bounded t-structures on $\pair{\SSS}_\D$}\}\right),
	\end{multline}
	where 
	$\SSS$ runs through all equivalence classes of proper collections of simple sheaves.  

	\item If $\X$ is of tubular type then there is a bijection
	\begin{multline}
		\{\text{bounded t-structures on $\D^b(\X)$ whose heart is not of finite length}\}/\Z \longleftrightarrow\\
		\R\backslash \Q\bigsqcup \left(\bar{\Q}\times \bigsqcup_{\SSS}\left( \{P\mid P\subset \P^1\}\times \{\text{bounded t-structures on $\pair{\SSS}_\D$}\}\right)\right),
	\end{multline}
	where 
	$\SSS$ runs through all equivalence classes of proper collections of simple sheaves.  
\end{enum2}

\end{cor}
Recall that an equioriented $\AA_s$-quiver refers to the  quiver \[\underset{1}{\tinybullet} \lra \underset{2}{\tinybullet} \lra\tinybullet \dots\tinybullet \lra \underset{s-1}{\tinybullet} \lra \underset{s}{\tinybullet}.\] (Since only such an orientation is involved in this article, $\vec{\AA}_s$ will always denote an equioriented $\AA_s$-quiver.) For convenience, we also define $\vec{\AA}_0$ to be the empty quiver and define $\mod\,k\vec{\AA}_0$ to be the zero category. 
Given a nonempty proper collection $\SSS$ of simple sheaves on $\X$, there are positive integers $m, k_1,\dots, k_m$ such that $\pair{\SSS}_\A\simeq \coprod_{i=1}^m\mod k\vec{\AA}_{k_i}$, where $\mod k\vec{\AA}_l$ is the category of finite dimensional right modules over the path algebra of the equioriented $\AA_l$-quiver,
and we have an exact equivalence $\pair{\SSS}_\D\simeq\coprod_{i=1}^m \D^b(\mod k\vec{\AA}_{k_i})$. By Corollary~\ref{cor}, if $\X$ is a weighted projective line  of domestic or tubular type then to classify bounded t-structures on $\D^b(\X)$ whose heart is not of finite length, it sufficies to classify bounded t-structures on each  $\D^b(\mod k\vec{\AA}_{k_i})$. Since bounded t-structurs on $\D^b(\mod k\vec{\AA}_l)$ have length heart, one can achieve this by calculating silting objects or simple-minded collections in $\D^b(\mod k\vec{\AA}_{k_i})$ by virtue of K\"onig-Yang correspondences. 
We know that $\D^b(\X)$ is triangle equivalent to the bounded derived category of finite dimensional right modules over a canonical algebra whose global dimension is at most $2$. So to obtain a bijective correspondence for bounded t-structures on $\D^b(\X)$ with length heart, we can again utilize K\"onig-Yang correspondences and try to compute  collections of simple objects in the heart (using Proposition~\ref{simple in heart}) or  silting objects in $\D^b(\X)$ (using \cite[Corollary 3.4]{LVY}) from the recollements in Theorem~\ref{thm1}(1) and Theorem~\ref{thm2}(4). 
 As illustrated after Corollary~\ref{bijection for not length heart} in \S\ref{sec: thm}, the two theorems reduce the combinatorics in the classification of bounded t-structures on $\D^b(\X)$ to the combinatorics in the classification of bounded t-structures on bounded derived categories of finite dimensional modules over representation-finite  finite dimensional hereditary algebras. 

To give an application of our description of bounded t-structures, we prove in \S\ref{sec: der equiv}  
a characterization of  when the  heart of a bounded t-structure on $\D^b(\X)$ is derived equivalent to the standard heart $\coh\X$, which is inspired by the work \cite{SR} of Stanley and van Roosmalen. 

\begin{thm}[{Theorem~\ref{der equiv}}]
	Let  $\X$ be a weighted projective line of domestic or tubular type and $(\D^{\leq 0}, \D^{\geq 0})$ a bounded t-structure on $\D^b(\X)$ with heart $\B$. Then the inclusion $\B\ra \D^b(\X)$ extends to a derived equivalence $\D^b(\B)\overset{\sim}{\ra} \D^b(\X)$ iff the Serre functor of $\D^b(\X)$ is right t-exact with respect to $(\D^{\leq 0}, \D^{\geq 0})$. 
\end{thm}
Here we say that the inclusion $\B\ra \D^b(\X)$ extends to a derived equivalence $\D^b(\B)\overset{\sim}{\ra} \D^b(\X)$ if some realization functor $\D^b(\B)\ra \D^b(\X)$ is an equivalence (see \S\ref{sec: der equiv}). As a corollary (see Corollary~\ref{tubular serre}), a similar assertion holds for the bounded derived category of finite dimensional right modules over a tubular algebra in the sense of Ringel \cite{Ringel}.

\subsection{Sketch of this article}
This article is organized as follows. 

In \S2, we collect preliminaries on t-structures and some facts on hereditary categories.  In \S2.1-2.2, we recall basic definitions and properties of t-structures and introduce width-bounded t-structures and HRS-tilt. 
In \S3.2-3.5, we recall recollements of triangulated categories, admissible subcategories, gluing t-structures and properties of glued t-structures. 
In \S2.6, we recall Ext-projective objects, and use an exceptional Ext-projective object to establish a recollement with which the given  t-structure is compatible. In \S2.7, we recall some facts on hereditary categories, including Happel-Ringel Lemma. 
In \S2.8, we recall and prove some facts on t-structures on the bounded derived category of finitely generated modules over a finite dimensional algebra, including  a part of K\"onig-Yang correspondences. 
In \S2.9, we describe bounded t-structures on the bounded derived category of finite dimensional nilpotent representations of a cyclic quiver.

In \S3, we collect preparatory materials and results on weighted projective lines.
In \S3.1, we recall basic definitions and facts on weighted projective lines. In \S3.2, we recap Auslander-Reiten theory. In \S3.3, we recall the classification and important properties of vector bundles over a weighted projective line of domestic or tubular type. In \S3.4, we recall descriptions of perpendicular categories of some exceptional sequences. In \S3.5, we recall and prove the non-vanishing of some morphism spaces in the category $\coh\X$ of coherent sheaves over a weighted projective line $\X$.  In \S3.5, we investigate full exceptional sequences in $\coh\X$, and prove the existence of certain nice terms in some cases. In \S3.7, we give some preliminary descriptions of some torsion pairs in $\coh\,\X$, and establish bijections between isoclasses of basic tilting sheaves, certain torsion pairs in $\coh\X$ and certain bounded t-structures on the bounded derived category $\D^b(\X)$ of $\coh\X$, and finally we investigate the Noetherianness and the Artinness of tilted hearts given by certain torsion pairs in $\coh\X$.

In \S4, we describe bounded t-structures on the bounded derived category $\D^b(\X)$ of coherent sheaves over a weighted projective line $\X$ of domestic or tubular type. In \S4.1, we investigate and describe bounded t-structures that restrict to bounded t-structures on the bounded derived category $\D^b(\coh_0\X)$ of the category $\coh_0\X$ of torsion sheaves. In \S4.2, we investigate those bounded t-structures on $\D^b(\X)$ that cannot restrict to t-structures on $\D^b(\coh_0\X)$ even up to the action of the group of exact autoequivalences of $\D^b(\X)$. In particular, we prove that the heart of such a bounded t-structure is necessarily of finite length and possesses only finitely many indecomposable objects, all of which are exceptional. In \S4.3, we prove some properties possessed by silting objects in $\D^b(\X)$.  This is mainly acquired via properties of full exceptional sequences obtained earlier and will yield information on bounded t-structures by virtue of K\"onig-Yang correspondences. In \S4.4, we complete our description of bounded t-structures on $\D^b(\X)$, in which we mainly use HRS-tilt and recollement. In \S4.5, we use our description of bounded t-structures to give a description of torsion pairs in $\coh\,\X$.

In \S5, we  prove a characterization of when the heart of a bounded t-structure $\tstr$ on $\D^b(\X)$ is derived equivalent to $\coh\,\X$ for a domestic or tubular $\X$, which is pertinent to the right t-exactness of the Serre functor of $\D^b(\X)$ and gives an application of our main result (i.e, description of bounded t-structures). We conjecture that this result holds for arbitrary weighted projective line and propose a potential approach at the end of \S5. 

\subsection{Notation and conventions}\label{sec: notation}

We denote by $\R$ (resp. $\Q$, $\Z$, $\Z_{\geq 1}$)  the set of real numbers (resp. rational numbers,  integers, positive integers). Pose $\bar{\R}=\R\cup \{\infty\}$ and $\bar{\Q}=\Q\cup\{\infty\}$.

For a finite dimensional algebra $\Lambda$ over a field $k$, $\mod \Lambda$ denotes  the category of finite dimensional  right  modules over $\Lambda$ and $\D^b(\Lambda)$  the bounded derived category of $\mod \Lambda$.

A subcategory of a category  is tacitly a full subcategory.  If $\B$ is a subcategory of a  category $\A$ (typically abelian or triangulated in our setup), denote 
\[\B^{\perp_{0,\A}}=\{X\in \A\mid \hom_\A(\B,X)=0\},\] 
which we will simply write as $\B^{\perp_0}$ if there is no confusion. Dually we have ${}^{\perp_{0,\A}}\B$ or ${}^{\perp_0}\B$. 
	
For an abelian category $\A$, its bounded derived category is denoted by $\D^b(\A)$. Let $\B$ be an additive subcategory  of $\A$. Following \cite{GL2}, we call $\B$  an \emph{exact subcategory}\footnote{Note the difference with a subcategory that is an exact category.}  of $\A$ if $\B$ is an abelian category and the inclusion functor $\iota: \B\ra \A$ is exact. 
$\B$ is called a \emph{thick subcategory} of $\A$ if $\B$ is closed under kernel, cokernel and extension. A thick subcategory of $\A$ is an exact subcategory of $\A$.
Given a collection  $\C$ of objects in $\A$, we  denote by $\pair{\C}_\A$ the smallest thick subcategory of $\A$ containing $\C$. 
The \emph{right perpendicular category} $\C^{\perp_\A}$ and the \emph{left perpendicular category} ${}^{\perp_\A}\C$ of $\C$ in the sense of {\cite{GL2}} are
\[
	\begin{aligned}
		\C^{\perp_\A} & =\{X\in \A\mid \hom_\A(C,X)=0=\ext^1_\A(C,X)\,\, \text{for all}\,\,C\in \C\},\\
		{}^{\perp_\A} \C & =\{X\in \A\mid \hom_\A(X,C)=0=\ext^1_\A(X,C)\,\,\text{for all}\,\,C\in \C\}.
	\end{aligned}
\]
It's shown in \cite[Proposition 1.1]{GL2} that if objects in $\C$ have projective dimension at most $1$, that is, $\ext^2_\A(X,-)=0$ for all $X\in \C$, then $\C^{\perp_\A}$ and ${}^{\perp_\A}\C$ are exact subcategories of $\A$ closed under extension. 

Let $\D$ be a triangulated category. We denote by $\aut \D$ the group of exact autoequivalences of  $\D$. A triangle in  $\D$ refers always to a distinguished triangle.  
For two  subcategories $\D_1,\D_2$  of $\D$, define a subcategory $\D_1*\D_2$ of $\D$ by
	\[\D_1*\D_2=\{X\in \D\mid \exists \,\text{a  triangle}\,Y\ra X\ra Z\cra , Y\in \D_1, Z\in\D_2\}.\]
 By the octahedral axiom, $*$ is associative.  
 Given a triangulated category $\D$ and a collection $\C$ of objects in $\D$, we denote by $\pair{\C}_\D$ the thick closure of $\C$ in $\D$, that is, the smallest triangulated subcategory of $\D$ containing $\C$ and closed under direct summand. We say that $\C$ classically generates $\D$ if $\pair{\C}_\D$ coincides with $\D$. 
 Moreover, we denote 
 \[\C^{\perp_\D}=\C^{\perp_\D} : =\{X\in \D\mid \hom_\D^n(\C,X)=0,\,\,\forall n\in \Z\}=\pair{\C}_\D^{\perp_0}.\]
 Dually one defines ${}^{\perp}\C={}^{\perp_\D}\C$. $\C^{\perp_\D}$ and ${}^{\perp_\D}\C$ are thick subcategories of $\D$. 
If $\D$ is a triangulated category linear over a field $k$,  we denote 
\[\hom^\bullet(X,Y)=\oplus_{n\in\Z}\hom^n(X,Y)[-n],\]
where the latter is deemed as a complex of $k$-spaces with zero differential. $\D$ is said to be \emph{of finite type} if $\oplus_{n\in\Z}\hom^n(X,Y)$ is a finite dimensional $k$-space for each $X,Y$ in $\D$.

If $\A$ is a hereditary abelian category and $\B$ is an exact subcategory of $\A$ closed under extension then   $\B$ is a hereditary abelian category and the inclusion functor $\iota: \B\ra \A$ induces a fully faithful exact functor $\D^b(\iota): \D^b(\B)\ra \D^b(\A)$ whose essential image consists of those objects in $\D^b(\A)$ with cohomologies in $\B$.\footnote{One can argue as follows for this simple fact.
 By [8, Lemma 3.2.3], we have an injection $\ext^2_\B(X,Y)\monic \ext^2_\A(X,Y)$ for $X,Y\in \B$. Since $\A$ is hereditary, $\ext^2_\B(X,Y)=0$. 
 So $\B$ is hereditary.  Since the exact subcategory $\B$ is closed under extension, the inclusion $\iota: \B\ra \A$ induces an isomorphism $\ext^1_\B(X,Y)\cong  \ext^1_\A(X,Y)$ for any $X,Y\in \B$. 
 Since $\B$ classically generates $\D^b(\B)$, the derived functor $\D^b(\iota): \D^b(\B)\ra \D^b(\A)$ is  fully faithful. The essential image of $\D^b(\iota)$ is clear.}
Denote $\D=\D^b(\A)$. If $\C$ is a collection of objects in $\A$ then $\B:=\pair{\C}_\A$ (resp. $\B:=\C^{\perp_\A}$, resp. $\B:={}^{\perp_\A}\C$) is an exact subcategory of $\A$ closed under extension and the functor $\D^b(\iota): \D^b(\B)\ra \D^b(\A)$ identifies canonically $\D^b(\pair{\C}_\A)$ (resp. $\D^b(\C^{\perp_\A})$, resp. $\D^b({}^{\perp_\A}\C)$) with the subcategory $\pair{\C}_\D$ (resp. $\C^{\perp_\D}$, resp. ${}^{\perp_\D}\C$) of $\D$. We will often make this identification in this article.

\subsection{Acknowledgements}

The question of this article originated from a seminar on Bridgeland's stability conditions  organized by Prof. Xiao-Wu Chen, Prof. Mao Sheng and Prof. Bin Xu. I thank these organizers who gave me the opportunity to report. I am grateful to the participants for their patience and critical questions. 
Thanks are once again due to Prof. Xiao-Wu Chen, my supervisor, for his guidance and kindness. 

 I thank Peng-Jie Jiao for discussion,  thank Prof. Helmut Lenzing for carefully answering my question on stable bundles over a tubular weighted projective line, thank Prof. Zeng-Qiang Lin for communication on realization functors, thank Prof. Hagen Meltzer for his lectures on weighted projective lines,    thank Prof. Dong Yang for explaning the results in \cite{KD} and for stimulating conversations, and thank Prof. Pu Zhang for a series of lectures on triangulated categories based on his newly-written book titled "triangulated categories and derived categories" (in Chinese). 

Moreover, I would like to express my deep gratitude to an anonymous referee for his/her long list of suggestions, which pointed out many mistakes and inaccuracies in an earlier version of this article and helped in improving the exposition and in reshaping some parts of this article.

This work is supported by the National Science Foundation of China (No. 11522113 and No. 115771329) and also by the Fundamental Research Funds for the Central Unviersities.

\section{Preliminaries} \label{cha t-str}

\subsection{Basics on t-structures} \label{sec: t-str}
We recall basic definitions concerning t-structures in this subsection. The standard reference is \cite{BBD}.

Let $\D$ be a triangulated category. 
	A \emph{t-structure} on $\D$ is a pair $(\D^{\leq 0}, \D^{\geq 0})$ of  strictly (=closed under isomorphism) full subcategories ($\D^{\leq n}:=\D^{\leq 0}[-n], \D^{\geq n}:=\D^{\geq 0}[-n]$)
	\begin{itemize}
	\item $\hom(\D^{\leq 0},\D^{\geq 1})=0$;
	\item $\D^{\leq -1}\subset \D^{\leq 0}$,  $\D^{\geq 1}\subset \D^{\geq 0}$;
	\item $\D=\D^{\leq 0}*\D^{\geq 1}$, i.e., for any object $X$ in $\D$, there exists a  triangle $A\ra X\ra B\cra $ with $A\in \D^{\leq 0}$ and $B\in \D^{\geq 1}$.  
	\end{itemize}
	For example, there is a \emph{standard $t$-structure} $(\D^b(\A)^{\leq 0},\D^b(\A)^{\geq 0})$  on the bounded derived category $\D^b(\A)$ of an abelian category $\A$  defined by \[\D^b(\A)^{\leq n}=\{K\in \D^b(\A)\mid H^i(K)=0, \forall i> n\},\] \[\D^b(\A)^{\geq n}= \{K\in\D^b(\A)\mid H^i(K)=0, \forall i<n\}.\]

 Given a t-structure $(\D^{\leq 0},\D^{\geq 0})$ on $\D$, the inclusion of $\D^{\leq n}$ (resp. $\D^{\geq n}$) into $\D$ admits a right (resp. left) adjoint $\tau_{\leq n}$ (resp. $\tau_{\geq n}$), which are called \emph{truncation functors.} Moreover, $\D^{\leq n}={}^{\perp_0}(\D^{\geq n+1})$, $\D^{\geq n}=(\D^{\leq n-1})^{\perp_0}$.  
$\D^{\leq n}$ is actually characterized by the property that it is a subcategory closed under suspension and extension for which the inclusion functor admits a right adjoint. A subcategory of $\D$ with such a property is called an \emph{aisle} \cite{KV}. 
A dual property characterizes $\D^{\geq n}$ and a subcategory of $\D$ with the dual property is called a \emph{co-aisle}.  There are bijections between t-structures, aisles and co-aisles, whence these notions are often used interchangeably.

 The \emph{heart} $\A$ of  $(\D^{\leq 0},\D^{\geq 0})$ is defined as the subcategory $\A:=\D^{\leq 0}\cap \D^{\geq 0}$.  
 $\A$ is 
 an  abelian subcategory of $\D$ and 
 we have a system $\{H^i\}$ of cohomological functors defined by \[H^i=\tau_{\geq 0}\tau_{\leq 0}(-[i]): \D\longrightarrow \A.\]
 $\D^{\leq 0}, \D^{\geq 0}$ and  $\A$ are closed under extension and direct summand.  Given a sequence $A\overset{f}{\ra}B\overset{g}{\ra} C$ of morphisms in $\A$,  $0\ra A\overset{f}{\ra} B\overset{g}{\ra} C\ra 0$ is a short  exact sequence  in $\A$ iff $A\overset{f}{\ra} B\overset{g}{\ra} C\overset{h}{\ra} A[1]$ is a triangle in $\D$ for some morphism $h:C\ra A[1]$ in $\D$.

 Denote $\D^{[m,n]}=\D^{\geq m}\cap \D^{\leq n}$.  An object $X\in \D$ lies in $\D^{[m,n]}$ iff $H^{l}(X)=0$ for  $l<m$ and $l> n$.  
 A $t$-structure $(\D^{\leq 0}, \D^{\geq 0})$ on $\D$ is called \emph{bounded} if $\D=\bigcup_{m,n\in\Z} \D^{[m,n]}$. 
 A bounded t-structure $\tstr$ is determined by its heart $\A$. In fact,
 \[\D^{\leq 0}=\cup_{n\geq 0} \A[n]*\A[n-1]*\dots *\A, \]
 \[\D^{\geq 0}=\cup_{n\leq 0} \A*\dots *\A[n+1]*\A[n].\]
 We will also denote by $(\D^{\leq 0}_\A, \D^{\geq 0}_\A)$ the bounded t-structure with heart $\A$. 

 Any group of exact autoequivalences of $\D$ acts on the set of t-structures. Given a t-structure $(\D^{\leq 0},\D^{\geq 0})$ on $\D$ and  an exact autoequivalence $\Phi$ of $\D$, \[\Phi((\D^{\leq 0},\D^{\geq 0})):= (\Phi(\D^{\leq 0}),\Phi(\D^{\geq 0}))\] is a t-structure on $\D$. 
 $\Phi((\D^{\leq 0}, \D^{\geq 0}))$ is bounded iff so is $(\D^{\leq 0},\D^{\geq 0})$. 

 Suppose $F:\D_1\ra \D_2$ is an exact functor between  two triangulated categories $\D_i$ ($i=1,2$) equipped with t-structures $(\D^{\leq 0}_i, \D^{\geq 0}_i)$. We say that $F$ is \emph{right t-exact} if $F(\D_1^{\leq 0})\subset \D_2^{\leq 0}$, \emph{left t-exact} if $F(\D_1^{\geq 0})\subset \D_2^{\geq 0}$, and \emph{t-exact} if it is both right and left t-exact.

 If $\C$ is a triangulated subcategory of  $\D$ and $(\D^{\leq 0},\D^{\geq 0})$ is a t-structure on $\D$, the pair \[(\C^{\leq 0},\C^{\geq 0}):=(\C\cap \D^{\leq 0},\C\cap \D^{\geq 0})\] gives a t-structure on $\C$ iff $\C$ is stable under some (equivalently, any) $\tau_{\leq l}$, i.e., $\tau_{\leq l}\C\subset \C$. Such a t-structure on $\C$ is called \emph{an induced t-structure by restriction}.

 \subsection{Width-bounded t-structures, HRS-tilt}\label{sec: HRS-tilt}
 
	 Let $(\D'^{\leq 0}, \D'^{\geq 0}), (\D^{\leq 0}, \D^{\geq 0})$ be two t-structures on a triangulated category $\D$. We say that $(\D'^{\leq 0}, \D'^{\geq 0})$ is \emph{width bounded}\footnote{
I learnt this notion from  Zeng-Qiang Lin's lectures on the paper \cite{Keller} of Keller. Moreover, Example~\ref{fin simple width bounded}(1) strenghtens slightly an example presented by him.}
	 with respect to $(\D^{\leq 0}, \D^{\geq 0})$ if $\D^{\leq m}\subset \D'^{\leq 0}\subset \D^{\leq n}$ for some $m,n$. 
 Define a relation $\sim$ on the
 set of t-structures: $(\D'^{\leq 0}, \D'^{\geq 0})\sim (\D^{\leq 0}, \D^{\geq 0})$ if $(\D'^{\leq 0}, \D'^{\geq 0})$ is width bounded with respect to $(\D^{\leq 0}, \D^{\geq 0})$. 

 \begin{lem}\label{width equiv relation}
	 $\sim$ is an equivalence relation.
\end{lem}
\begin{proof}
Reflexivity of $\sim$ is clear. One sees the symmetry of $\sim$ by noting that $\D^{\leq m}\subset \D'^{\leq 0}\subset \D^{\leq n}$ iff  $\D'^{\leq -n}\subset \D^{\leq 0}\subset \D'^{\leq -m}$ and sees the transitivity of $\sim$ by noting that $\D^{\leq m}\subset \D'^{\leq 0}\subset \D^{\leq n}$ iff  $\D^{\leq m}\subset \D'^{\leq 0}$ and $\D'^{\geq 0}\supset \D^{\geq n}$. 
\end{proof}
	
Obviously,  if $(\D'^{\leq 0}, \D'^{\geq 0})$ is width bounded with respect to $(\D^{\leq 0}, \D^{\geq 0})$ then $(\D'^{\leq 0}, \D'^{\geq 0})$ is a bounded t-structure iff $(\D^{\leq 0}, \D^{\geq 0})$ is. Hence $\sim$ restricts to an equivalence relation on the set of bounded t-structures.

Observe that if $\A$ and $\B$ are the respective hearts  of two bounded t-structures on $\D$, the t-structure $(\D_\B^{\leq 0},\D_\B^{\geq 0})$ is width bounded with respect to the t-structure $(\D_\A^{\leq 0},\D_\A^{\geq 0})$ iff $\B\subset \D_\A^{[m,n]}$ for some $m\leq n$. Indeed, if $\D_\A^{\leq m}\subset \D_\B^{\leq 0}\subset \D_\A^{\leq n}$ then $\B\subset \D_\B^{\leq 0}\subset \D_\A^{\leq n}, \B\subset \D_\B^{\geq 0}\subset \D_\A^{\geq m}$ and so $\B\subset \D_\A^{[m,n]}$;  
	conversely, if $\B\subset \D_\A^{[m,n]}$ then $\D_\B^{\leq 0}\subset \D_\A^{\leq n}, \D_\B^{\geq 0}\subset \D_\A^{\geq m}$ since $\D_\B^{\leq 0}$ (resp. $\D_\B^{\geq 0}$) is the smallest subcategory of $\D$ containing $\B$ and closed under extension and suspension (resp. desuspension).

\begin{exm}\label{fin simple width bounded}
	\begin{enum2}
	\item If $\D$ admits a bounded t-structure with  length heart containing finitely many (isoclasses of) simple objects, for example, $\D=\D^b(\Lambda)$ for a finite dimensional algebra $\Lambda$ over a field $k$,  then  bounded t-structures on $\D$ are width bounded with respect to each other. By Lemma~\ref{width equiv relation}, it suffices to show that a bounded t-structure with length heart $\C$ containing finitely many simple objects is width-bounded with respect to  any given bounded t-structure $(\D'^{\leq 0}, \D'^{\geq 0})$ on $\D$. Let $\{S_i \mid 1\leq i\leq t\}$ be a complete set of  simple objects in $\CC$. Then  
  $S_i\in \D'^{[k_i,l_i]}$ for each $i$ and some $k_i,l_i\in \Z$. Take $k=\min\{k_i,l_i\mid 1\leq i\leq t\}, l=\max\{k_i,l_i\mid 1\leq i\leq t\}$. $\CC\subset \D'^{[k,l]}$ shows our assertion.

\item Let $X$ be a  smooth projective variety over a field $k$  and  $\D^b(X)$  the bounded derived category of coherent sheaves over $X$. Then  bounded t-structures on $\D^b(X)$ are width bounded with respect to each other. It sufficies to show that the standard t-structure $(\D^{\leq 0}_{\std}, \D^{\geq 0}_{\std})$  is width bounded with respect to any given bounded t-structure $\tstr$ on $\D^b(X)$.
Let $\iota: X\ra \P_k^n$ be a closed immersion, where $\P_k^n$ is the $n$-dimensional projective space over $k$, and let $\O_X(i)=\iota^*\O(i)$. It follows from Beilinson's theorem (see e.g. \cite[Theorem 3.1.4]{OSS}) that for each $j<-n$, 
we have an exact sequence
\[0\ra \O_X(j)\ra V_n\otimes\O_X(-n)\ra \dots \ra V_0\otimes \O_X\ra 0,\] where $V_i=H^n(\P_k^n,\Omega_{\P^n_k}^i(i+j))$ ($\Omega_{\P^n_k}^i$ is the $i$-th wedge product of the cotangent bundle $\Omega_{\P^n_k}$). 
Since $\oplus_{i=0}^n\O_X(-i)$ lies  in some $\D^{\leq l}$, $\O_X(j)$ lies in $\D^{\leq l+n}$ for any $j\leq 0$. Now that $\D^{\leq 0}_{\std}$ is the smallest aisle containing $\{\O_X(j)\mid j\leq 0\}$, we have $\D_{\std}^{\leq 0}\subset \D^{\leq l+n}$.
On the other hand, applying the duality functor $\DD=R\chom(-,\O_X)$,  
we obtain a bounded t-structure $(\DD(\D^{\geq 0\, \op}), \DD(\D^{\leq 0\, \op}))$  on $\D^b(X)$. By the discussion above, $\DD(\D^{\leq 0\, \op})\subset \D_{\std}^{\geq m}$ for some $m$. Since  
$\O_X$ admits a finite injective resolution of quasi-coherent sheaves,  we have $(\DD\D_{\std}^{\geq m})^{\op}\subset \D_{\std}^{\leq r}$ for some $r$. So $\D^{\leq 0}\subset \D_{\std}^{\leq r}$.  $\D^{\leq -r}\subset \D^{\leq 0}_{\std}\subset \D^{\leq l+n}$ shows our assertion. 
\end{enum2}
\end{exm}

Given a bounded t-structure $(\D^{\leq 0}, \D^{\geq 0})$ on $\D$ with heart $\A$, \cite{HRS} gives a useful  and important construction of a class of width-bounded t-structures  with respect to $(\D^{\leq 0}, \D^{\geq 0})$ from torsion pairs in $\A$, which is called \emph{HRS-tilt.} 
Now it is well-known (see e.g. \cite[\S 1.1]{Poli}) that
\begin{prop}\label{torsion pair t-structure}\label{HRS-tilt}
	Torsion pairs in the heart of a t-structure $(\D^{\leq 0},\D^{\geq 0})$ are in bijective correspondence with t-structures $({\D'}^{\leq 0},{\D'}^{\geq 0})$ on $\D$ satisfying $\D^{\leq -1}\subset \D'^{\leq 0}\subset \D^{\leq 0}$. 
\end{prop}
Let us explain  the correspondence.
	Assume that $(\D'^{\leq 0},\D'^{\geq 0})$ is a $t$-structure with heart $\B$ such that $\D^{\leq -1}\subset \D'^{\leq 0}\subset \D^{\leq 0}$. Then $(\A\cap \B,\A\cap \B[-1])$ and $(\A[1]\cap \B,\A\cap \B)$ are  torsion pairs in $\A$ and $\B$, respectively. 
	Conversely, let $(\T,\F)$ be a torsion pair in the abelian category $\A$. Denote \[\D'^{\leq 0}=\D^{\leq -1}*\T,\quad \D'^{\geq 0}=\F[1]*\D^{\geq 0}.\]
	Then $(\D'^{\leq 0},\D'^{\geq 0})$ is a  $t$-structure on $\D$ with $\D^{\leq -1}\subset \D'^{\leq 0}\subset \D^{\leq 0}$ 
	and  $(\F[1],\T)$ is a torsion pair in its heart $\B$. In particular, $\B=\F[1]*\T$. 
		The t-structure $(\D^{\leq -1}*\T, \F[1]*\D^{\geq 0})$ is so-called \emph{HRS-tilt} with respect to the torsion pair $(\T,\F)$ in $\A$ and $\B=\F[1]*\T$ is called the \emph{tilted heart}.  

	As noted before, such a t-structure $({\D'}^{\leq 0},{\D'}^{\geq 0})$ is bounded iff  $(\D^{\leq 0},\D^{\geq 0})$ is. Moreover, if $(\D^{\leq 0},\D^{\geq 0})$ is bounded then $\D^{\leq -1}\subset {\D'}^{\leq 0}\subset {\D}^{\leq 0}$ iff $\B\subset \A[1]*\A$.

\subsection{Recollement, admissible subcategory, exceptional sequence}\label{sec: admissible subcategory}
A \emph{recollement} of triangulated categories \cite[\S 1.4]{BBD} is a diagram
\begin{equation} \xymatrix{\XX  \ar[rr]|{i_*}   & &\ar@/_1pc/[ll]|{i^*} \ar@/^1pc/[ll]|{i^!}\D \ar[rr]|{j^*} & &\ar@/_1pc/[ll]|{j_!} \ar@/^1pc/[ll]|{j_*} \YY}
\end{equation}\label{eq:recollement}
of three triangulated categories $\D, \XX, \YY$ and six exact functors $i^*, i_*, i^!, j_!, j^*, j_*$ between them such that
\begin{itemize}
\item $(i^*,i_*,i^!), (j_!,j^*,j_*)$ are adjoint triples;
\item  $i_*,j_!,j_*$ are fully faithful;
\item $\ker\, j^*=\im\, i_*$.  
	\end{itemize}
Given such a recollement,  
there are  two functorial triangles in $\D$: 
	\begin{equation}\label{recollement triangle}
		j_!j^*\ra \id\ra  i_*i^*\cra,\quad i_*i^!\ra \id\ra j_*j^*\cra,
	\end{equation}
	where the natural transformations between these functors are given by the respective unit or counit of the relevant adjoint pair.

	A well-known equivalent notion is so-called admissible subcategories, due to \cite{Bon}. 
Let us recall some classical results from {\cite{Bon}.
For a triangulated category $\D$, a strictly full triangulated subcategory $\CC$ is called \emph{right} (resp. \emph{left}) \emph{admissible} if the inclusion functor $\CC\monic \D$ admits a right (resp. left) adjoint; $\CC$ is called \emph{admissible} if it is both left and right admissible. If $\CC$ is right admissible then ${}^{\perp}(\CC^{\perp})=\CC$ and the inclusion functor $\CC^{\perp}\monic \D$ admits a left adjoint. In particular, $\CC$ is closed under direct summand and thus is a thick subcategory of $\D$. Moreover, the projection $\CC^{\perp}\ra \D/\CC$ is an exact equivalence. One has dual results for left admissible subcategories.   Hence if $\CC$ is admissible then we have \[{}^\perp \CC\overset{\simeq}{\lra}\D/\CC\overset{\simeq}{\lla}\CC^\perp\] and
we can form  (equivalent) recollements
\begin{equation}\label{eq:three equiv recollement}
		\begin{split}
			\xymatrix{\CC  \ar[rr]|{i_*}   & &\ar@/_1pc/[ll] \ar@/^1pc/[ll]\D \ar[rr] & &\ar@/_1pc/[ll]|{j_!} \ar@/^1pc/[ll] {}^\perp\CC,  \\
			\CC  \ar[rr]|{i_*}   & &\ar@/_1pc/[ll] \ar@/^1pc/[ll]\D \ar[rr]|{\hat{j}^{*}} & &\ar@/_1pc/[ll] \ar@/^1pc/[ll]\D/\CC,\\
			\CC  \ar[rr]|{i_*}   & &\ar@/_1pc/[ll] \ar@/^1pc/[ll]\D \ar[rr] & &\ar@/_1pc/[ll] \ar@/^1pc/[ll]|{\check{j}_*} \CC^\perp,} 
		\end{split}
	\end{equation}
where $i_*, j_!, \check{j}_*$ are the inclusion functors and $\hat{j}^*$ is the Verdier quotient functor.

We will need the following well-known fact. Recall that a Serre functor of a triangulated category is always exact (\cite[Proposition 3.3]{BK}; see also \cite[Proposition I.1.8]{RVDB}). 
\begin{prop}\label{admissible Serre functor}
Let $\D$ be a Hom-finite $k$-linear triangulated category with a Serre functor $\SS$, where $k$ is a field, and $\CC$ an admissible subcategory of $\D$.	Denote by $i_*: \CC\ra \D$  the inclusion functor and by $i^!: \D\ra \CC$ (resp. $i^*: \D\ra \CC$) the right (resp. left) adjoint of $i_*$. Then 
\begin{enum2}
\item $i^!\SS i_*$ is a Serre functor of $\CC$ with a quasi-inverse $i^*\SS^{-1}i_*$;

\item ${}^\perp \CC$ and $\CC^\perp$ admit Serre functors;  

\item  $\CC^{\perp}$ and ${}^{\perp}\CC$ are admissible subcategories of $\D$.
\end{enum2}
\end{prop}

\begin{proof}
	(1) One easily sees that $i^!\SS i_*$ (resp. $i^*\SS^{-1}i_*$) is a right (resp. left) Serre functor of $\C$.  
	Thus $i^!\SS i_*$ is a Serre functor of $\C$ with a quasi-inverse $i^*\SS^{-1}i_*$.

	(2) This is \cite[Proposition 3.7]{BK}.  
	
	(3) Recall the well-known fact that if $\D_1,\D_2$ are two Hom-finite $k$-linear triangulated categories with Serre functors $\SS_1,\SS_2$ respectively and $F:\D_1\ra \D_2$ is an exact functor with a left (resp. right) adjoint $G$ then $F$ admits a right (resp. left) adjoint $\SS_1\circ G\circ \SS_2^{-1}$ (resp. $\SS_1^{-1}\circ G\circ \SS_2$). Thus (3) follows from (2). 

\end{proof}

Important examples of admissible subcategories are those generated by an exceptional sequence \cite{Bon}. 
Recall that a sequence $(E_1,\dots, E_n)$ of objects in a $k$-linear triangulated category $\D$ of finite type,  where $k$ is a field, is called an \emph{exceptional sequence} if 
\begin{itemize}
	\item each $E_i$ is an exceptional object, i.e., $\hom^{\neq 0}(E_i,E_i)=0$ and $\End(E_i)=k$;
\item $\hom^\bullet(E_j,E_i)=0$ if $j>i$. 
\end{itemize}
An exceptional sequence $(E_1,\dots, E_n)$ is said to be \emph{full} if $E_1,\dots, E_n$ classically generate $\D$. 

Let $\C=\pair{E_1,\dots, E_n}_\D$ be the thick closure of  $\{E_i\mid 1\leq i\leq n\}$ and $i_*: \C\ra \D$ be the inclusion functor. The left and right adjoint functors of $i_*$ exist, which we denote by $i^*$, $i^!$ respectively.   
Let us recall from \cite{Bon} how  $i^*$ maps an object. Suppose $X\in \D$. Denote $X_0=X$.
If $X_i$ is defined for $0\leq i<n$, let \[X_{i+1}=\cocone(X_i\overset{\text{co-ev}}{\lra}D\hom^\bullet(X_i,E_{i+1})\otimes E_{i+1}).\] Then $X_{i+1}\in {}^\perp\{E_1,\dots, E_{i+1}\}$. Define $i^*X=X_n$. 
We have $i^*X\in {}^\perp \C$ and $i^*X$ fits into a triangle $i^*X\ra X\ra Y\cra $ where $Y\in \C$. This choice of $i^*$ on objects actually defines a unique functor up to unique isomorphism, which is left ajoint to $i_*$. Dually one defines $i^!$.

\subsection{Gluing t-structures}\label{sec: glue t-str}
Now fix a  recollement of triangulated categories of the form (2.3.1). 
As the following theorem shows, one can obtain a t-structure on $\D$ from t-structures on $\XX$ and $\YY$, which is called a \emph{glued t-structure}. Such a glued t-structure  on $\D$ from the recollement  is also said to be compatible with the recollement. 

\begin{thm}[{\cite[Th\'eor\`eme 1.4.10]{BBD}}]\label{compatible t-str}
Given t-structures $(\XX^{\leq 0},\XX^{\geq 0})$ and $(\YY^{\leq 0},\YY^{\geq 0})$ on $\XX$ and $\YY$ respectively, denote 
	\begin{equation}\label{eq: glued t-str}
		\begin{split}
		\D^{\leq 0}=\{X\in \D\mid i^*X\in \XX^{\leq 0}, j^*X\in \YY^{\leq 0}\},\\
		\D^{\geq 0}=\{X\in \D\mid i^!X\in \XX^{\geq 0}, j^*X\in \YY^{\geq 0}\}.
		\end{split}
	\end{equation}
	Then $(\D^{\leq 0},\D^{\geq 0})$ is a t-structure on $\D$.
\end{thm}
With the given t-structures on $\XX, \YY$ and the glued t-structure on $\D$, $i^*, j_!$ becomes right t-exact, $i_*,j^*$ t-exact and $i^!, j_*$ left t-exact.

The following proposition answers the natural question when a t-structure on $\D$ is compatible with a given recollement.

\begin{prop}[{\cite[Proposition 1.4.12]{BBD}}]\label{criterion compatible t-str}
	Given a t-structure $(\D^{\leq 0},\D^{\geq 0})$ on $\D$, the following conditions are equivalent:
	\begin{enum2} 
	\item $j_!j^*$ is right t-exact;

	\item $j_*j^*$ is left t-exact;

	\item the t-structure is compatible with the recollement (2.3.1).
	\end{enum2}

\end{prop}
Moreover, we have
\begin{lem}[{\cite[Corollary 3.4, Lemma 3.5]{Liu-V}}]\label{t-str recollement bijection}
	There is a bijection
	\begin{multline}
		\{\text{t-structures on $\XX$}\}\times \{\text{t-structures on $\YY$}\}\longleftrightarrow\\
		\{\text{t-structures on $\D$ compatible with the recollement (2.3.1)}\},
	\end{multline}
	which restricts to a bijection between bounded t-structures.
\end{lem}
	Indeed, once the equivalent conditions in Proposition~\ref{criterion compatible t-str} are satisfied, to obtain  $\tstr$  using formula ~\eqref{eq: glued t-str},  the unique choice of the t-structure on $\XX$ resp. $\YY$ is 
\begin{equation}
	(i^*\D^{\leq 0}, i^!\D^{\geq 0})\,\,\, \text{resp.}\,\,\, (j^*\D^{\leq 0},j^*\D^{\geq 0}).
\end{equation}
This t-structure on $\XX$ resp. $\YY$ will be called \emph{the corresponding t-structure} on $\XX$ resp. $\YY$ to the t-structure $\tstr$ on $\D$. Moreover we have 
\begin{equation}
	(i_*i^*\D^{\leq 0},i_*i^!\D^{\geq 0})=(\im\, i_*\cap \D^{\leq 0},\im\, i_*\cap \D^{\geq 0}). 
\end{equation}
Since we can  identify $\XX$ with  $\im\, i_*$ via $i_*$, we know that  the t-structure on $\XX$ is essentially induced by restriction.

Suppose $\CC$ is an admissible subcategory of $\D$ and $\tstr$ is a t-structure on $\D$. Let 
\begin{equation}\label{no matter}
	\xymatrix{\CC  \ar[rr]|{i_*}   & &\ar@/_1pc/[ll]|{i^*} \ar@/^1pc/[ll]|{i^!}\D \ar[rr]|{j^*} & &\ar@/_1pc/[ll]|{j_!} \ar@/^1pc/[ll]|{j_*} \CC'}
\end{equation}
be a recollement, where $i_*$ is the inclusion functor. Since $j_!j^*X=\cocone(X\ra i_*i^*X)$ for each $X\in \D$ by ~\eqref{recollement triangle},  $j_!j^*$ is right t-exact iff $\cocone(X\ra i_*i^*X)$ lies in $\D^{\leq 0}$ for each $X\in \D^{\leq 0}$. So given another recollement \begin{equation}\label{no matter 2}
	\xymatrix{\CC  \ar[rr]|{i_*}   & &\ar@/_1pc/[ll]|{i^*} \ar@/^1pc/[ll]|{i^!}\D \ar[rr]|{k^*} & &\ar@/_1pc/[ll]|{k_!} \ar@/^1pc/[ll]|{k_*} \CC'',}
\end{equation}
$(\D^{\leq 0}, \D^{\geq 0})$ is compatible with the recollement ~\eqref{no matter} iff it is compatible with the (equivalent) recollement ~\eqref{no matter 2}. 
Thus it makes sense to say that $(\D^{\leq 0}, \D^{\geq 0})$ is compatible with $\CC$ if $(\D^{\leq 0},\D^{\geq 0})$ is compatible with any recollement of the form ~\eqref{no matter}, for example, any one of the recollements ~\eqref{eq:three equiv recollement}. This is convenient for use. 
 If $(\D^{\leq 0}, \D^{\geq 0})$ is compatible with the admissible subcategory $\CC$ then $(\D^{\leq 0}\cap \CC, \D^{\geq 0}\cap \CC)$ is a t-structure on $\CC$. 
 In general, consider a finite admissible filtration (\cite[Definition 4.1]{BK}) \[\D_n\subset \D_{n-1}\subset \dots \subset \D_0=\D\] of a triangulated category $\D$. That is, each $\D_i$ ($1\leq i\leq n$) is an admissible subcategory of $\D_{i-1}$, equivalently, each $\D_i$ is an admissible subcategory of $\D$. We say the t-structure $(\D^{\leq 0},\D^{\geq 0})$ is compatible with the admissible filtration if it is compatible with each $\D_i$. 
 
 Clearly we have the following two facts.

\begin{lem}\label{compatible filtration}
	$(\D^{\leq 0},\D^{\geq 0})$ is compatible with the admissible filtration \[\D_n\subset \dots \subset \D_1\subset \D_0=\D\] of  $\D$ iff the t-structure $(\D^{\leq 0}\cap \D_i,\D^{\geq 0}\cap \D_i)$ on $\D_i$ is compatible with $\D_{i+1}$ for each $1\leq i\leq n-1$.
\end{lem}
Here by the statement that the t-structure $(\D^{\leq 0}\cap \D_i,\D^{\geq 0}\cap \D_i)$ on $\D_i$ is compatible with $\D_{i+1}$ for each $1\leq i\leq n-1$, we  actually mean that: $(\D^{\leq 0},\D^{\geq 0})$ is compatible with $\D_1$ (hence $(\D^{\leq 0}\cap \D_1, \D^{\geq 0}\cap \D_1)$ is a t-structure on $\D_1$); $(\D^{\leq 0}\cap \D_1, \D^{\geq 0}\cap \D_1)$ is compatible with $\D_2$ (hence $(\D^{\leq 0}\cap \D_2, \D^{\geq 0}\cap \D_2)$ is a t-structure on $\D_2$); and so on. This situation arises naturally from reduction/induction argument.

\begin{lem}\label{compatible filtration autoequi}
	Suppose that the t-structure $(\D^{\leq 0},\D^{\geq 0})$ is compatible with the admissible filtration \[\D_n\subset \D_{n-1}\subset \dots\subset \D_0=\D\] and let $\Phi$ be an exact autoequivalence of $\D$. Then the t-structure $(\Phi(\D^{\leq 0}), \Phi(\D^{\geq 0}))$ is compatible with the admissible filtration
	\[\Phi(\D_n)\subset \Phi(\D_{n-1})\subset \dots \subset \Phi(\D_0)=\D.\]
\end{lem}

\subsection{On the hearts of the t-structures in a recollement context}\label{sec: glue heart}
 Fix a recollement of the form (2.3.1). 
	 Each t-structure   $(\XX^{\leq 0},\XX^{\geq 0})$ on $\XX$  
	 induces  (up to shift) two t-structures on $\D$ in the following fashion. For each $p\in \Z$, since the inclusion $i_*\XX^{\leq p}\monic \D$ admits a right adjoint $i_*\tau_{\leq p}i^!$, $i_*\XX^{\leq p}$ is an aisle in $\D$ and 
	 \[(i_*\XX^{\leq p},\,\, (i_*\XX^{\leq p})^{\perp_{0,\D}}[1])\] is a t-structure on $\D$.
	 Denote by $\check{\tau}_{\geq p+1}$ the left adjoint of the inclusion  $(i_*\XX^{\leq p})^{\perp_{0,\D}}\monic \D$. Then  we have a functorial triangle \[i_*\tau_{\leq  p}i^!\ra \id\ra \check{\tau}_{\geq p+1}\cra \] for each $p\in\Z$.  
		 Dually,  the inclusion $i_*\XX^{\geq p}\monic \D$ admits a left adjoint $i_*\tau_{\geq p}i^*$, and we have a t-structure \[(({}^{\perp_{0,\D}}i_*\XX^{\geq p})[-1],\,\, i_*\XX^{\geq p})\] and  a functorial triangle \[\hat{\tau}_{\leq p-1}\ra \id\ra i_*\tau_{\geq p}i^*\cra \] for each $p\in\Z$, where $\hat{\tau}_{\leq p-1}$ is the right adjoint of the inclusion $({}^{\perp_{0,\D}}i_*\XX^{\geq p})\monic \D$. 
 A similar argument shows that a t-structure on $\YY$ also induces two t-structures on $\D$. 
 \begin{rmk}
	 In \cite[\S 1.4.13]{BBD}, these induced t-structures are described via gluing.
 \end{rmk}

Suppose $(\XX^{\leq 0}, \XX^{\geq 0})$, $(\YY^{\leq 0}, \YY^{\geq 0})$ are t-structures on $\XX, \YY$ respectively and let $(\D^{\leq 0}, \D^{\geq 0})$ be the glued t-structure. Denote the respective heart by $\B_1, \B_2$ and $\B$.  
Let $\epsilon$ be the inclusion functor from $\B_1, \B_2$ resp. $\B$ to $\XX, \YY$ resp. $\D$.
 For $T\in \{i^*,i_*,i^!, j_!,j^*,j_*\}$, denote ${}^p T=H^0\circ T\circ \epsilon.$  
 Then $({}^pi^*, {}^pi_*, {}^pi^!)$ and $({}^pj_!, {}^pj^*, {}^pj_*)$ are adjoint triples, the compositions ${}^pj^*\circ {}^pi_*, {}^pi^*\circ {}^pj_!, {}^pi^!\circ {}^pj_* $ vanish,   and ${}^pi_*, {}^pj_!, {}^pj_*$ are fully faithful. 
 $\im\, {}^pi_*=\ker\, {}^pj^*$ is a Serre subcategory of $\B$,  the functor ${}^pi_*$ identifies  $\B_1$  with $\im\, {}^pi_*$ and the functor ${}^pj^*$ identifies the quotient category $\B/\im\, {}^pi_*$ with $\B_2$. 
 The composition ${}^pj_!{}^pj^*\ra \id\ra {}^pj_*{}^pj^*$ provides a unique morphism of functors ${}^pj_!\ra {}^pj_*$. Define 
 \begin{equation}
	 j_{!*}=\im\, ({}^pj_!(-)\ra {}^p j_*(-)) : \B_2\lra \B.
 \end{equation}

 The following proposition describes simple objects in $\B$.

		 \begin{prop}[{\cite[Proposition 1.4.23, 1.4.26]{BBD}}]\label{simple in heart}
\begin{enum2}		 
\item For $X\in \B_2$, we have \[j_{!*} X=\check{\tau}_{\geq 1}j_!X=\hat{\tau}_{\leq -1}j_*X.\]

\item Simple objects in $\B$ are those ${}^pi_*S$, for $S$ simple in $\B_1$, and those $j_{!*}S$, for $S$ simple in $\B_2$.
\end{enum2}
\end{prop}
 For more details, see \cite[\S 1.4]{BBD}, from which the above are taken.
The following lemma strenghens \cite[Proposition 3.9]{Liu-V}.
 \begin{lem} \label{no-ar-le}
	 $\B$ is noetherian (or artinian, or  of finite length) iff so are $\B_1, \B_2$.
 \end{lem}
\begin{proof}
	\cite[Lemma 1.3.3]{CK} states that if $\A_1$ is a Serre subcategory of an abelian category $\A$ then $\A$ is noetherian iff $\A_1$ and $\A/\A_1$ are noetherian and if each object in $\A$ has a largest subobject that belongs to $\A_1$. We claim that in our setting, 
	each $B\in \B$ admits a largest subobject ${}^pi_*{}^pi^!B$ in ${}^pi_*\B_1$. By \cite[Lemme 1.4.19]{BBD}, we have an exact sequence
	\[0\ra  {}^pi_*{}^pi^!B\overset{\eta}{\ra} B\ra {}^pj_*{}^pj^*B\ra {}^pi_*H^1i^!B \ra 0.\]
	Suppose $\mu: {}^pi_*Z\ra B$ is a monomorphism in $\B$, where $Z\in \B_1$. Note that  \[\hom({}^pi_*Z, {}^pj_*{}^pj^*B)=\hom(Z, {}^pi^!{}^pj_*{}^pj^*B)=0.\]
 So there exists $\nu: {}^pi_*Z\ra  {}^pi_*{}^pi^!B$ such that $\mu=\eta\nu$. Since $\mu$ is a monomorphism, $\nu$ is a monomorphism. So ${}^pi_*Z$ is a subobject of ${}^pi_*{}^pi^!B$. This shows our claim that ${}^pi_*{}^pi^!B$ is the largest subobject of $B$ in ${}^pi_*\B_1$. Hence the assertion on noetherianness follows. 
By duality, we conclude the assertion on artinianness. 
	Combining these two assertions, we know that $\B$ is of finite length iff $\B_1, \B_2$ are of finite length.
\end{proof}

An easy induction argument yields 
\begin{cor}\label{filt no-ar-le}
	Suppose a t-structure $(\D^{\leq 0},\D^{\geq 0})$ on $\D$ is compatible with  the admissible filtration \[0=\D_{n+1}\subset \D_n\subset \dots \subset \D_1\subset \D_0=\D.\] Then $(\D^{\leq 0}, \D^{\geq 0})$ has noetherian resp. artinian resp. length heart iff the corresponding t-structure on each $\D_{i+1}^{\perp_{\D_i}}$ $($or ${}^{\perp_{\D_i}}\D_{i+1},$ or $\D_i/\D_{i+1})$ $(0\leq i\leq n)$ has noetherian resp. artinian resp. length heart.
\end{cor}

\subsection{Recollement and Ext-projectives}\label{sec: Ext-proj}
Let $\D$ be  a $k$-linear triangulated category of finite type, where $k$ is a field, and $(\D^{\leq 0},\D^{\geq 0})$ a t-structure on $\D$. Recall from \cite[\S 1]{ASS}  that $X\in \D$ is Ext-projective in $\D^{\leq l}$, or $\D^{\leq l}$-projective for short, if  $X\in \D^{\leq l}$ and $\hom^1(X, \D^{\leq l})=0$; dually, $X\in \D$  is Ext-injective in  $\D^{\geq l}$, or $\D^{\geq l}$-injective, if  $X\in \D^{\geq l}$ and $\hom^1(\D^{\geq l},X)=0$.  

We  use the following criterion to identify Ext-projectives (and Ext-injectives) when $\D$ admits a Serre functor.
\begin{lem}[{\cite[Lemma 1.5]{ASS}}]\label{ext-proj}
	Suppose $\D$ admits a Serre functor $\SS$ and $X$ is an object in $\D$. Then $X$ is $\D^{\leq 0}$-projective iff $X\in \D^{\leq 0}$ with $\SS X\in \D^{\geq 0}$ iff $\SS X$ is $\D^{\geq 0}$-injective. 
\end{lem}

The following easy observation is essential for us. 
\begin{lem}\label{ext-proj recollement}
	 Suppose $E\in \D$ is an exceptional object. 
	If $E$ is Ext-projective in some $\D^{\leq l}$ and $E^{\perp_\D}$ is right admissible  then $(\D^{\leq 0},\D^{\geq 0})$ is compatible with the recollement
	\[ \xymatrix{E^{\perp_\D} \ar[rr]|{i_*}   & &\ar@/_1pc/[ll] \ar@/^1pc/[ll]\D\ar[rr]|{j^*} & &\ar@/_1pc/[ll]|{j_!} \ar@/^1pc/[ll] \pair{E}_\D,}\]
where $i_*, j_!$ are the inclusion functors.

\end{lem}

\begin{proof}
	Since $E$ is an exceptional object, $\pair{E}_\D$ is admissible and thus $E^{\perp_\D}$ is left admissible with ${}^{\perp_\D}(E^{\perp_\D})=\pair{E}_\D$. If $E^{\perp_\D}$ is right admissible then $E^{\perp_\D}$ is admissible and the given diagram is indeed a diagram of recollement. To show that the t-structure is compatible, it suffices to show that $j_!j^*$ is right t-exact, i.e., for each $X\in \D^{\leq 0}$, $j_!j^*(X)\in \D^{\leq 0}$. Note that for $m>-l$, $\hom(E,\D^{\leq 0}[m])=0$ since $E$ is $\D^{\leq l}$-projective. Therefore
	\[
		\begin{aligned}
			j_!j^*(X) & = \hom^\bullet(E,X)\otimes E\\
										   & = \oplus \hom(E,X[m])\otimes E[-m]\\
										   & =\oplus_{m\leq -l} \hom(E,X[m])\otimes E[-m]\\
								  & \in \D^{\leq 0}.
		\end{aligned}
	\]
\end{proof}

\begin{rmk} \begin{enum2}\item There is a dual version for Ext-injectives.

	\item In our application, $\D$ has a Serre functor and thus  $E^{\perp_\D}$ and ${}^{\perp_\D}E$ are indeed admissible by Proposition~\ref{admissible Serre functor}.  
	\end{enum2}

\end{rmk}

Assume that $\D$ has a Serre functor and $(E_n, \dots, E_1)$ is  an exceptional sequence such that each $E_i$ is $\D^{\leq 0}$-projective. Let $\D_0=\D$; for $1\leq i\leq n$, let $\D_i=\{E_i, E_{i-1}, \dots, E_1\}^{\perp_\D}$. Note that $\D_i=E_i^{\perp_{\D_{i-1}}}$ for $1\leq i\leq n$. We already know that $\pair{E_i, E_{i-1}, \dots, E_1}_\D$ is admissible in $\D$ and thus $\D_i$ is admissible in $\D$ by Proposition~\ref{admissible Serre functor}. The following fact is immediate from Lemma~\ref{ext-proj recollement} and Lemma~\ref{compatible filtration}. (We also have a similar result when each $E_i$ is $\D^{\geq 0}$-injective.)

\begin{cor}\label{ex collection compatible filt}
	With the above hypotheses and notation, $(\D^{\leq 0}, \D^{\geq 0})$ is compatible with the admissible filtration \[\D_n\subset \dots \subset \D_i(=\{E_i,\dots, E_1\}^{\perp_\D}=E_i^{\perp_{\D_{i-1}}})\subset \dots \subset \D_1\subset \D.\]
\end{cor}

	Now let us be given a recollement of the form (2.3.1).
	Suppose that $\XX$ resp. $\YY$ is equipped with a t-structure $(\XX^{\leq 0},\XX^{\geq 0})$ resp. $(\YY^{\leq 0}, \YY^{\geq 0})$, and $\D$ with the glued t-structure $(\D^{\leq 0},\D^{\geq 0})$. One easily verifies the following fact.
	\begin{lem}\label{ext-proj from recollement}
		\begin{enum2}\item  If $X$ is $\D^{\leq 0}$-projective which does not lie in $\ker\, i^*=\im\, j_!$ then $i^*X$ is nonzero $\XX^{\leq 0}$-projective.

		\item If $Y$ is nonzero $\YY^{\leq 0}$-projective then $j_!Y$ is nonzero $\D^{\leq 0}$-projective. Moreover, $j_!$ induces a bijection between isoclasses of indecomposable Ext-projectives in $\YY^{\leq 0}$ and  isoclasses of indecomposable Ext-projectives in $\D^{\leq 0}$ which lie in $\ker\, i^*=\im\, j_!$.
		\end{enum2}
\end{lem}

\subsection{Some facts on hereditary categories}\label{sec: hereditary}

Let $\A$ be a hereditary category linear over an algebraically closed field  $k$ with finite-dimensional morphism and extension spaces. It's well-known that each object $X\in \D^b(\A)$ decomposes as $X\cong \oplus_i H^i(X)[-i]$. In particular, each indecomposable object in $\D^b(\A)$ is a shift of an indecomposable object in $\A$.

 The following  Happel-Ringel Lemma (see e.g. \cite[Proposition 5.1]{Lenzing2}) is fundamental  for hereditary categories. 
\begin{prop}[Happel-Ringel Lemma]\label{Happel-Ringel lemma}	
 Let $E$ and $F$ be indecomposable objects of $\A$ such that $\ext^1(F, E)=0$. Then each nonzero morphism $f : E \ra F$ is a monomorphism or an epimorphism.  In particular, each indecomposable object in $\A$ without self-extension is exceptional.
\end{prop}

Recall that an  object $T$ in a triangulated category  is a  \emph{partial silting object} if $\hom^{>0}(T,T)=0$ and $T$ is \emph{basic}  if its indecomposable direct summands are pairwise non-isomorphic.
The following fact shows that a basic partial silting object in $\D^b(\A)$ can yield an exceptional sequence. Note that  $\D^b(\A)$ is a Krull-Schmidt category since $\A$ is Hom-finite.
 
\begin{prop}[{\cite[Proposition 3.11]{AI}}]\label{hereditary ext vanish}
 Let $X$ be a basic partial silting object in $\D^b(\A)$.
	Then pairwise non-isomorphic indecomposable direct summands of $X$ can be ordered to form an exceptional sequence. 
\end{prop}

Although it is stated for specific hereditary categories in \cite[Proposition 3.11]{AI}, the above fact follows from Happel-Ringel Lemma. 

We will need to  relate Ext-projectives to an exceptional sequence.
\begin{prop}[{\cite[Theorem (A)]{ASS}}]\label{order silting}
	Let $(\D^{\leq 0},\D^{\geq 0})$ be a t-structure in $\D^b(\A)$. Then finitely many pairwise non-isomorphic indecomposable $\D^{\leq 0}$-projectives can be ordered to form an  exceptional sequence in $\D^b(\A)$.
\end{prop}

Proposition~\ref{order silting} follows from Proposition~\ref{hereditary ext vanish} since the direct sum of finitely many pairwise non-isomorphic indecomposable $\D^{\leq 0}$-projectives is a basic partial silting object.

	\subsection{Bounded t-structures on $\D^b(\Lambda)$ for a finite dimensional algebra $\Lambda$}\label{sec: algebra t-str}\label{sec: length heart}

	Recall from \cite{KV, AI} that an object $X$ in a triangulated category $\D$ is called \emph{silting} if it is partial silting, i.e., $\hom^{>0}(X,X)=0$, and if $\pair{X}_\D=\D$. It is \emph{tilting} if additionally $\hom^{< 0}(X,X)=0$. Two silting objects $X$ and $Y$ are said to be equivalent if $\add\, X=\add\, Y$. 
	
	Let $\Lambda$ be a finite dimensional algebra over a field $k$. 
	Denote by  $\K^b(\proj \Lambda)$ the bounded homotopy category of finite dimensional projective right modules over $\Lambda$. The following part of K\"onig-Yang correspondences will be used repeatedly in the sequel. See \cite{KD} for bijective correspondences between more concepts.

\begin{thm}[{\cite[Theorem 6.1]{KD}}]\label{silting t-str}
	Equivalence classes of silting objects in $\K^b(\proj \Lambda)$ are in bijective correspondence with bounded t-structures on $\D^b(\Lambda)$ with length heart. 	
\end{thm}
Let us recall this correspondence from \cite{KD}.
For a silting object $M$ in $\K^b(\proj \Lambda)$, the associated t-structure on $\D^b(\Lambda)$ is given by the pair \[\D^{\leq 0}=\{N\in \D^b(\Lambda)\mid \hom^{>0}(M,N)=0\},\] \[\D^{\geq 0}=\{N\in \D^b(\Lambda)\mid \hom^{<0}(M,N)=0\}.\] Moreover, the heart of $(\D^{\leq 0}, \D^{\geq 0})$ is equivalent to $\mod\, \End(M)$ (\cite[Lemma 5.3]{KD}).  We refer the reader to \cite[\S5.6]{KD} for the general construction (essentially due to Rickard~\cite{Rick}) of a silting object associated to a given bounded t-structure $(\D^{\leq 0}, \D^{\geq 0})$ on $\D^b(\Lambda)$ with length heart. When $\Lambda$ has finite global dimension, in which case the natural inclusion $\K^b(\proj \Lambda)\ra \D^b(\Lambda)$ is an exact equivalence, 
  the associated basic silting object in $\K^b(\proj \Lambda)=\D^b(\Lambda)$ is just the direct sum of a complete set of  indecomposable Ext-projectives in the aisle $\D^{\leq 0}$.

\begin{lem} [{\cite[Lemma 6.7]{Liu-V}}]\label{finite rep length heart}\label{finite rep t-str}
	If $\Lambda$ is a representation-finite hereditary algebra then each bounded t-structure on $\D^b(\Lambda)$ has length heart.
\end{lem}
Hence by Theorem~\ref{silting t-str}, to classify bounded t-structures on $\D^b(\Lambda)$, where $\Lambda$ is a representation-finite hereditary algebra,  it sufficies, say, to classify silting objects in $\D^b(\Lambda)$, which is indeed computable.   

The following fact characterizes when a silting object is a tilting object in the presence of a Serre functor. 
\begin{lem}[{\cite[Lemma 4.6]{LVY}}]\label{silting tilting}
Assume that $\Lambda$ has finite global dimension and $\SS$ is a Serre functor of $\D^b(\Lambda)$. 
Let $T$ be a silting object in $\D^b(\Lambda)$ and $\B$ the heart of the corresponding t-structure $(\D^{\leq 0}, \D^{\geq 0})$. Then  $T$ is tilting iff $\SS$ is right t-exact with respect to $(\D^{\leq 0}, \D^{\geq 0})$ iff $\SS T$ lies in $\B$.
\end{lem}

We will also need the next two facts. 
\begin{lem}\label{simple ext-proj A}
 Let $k\vec{\AA}_s$ be the path algebra of the equioriented $\AA_s$-quiver. 
 Suppose $(\D^{\leq 0}, \D^{\geq 0})$ is a bounded t-structure on $\D^b(k\vec{\AA}_{s})$. Then some simple $k\vec{\AA}_s$-module is Ext-projective in some $\D^{\leq l}$.
\end{lem}

\begin{proof}
	Denote $\A=\mod k\vec{\AA}_s, \D=\D^b(k\vec{\AA}_s)$ for short. It is well-known that $\A$ is a uniserial hereditary abelian category, each indecomposable object in $\A$ is exceptional, and $\D$ has a Serre functor (isomorphic to the Nakayama functor). 
	We use induction on $s$ to show our assertion. If $s=1$, we have $\mod k\vec{\AA}_1=\mod k$ and the assertion obviously holds. Suppose $s>1$. By  Lemma~\ref{finite rep length heart}, the heart $\B$ of $\tstr$ is of finite length. Take an indecomposable direct summand $N[p]$ ($N\in \A$) of  the corresponding  silting object. Then $N$ is $\D^{\leq p}$-projective.  
	If $N$ is a simple module then $N$ is the desired. Otherwise, let \[\A_1=\pair{\tau^m(\top (N))\mid 1\leq m< l(N)}_\A, \quad \bar{\A}_1=\pair{\tau^m(\top (N))\mid 0\leq m<l(N)}_\A,\] where $\tau=D\Tr$ represents the Auslander-Reiten translation and $l(N)$ is the length of $N$. 
	For a simple module $S$, denote by ${}^{[l]}S$ the unique indecomposable module with top $S$ and of length $l$.  Since $\oplus_{0\leq i<l(N)}{}^{[l(N)-i]}\tau^i \top(N)$ is a projective generator for $\bar{\A_1}$ with endomorphism algebra isomorphic to $k\vec{\AA}_{l(N)}$, we have $\bar{\A_1}\simeq\mod k\vec{\AA}_{l(N)}$. 	
	
	We know that $N^{\perp_\A}$ is an exact subcategory of $\A$ closed under extension. Take \[\A_2=\add\{M\in N^{\perp_\A}\mid \text{$M$ is indecomposable and $M\notin \A_1$}\}.\] We claim  $N^{\perp_\A}=\A_1\coprod \A_2$, which implies that $\A_2$ is an exact subcategory of $\A$ closed under extension. 
	Since $N^{\perp_\A}=\add\A_1\cup \A_2$, it sufficies to show that $\hom(\A_1, \A_2)=0=\hom(\A_2, \A_1)$. Note that 
\[\begin{aligned}
\A_1 & =\add\{{}^{[l]}\tau^i\top(N)\mid 1\leq i< l(N), 1\leq l\leq l(N)-i\},\\
		N^{\perp_\A} & =\{M\in N^{\perp_\A}\mid \hom(N, M)=0=\ext^1(N, M)\}\\
		& =\{M\in N^{\perp_\A} \mid \hom(N,M)=0=\hom(M, \tau N)\}.
\end{aligned}\]
Let $M$ be an indecomposable $k\vec{\AA}_s$-module.  Suppose $\hom({}^{[l]}\tau^i\top(N), M)\neq 0$ for some $1\leq i<l(N), 1\leq l\leq l(N)-i$. Then for some $1\leq k\leq l$, ${}^{[k]}\tau^i\top(N)$ is a subobject of $M$. If $M\notin \A_1$ then ${}^{[k+i]}\top(N)$ is a subobject of $M$. Meanwhile, ${}^{[k+i]}\top(N)$ is a quotient object of $N$ and thus $\hom(N,M)\neq 0$. This shows that if $\hom(N,M)=0$ then $\hom(\A_1, M)=0$. Simlarly, if  $\hom(M, {}^{[l]}\tau^i\top(N))\neq 0$ for some $1\leq i<l(N), 1\leq l\leq l(N)-i$, then $M$ has a nonzero quotient object which is moreover a subobject of $\tau N$; so $\hom(M,\A_1)=0$ if $\hom(M,\tau N)=0$. It follows that $\hom(\A_1, M)=0=\hom(M, \A_1)$ for an indecomposable module $M\in \A_2$. This shows our claim. 

 By Proposition~\ref{admissible Serre functor}, $N^{\perp_\D}$ is admissible in $\D$. Since $N$ is an exceptional Ext-projective object in $\D^{\leq p}$, by Lemma~\ref{ext-proj recollement}, $(\D^{\leq 0}, \D^{\geq 0})$ is compatible with the admissible subcategory $N^{\perp_\D}$ and $(\D^{\leq 0}\cap N^{\perp_\D}, \D^{\geq 0}\cap N^{\perp_\D})$ is a bounded t-structure on $N^{\perp_\D}$. 
 Obviously, this t-structure is compatible with the admissible subcategory $\D^b(\A_2)$ of $N^{\perp_\D}=\D^b(N^{\perp_\A})$. Hence by Lemma~\ref{compatible filtration}, $(\D^{\leq 0}, \D^{\geq 0})$ is compatible with the recollement \[ \xymatrix{ \D^b(\A_2)\ar[rr]|{i_*}   & &\ar@/_1pc/[ll]|{i^*} \ar@/^1pc/[ll]|{i^!}\D\ar[rr]|{j^*} & &\ar@/_1pc/[ll]|{j_!} \ar@/^1pc/[ll]|{j_*} {}^{\perp_\D}\D^b(\A_2),}\]
where $i_*, j_!$ are the inclusion functors. Note that 
\[
	\begin{aligned}
		{}^{\perp_\D}\D^b(\A_2) & = \pair{N, \tau^m(\top(N))\mid 1\leq m< l(N)}_\D\\
								& =\pair{\tau^m (\top(N))\mid 0\leq m< l(N)}_\D\\
					   & =\D^b(\bar{\A}_1)\\
					   & \simeq \D^b(k\vec{\AA}_{l(N)}).
	\end{aligned}
\]
Consider the bounded t-structure $(j^*\D^{\leq 0}, j^*\D^{\geq 0})$ on $\D^b(\bar{\A}_1)\simeq \D^b(k\vec{\AA}_{l(N)})$. By the induction hypothesis, some  $\tau^m(\top(N))\, (0 \leq m<l(N))$ is Ext-projective in some $j^*\D^{\leq l}$. Hence the simple module $\tau^m (\top (N))= j_! \tau^m(\top (N))$ is $\D^{\leq l}$-projective by Lemma~\ref{ext-proj from recollement}, as desired.
\end{proof}

\begin{cor}\label{A_s exseq}  Let $k\vec{\AA}_s$ be the path algebra of the equioriented $\AA_s$-quiver. Each silting object in $\D^b(k\vec{\AA}_s)$ contains a shift of some simple module as its direct summand. Each full exceptional sequence in $\mod k\vec{\AA}_s$ contains a simple module.
\end{cor}
\begin{proof}
	The first assertion follows from Lemma~\ref{simple ext-proj A}. For a full exceptional sequence $(E_1,\dots, E_n)$ in $\mod k\vec{\AA}_s$, it is observed in \cite[Proposition 3.5]{AI} that we can take suitable $l_i$ ($1\leq i\leq n$), say $l_i=i$ here, such that $\oplus_{i=1}^n E_i[l_i]$ is a silting object in $\D^b(k\vec{\AA}_s)$. So the second assertion follows.  
\end{proof}

\subsection{Bounded t-structures on $\D^b(\nilp k\tilde{\AA}_{t-1})$}\label{sec: A_t}

Let $k$ be a field. Denote by $\tilde{\AA}_{t-1}$ the quiver which is an oriented cycle with $t$ vertices and by $\A_t=\nilp k\tilde{\AA}_{t-1}$ the category of finite dimensional  nilpotent $k$-representations of $\tilde{\AA}_{t-1}$. Let us recall some standard  facts on $\A_t$.
$\A_t$ is a connected  hereditary uniserial length  abelian category and admits an autoequivalence $\tau$ of period $t$ such that $\tau(-)[1]$ is the Serre functor of $\D^b(\A_t)$. Moreover,
$\A_t$ has almost split sequences with Auslander-Reiten translation given by $[M]\dashrightarrow [\tau M]$, and its Auslander-Reiten quiver is a tube of rank $t$ (see \S\ref{sec: AR theory} if one is unfamiliar with Auslander-Reiten theory). 
If $S$ is a simple object in $\A_t$ then each simple object is of the form $\tau^i S$ for some $i\in\Z/t\Z$.  Denote by $S^{[n]}$ (resp. ${}^{[n]}S$) the unique (up to isomorphism) indecomposable object in $\A_t$ of length $n$  and with socle (resp. top) $S$. For an indecomposable object $X$ in $\A_t$, its length is denoted by $l(X)$, and its simple socle resp. top by $\soc(X)$ resp. $\top(X)$. Then  $X=(\soc (X))^{[l(X)]}={}^{[l(X)]}(\top (X))$. $X$ is exceptional iff $l(X)<t$.

Recall from \cite{HRS} that for a torsion pair $(\T,\F)$ in an abelian category $\A$, $\T$ is called a \emph{tilting torsion class}  if $\T$ is a cogenerator for $\A$, i.e, for each $A\in \A$, there is a monomorphism $A\monic T$  with $T\in \T$; dually, $\F$ is called a \emph{cotilting torsion-free class} if $\F$ is a generator for $\A$. 
\begin{lem}\label{A_t torsion pair}
	For a torsion pair $(\T,\F)$ in $\A_t$, exactly one of the following holds
	\begin{enum2}
	\item $\T$ is a tilting torsion class, equivalently, $\T$ contains a non-exceptional indecomposable object;
	\item $\F$ is  a cotilting torsion-free class, equivalently, $\F$ contains a non-exceptional indecomposable object.
	\end{enum2}
\end{lem}
\begin{proof}
Since there exists a nonzero morphsim between two non-exceptional indecomposable objects in $\A_t$, $\T$ and $\F$ cannot contain non-exceptional indecomposable objects in the meantime. If $\T$ is a tilting torsion class then it's easy to see that $\T$ contains a non-exceptional indecomposable object. Conversely, if $\T$ contains a non-exceptional indecomposable object $T$ then ${}^{[l]}\top(T)\in \T$ for all $l\in \Z_{\geq 1}$ since $\T$ is closed under quotient and extension.  Since any indecomposable object in $\A_t$ is an subobject of ${}^{[l]}\top(T)$ for some $l$, $\T$ is a tilting torsion class. Dual argument applies to conclude the asserted equivalence for $\F$. 
\end{proof}

We will need  the following criterion to  make sure that certain subcategory of $\D^b(\A_t)$ contains a non-exceptional indecomposable object. 
\begin{lem}\label{order sequence orthogonal}
	Let $\CC$ be a subcategory of $\A_t$ closed under extension and direct summand.  
	If each simple object in $\A_t$ occurs as a composition factor of some indecomposable object in $\CC$, equivalently, there is a sequence \[(X_0,X_1,\dots, X_{n-1}, X_n=X_0)\] of indecomposable objects in $\A_t$ with $\ext^1(X_i, X_{i-1})\neq 0$ $(1\leq i\leq n)$, then $\CC$ contains a non-exceptional indecomposable object. 
\end{lem}

\begin{proof}
	We claim that if $Y,Z$ are two non-isomorphic exceptional objects in $\A_t$ with $\ext^1(Z,Y)\neq 0$, then $\CC$ contains an indecomposable object $C$ such that $Y$ is a subobject of $C$ in $\A_t$ and $Z$ a quotient object of $C$ in $\A_t$. 
	Indeed, if $\ext^1(Z,Y)\neq 0$ then there are two objects $A,B$ in $\A_t$ such that $B$ is indecomposable, $A$ is a quotient object of $Y$ and $A,B$ fits into the exact sequence $0\ra A\ra Z\ra B\ra 0$.  Let $C$ be the unique (up to isomorphism) indecomposable object which fits into the exact sequence $0\ra Y\ra C\ra B\ra 0$. Then $Y$ (resp. $Z$) is a subobject (resp. quotient object) of $C$.
	Moreover, we have $\ext^1(C,A)=0$ and there is an exact sequence $0\ra Y\ra A\oplus C\ra Z\ra 0$. 
	Hence $C\in \CC$. This shows our claim.

	Now suppose that $\CC$ contains a sequence $(X_0,X_1,\dots,X_{n-1}, X_n=X_0)$ with the given property. Assume for a contradiction that $\CC$ contains no non-exceptional indecomposable object. In particular, each $X_i$ is exceptional. 
 Applying our claim to $Y=X_1,Z=X_2$, we obtain an indecomposable object $C_1\in \CC$ such that $X_1$ (resp. $X_2$) is a subobject (resp. quotient object) of $C_1$. 
 Then $\ext^1(X_1,X_0)\neq 0$ implies $\ext^1(C_1,X_0)\neq 0$; $\ext^1(X_3,X_2)\neq 0$ implies $\ext^1(X_3,C_1)\neq 0$. 
 Hence we have a sequence $(X_0, C_1, X_3, \dots, X_n)$  of length $(n-1)$  in $\CC$ 
which also satisfies the given property. 
 By assumption, $C_1$ is exceptional. Then repeating the above argument for $n$ times will eventually give us a sequence $(C)$ of length $1$ with $C$ indecomposable and $\ext^1(C,C)\neq 0$, whence $C$ is a non-exceptional indecomposable object in $\CC$, a contradiction. Hence $\CC$ must contain a non-exceptional object.

\end{proof}

We show an analogue of Lemma~\ref{simple ext-proj A} to perform induction. 

\begin{lem}\label{simple Ext-proj}
For a bounded t-structure  $(\D^{\leq 0}, \D^{\geq 0})$  on $\D^b(\A_t)$, which is not a shift of the standard t-structure, there is   some simple object  in $\A_t$ that is Ext-projective in some $\D^{\leq l}$. 

\end{lem}

\begin{proof}
	Let $\B$ be the heart of $(\D^{\leq 0}, \D^{\geq 0})$. 
	Each bounded t-structure on $\D^b(\A_t)$ is width-bounded with respect to the standard t-structure (see Example~\ref{fin simple width bounded}). Hence, $\B\subset \D_{\A_t}^{[m,n]}$ for some  $m,n$. We take $m$ to be maximal and $n$ minimal. 
Since there exists a nonzero morphism between two non-exceptional indecomposable objects in $\A_t$ and since $\hom(\B[-m], \B[-n])=0$,  
either i) $\B[-m]\cap \A_t$ or ii) $\B[-n]\cap \A_t$ contains no non-exceptional indecomposable  object. Suppose case i) occurs.
Then $\B[-m]\cap \A_t$ contains only finitely many indecomposables. Moreover, Lemma~\ref{order sequence orthogonal} implies that there is some indecomposable object $X$ such that $\ext^1(X, Y)=0$ for indecomposable object $Y\in \B[-m]\cap \A_t$ non-isomorphic to $X$.
Then we have $\hom^{>0}(X[m], \B)=0$, whence $X$ is $\D^{\leq m}$-projective. If case ii) happens then similarly we find an indecomposable object
$Y\in \A_t$ which is $\D^{\geq n}$-injective. This gives us a $\D^{\leq n}$-projective $\tau^{-1}Y[-1]$.
Anyway we have an exceptional object  $B\in \A_t$ that is Ext-projective  in some $\D^{\leq l}$.

Similarly as in the proof of Lemma~\ref{simple ext-proj A}, one can show that $B^{\perp_{\A_t}}$ decomposes as $B^{\perp_{\A_t}}=\B_1\coprod\B_2$, where \[\B_1=\pair{\tau^m (\top(B))\mid 1\leq m< l(B)}_{\A_t}\] and $\B_2$ is an exact subcategory of $\A_t$ closed under extension, that \[\bar{\B}_1:=\pair{\tau^m(\top (B))\mid 0\leq m< l(B)}_{\A_t}\simeq \mod k\vec{\AA}_{l(B)},\] 
and that $(\D^{\leq 0}, \D^{\geq 0})$ is compatible with the recollement 
\[ \xymatrix{ \D^b(\B_2)\ar[rr]|{i_*}   & &\ar@/_1pc/[ll]|{i^*} \ar@/^1pc/[ll]|{i^!}\D\ar[rr]|{j^*} & &\ar@/_1pc/[ll]|{j_!} \ar@/^1pc/[ll]|{j_*} \pair{\bar{\B}_1}_\D=\D^b(\bar{\B}_1),}\]
	where $i_*, j_!$ are inclusion functors. 
Moreover, we have a  bounded t-structure $(j^* \D^{\leq 0},j^* \D^{\geq 0})$ on $\D^b(\bar{\B}_1)\simeq \D^b(k\vec{\AA}_{l(B)}).$ We know from Lemma~\ref{simple ext-proj A} that  some $\tau^m (\top(B))$ is Ext-projective in some $j^* \D^{\leq l}$, which  gives us the desired Ext-projective object $\tau^m(\top(B))$  in $\D^{\leq l}$ by Lemma~\ref{ext-proj from recollement}. 
\end{proof}

Let $\SSS$ be a (possibly empty) proper collection of   simple objects in $\A_t$, where properness means that $\SSS$ does not contain a complete set of simple objects in $\A_t$ and  simple objects in $\SSS$  are pairwise non-isomorphic. Two such collections are said to be equivalent if they yield the same isoclasses of simple objects.  
If $\SSS$ is nonempty then there exist uniquely determined  $\{S_1,\dots, S_n\}\subset \SSS$ and positive integers $l_1,\dots, l_n$ such that 
\begin{equation}
	\SSS=\bigsqcup_{i=1}^n \{\tau^j S_i\mid 0\leq j< l_i\}.
\end{equation}
Since $\oplus_{1\leq i\leq n}\oplus_{0\leq j<l_i}{}^{[l_i-j]}\tau^{j}S_i$ is a projective generator for $\pair{\SSS}_{\A_t}$ whose endomorphism algebra is isomorphic to $k\vec{\AA}_{l_1}\times \dots \times k\vec{\AA}_{l_n}$, we have an equivalence
\begin{equation}\label{simple gen}
\pair{\SSS}_{\A_t}\simeq \coprod_{i=1}^n\mod k\vec{\AA}_{l_i},
\end{equation}
where $k\vec{\AA}_{l}$ is the path algebra of the equioriented $\AA_{l}$-quiver. In the sequel, we will also write in the form~\eqref{simple gen} when $\SSS$ is empty by defining the right hand side of ~\eqref{simple gen} to be the zero category. 
Since $\SSS^{\perp_{\A_t}}$ is a uniserial length abelian $k$-category whose Ext-quiver is an oriented cycle with $t-\sharp \SSS$ vertices, we have  an equivalence 
\begin{equation}\SSS^{\perp_{\A_t}}\simeq \A_{t-\sharp \SSS}.
\end{equation}

Bounded t-structures on $\D^b(\A_t)$ can be described as follows. 
\begin{prop}\label{A_t t-str}
	Given a bounded t-structure $(\D^{\leq 0}, \D^{\geq 0})$ on $\D^b(\A_t)$, there is a unique (up to equivalence) proper collection $\SSS$ of  simple objects in $\A_t$ such that 
	\begin{itemize} 
	\item $(\D^{\leq 0}, \D^{\geq 0})$ is compatible with the recollement
		\[ \xymatrix{\D^b(\SSS^{\perp_{\A_t}})=\SSS^{\perp_\D}  \ar[rr]|{i_*}   & &\ar@/_1pc/[ll] \ar@/^1pc/[ll] \D=\D^b(\A_t) \ar[rr] & &\ar@/_1pc/[ll]|{j_!} \ar@/^1pc/[ll] \pair{\SSS}_\D,}\] where $i_*,j_!$ are the inclusion functors;
	\item the corresponding t-structure on $\SSS^{\perp_\D}$ has heart $\SSS^{\perp_{\A_t}}[m]$ for some $m$.
	\end{itemize}
	In particular, each bounded t-structure on $\D^b(\A_t)$ has length heart.	
\end{prop}
\begin{proof}
	
	Since each bounded t-structure on $\pair{\SSS}_\D=\D^b(\pair{\SSS}_{\A_t})\simeq \D^b(\coprod_{i=1}^n\mod k\vec{\AA}_{l_i})$ has length heart (by Lemma~\ref{finite rep length heart}) and $\SSS^{\perp_{\A_t}}[m]$ is of finite length, by Lemma~\ref{no-ar-le}, the second assertion follows from the first. 
	We use induction on $t$ to prove the first assertion. 
	
	Suppose $t=1$. We have a unique (up to isomorphism) simple object $S$ in $\A_1$. So the asserted $\SSS$ is the empty set. We need show that any bounded t-structure on $\D^b(\A_1)$, whose heart is denoted by $\B$, is a shift of the standard one. Note that each indecomposable object in $\D^b(\A_1)$ is of the form $S^{[r]}[l]$ for some $r\in \Z_{\geq 1}, l\in \Z$. Since $\hom(S^{[r]}[l], S^{[r']}[l'])\neq 0$ for $l\leq l'$, we have $\B\subset \A_1[l]$ for some $l$. Then $\B=\A_1[l]$, as desired.

	Now consider $t>1$. If $\B$ is a shift of $\A_t$, just take $\SSS=\emptyset$. Suppose that $\B$ is not a shift of $\A_t$.  By Lemma~\ref{simple Ext-proj} and Lemma~\ref{ext-proj recollement}, for some simple $S$ in $\A_t$, $(\D^{\leq 0}, \D^{\geq 0})$ is compatible with the admissible subcategory $\D_1:=S^{\perp_\D}=\D^b(S^{\perp_{\A_t}})$. 
	$\A:=S^{\perp_{\A_t}}$ is equivalent to $\A_{t-1}$, and simple objects in $\A$ are $\tau S^{[2]}$ and those $S'$, which are simple in $\A_t$ and non-isomorphic to $\tau S$ and $S$. 
	By the induction hypothesis, for a proper collection  $\SSS_1$ of simple objects in $S^{\perp_{\A_t}}$, the corresponding t-structure on $\D_1=\D^b(S^{\perp_{\A_t}})$ is compatible with the admissible subcategory $\SSS_1^{\perp_{\D_1}}$ and the corresponding t-structure on $\SSS_1^{\perp_{\D_1}}$ has heart $\SSS_1^{\perp_{\A}}[m]$ for some $m$.  
	If $\tau S^{[2]}\in \SSS_1$, take $\SSS=\{\tau S, S\}\cup (\SSS_1\backslash \tau S^{[2]});$ if $\tau S^{[2]}\notin \SSS_1$, take $\SSS=\SSS_1\cup \{S\}$. 
	Then $\SSS_1^{\perp_{\D_1}}=\SSS^{\perp_\D}$ and $\SSS_1^{\perp_{\A}}=\SSS^{\perp_{\A_t}}$. By Lemma~\ref{compatible filtration}, $(\D^{\leq 0}, \D^{\geq 0})$ is compatible with the admissible subcategory $\SSS^{\perp_\D}$ and the corresponding t-structure on $\SSS^{\perp_\D}$ has heart $\SSS^{\perp_{\A_t}}[m]$.

	Let $(\D_1^{\leq 0},\D_1^{\leq 0})$ and $(\D_2^{\leq 0},\D_2^{\geq 0})$ be the corresponding t-structures on $\SSS^{\perp_\D}$ and $\pair{\SSS}_\D$, respectively. Note that $\D_1^{\leq 0}$ contains no nonzero Ext-projective object. Let $T$ be the direct sum of a complete set of indecomposable $\D^{\leq 0}$-projectives. Then by Lemma~\ref{ext-proj from recollement}, $T\in \pair{\SSS}_{\D}$ and $T$ is the direct sum of a complete set of indecomposable $\D_2^{\leq 0}$-projectives. Thus $T$ is a silting object in $\pair{\SSS}_\D=\D^b(\pair{\SSS}_{\A_t})$. In particular, $\pair{T}_\D=\D^b(\pair{\SSS}_{\A_t})$. As a complete set of simple objects in $\pair{\SSS}_{\A_t}$, the collection $\SSS $ is uniquely determined. This finishes the proof.
\end{proof}

\section{Weighted projective lines}\label{cha wpl}
For self-containedness, we review the basic theory of weighted projective lines in details in \S\ref{sec: wpl def}-\ref{sec: wpl perp}. The materials in \S\ref{sec: wpl def} are taken from the original article \cite{GL}, which introduced the notion of weighted projective lines.   For a recent survey of the theory, see \cite{Lenzing}. We fix an algebraically closed field $k$ in this section.

\subsection{Basic definitions and properties}\label{sec: wpl def}
Given a sequence $\udp=(p_1,\dots, p_t)$($t> 2$) of positive integers, define  an abelian group $L(\udp)$ of rank one by \[L(\udp)=\pair{\vec{x_1},\dots,\vec{x_t}, \vec{c}\mid p_1\vec{x_1}=\dots =p_t\vec{x_t}=\vec{c}}.\] Denote $\vec{\omega}=(t-2)\vec{c}-\sum_{i=1}^t\vec{x_i}$, which is called \emph{the dualizing element}. Each $\vec{x}\in L(\udp)$ can be written uniquely in the form \[\vec{x}=\sum_{i=1}^tl_i\vec{x_i}+l\vec{c}, \quad 0\leq l_i<p_i, l_i,l\in\Z.\] $L(\udp)$ is an ordered group if we define $\vec{x}\geq 0$ iff $\vec{x}\in\sum_{i=1}^t\Z_{\geq 0}\vec{x_i}$. Let $p=\lcm(p_1,\dots, p_t)$. We have a group homomorphism, called a degree map, \[\delta: L(\udp)\ra \Z, \quad\vec{x_i}\mapsto \frac{p}{p_i}.\]

\def\frakp{\mathfrak{p}}
\def\Proj{\textup{Proj}}
Let $\P^1=\P^1(k)$ be (the set of closed points of) the projective line over $k$.
Given  a sequence $\udp=(p_1,\dots, p_t)$ of positive integers and a sequence  $\udl=(\lambda_1,\dots, \lambda_t)$ of distinct points in  $\P^1$ (normalized such that $\lambda_1=\infty, \lambda_2=0,\lambda_3=1$), we define an algebra \[S=S(\udp, \udl)=k[X_1,\dots, X_t]/(X_i^{p_i}-X_2^{p_2}+\lambda_iX_1^{p_1}, 3\leq i\leq t).\]  Write $x_i=\bar{X_i}\in S$. 
$S$ becomes $L(\udp)$-graded with the assignment $\deg(x_i)=\vec{x_i}$ and thus $S=\oplus_{\vec{x}\in L(\udp)} S_{\vec{x}}$, where $S_{\vec{x}}$ consists of those homogeneous elements of degree $\vec{x}$. Using $S$ as the homogeneous coordinate algebra, \cite{GL} introduced  a weighted projective line $\X=\X(\udp,\udl)$. $\X$ is defined to be the $L(\udp)$-graded projective spectrum of $S$, which is the set  \[\Proj^{L(\udp)}S:=\{\text{$L(\udp)$-graded prime ideal $\frakp$ of $S$} \mid \frakp \nsupseteq S_+:=\oplus_{\vec{x}>0}S_{\vec{x}}\}\]  equipped with Zariski topology and a $L(\udp)$-graded structure sheaf $\O=\O_\X$. 
There is a bijection 
\begin{equation}\label{wpl-pl}
	\X(k)\lra \P^1,\quad [x_1,\dots, x_t]\mapsto [x_1^{p_1},x_2^{p_2}]
\end{equation}
between the set of closed points of $\X$ and  $\P^1$. By virtue of this bijection, the weighted projective line $\X$ is understood to be  the usual projective line $\P^1$, where weights $p_1,\dots, p_t$ are attached respectively to the $t$ points $\lambda_1,\dots,\lambda_t$. 
 We can  define $L(\udp)$-graded $\O_\X$-modules and coherent $L(\udp)$-graded $\O_\X$-modules. The category $\coh\X$ of $L(\udp)$-graded coherent $\O_\X$-modules over $\X=\X(\udp, \udl)$ is a noetherian hereditary abelian category with finite dimensional morphism and extension spaces. In particular, $\coh\X$ is a Krull-Schmidt category.   
 We have an analogue of Serre's theorem, that is,  we have an equivalence \[\coh\X\simeq \frac{\mod^{L(\udp)} S}{\mod^{L(\udp)}_0 S},\] where $\mod^{L(\udp)} S$ is the abelian category of $L(\udp)$-graded finite generated modules over $S$ and $\mod^{L(\udp)}_0 S$ is the Serre subcategory of $\mod^{L(\udp)} S$ consisting of modules of finite length. 
One may as well take the latter quotient category as the definition of $\coh\X$.

For $\vec{x}\in L(\udp)$, 
we have a natural $k$-linear autoequivalence of $\mod^{L(\udp)}S$ given by degree shifting by $\vec{x}\in L(\udp)$ on $L(\udp)$-graded $S$-modules $M$: $M(\vec{x})_{\vec{y}}=M_{\vec{x}+\vec{y}}$. 
And this  induces a $k$-linear autoequivalence $-(\vec{x})$ of $\coh\X$: $F\mapsto F(\vec{x}), F\in \coh\X$.
We denote by $\tau$ the $k$-linear autoequivalence $-(\vec{\omega})$ of $\coh\X$, where $\vec{\omega}$ is the dualizing element. 
\begin{thm}[Serre duality]\label{serre duality}
	For $X, Y\in \coh\X$, we have an isomorphism \[D\ext^1(X,Y)\cong \hom(Y, \tau X)\] functorial in $X,Y$, where $D=\hom_k(-,k)$. 
\end{thm}

Consequently, the bounded derived category $\D^b(\X)=\D^b(\coh\X)$ of $\coh\X$ has a Serre functor $\tau(-)[1]$.

There is a linear form $\rk: K_0(\X)\ra \Z$ on the Grothendieck group $K_0(\X)$ of $\coh\X$, called \emph{rank}, which is preserved under the action of $L(\udp)$. As usual, we have the notion of a \emph{locally free sheaf}, or a \emph{vector bundle}.  A \emph{line bundle} is a vector bundle of rank $1$. 
A coherent sheaf $F$ over $\X$ is called \emph{torsion} if it is of finite length in $\coh\X$, equivalently, if $\rk(F)=0$. Each coherent sheaf  over $\X$ decomposes as the direct sum of a torsion sheaf and a vector bundle. The subcategory of vector bundles resp. torsion sheaves over $\X$ is denoted by $\vect\X$ resp. $\coh_0\X$. We have $\hom(\coh_0\X,\vect\X)=0$. 

The function $\w: \P^1\ra \Z_{\geq 1}, \lambda\mapsto \left\{\begin{array}{ll} 1 & \text{if $\lambda\neq \lambda_i, \forall i$} \\p_i &  \text{if $\lambda=\lambda_i$}\end{array}\right.$  is called the \emph{weight function} of $\X$. A weight function of $\X$  obviously shares the same data as that given by the pair $(\udp, \udl)$.  $(p_1,\dots, p_t)$ is called the \emph{weight sequence} of $\X$. For $\lambda\in \P^1$, by virtue of the bijection~\eqref{wpl-pl}, we  denote by $\coh_\lambda\X$ the category of those torsion sheaves supported at $\lambda$. 
\begin{prop}
 The category $\coh_0\X$ of torsion sheaves decomposes into a coproduct $\coprod_{\lambda\in \P^1}\coh_\lambda\X$ of uniserial categories. 
	 The number of simple objects in $\coh_\lambda\X$ is $\w(\lambda)$. 
\end{prop}

 $\lambda_i$'s are called  \emph{exceptional} points and the remaining points of $\P^1$ \emph{ordinary} points. For an ordinary point $\lambda$, the unique simple sheaf $S$ supported at $\lambda$ fits into the exact sequence
\[0\lra\O\overset{X_2^{p_2}-\lambda X_1^{p_1}}{\lra} \O(\vec{c})\lra S\lra 0.\]
For an exceptional point $\lambda_i$, the exact sequences \[0\lra \O(j\vec{x}_i)\overset{X_i}{\lra}\O((j+1)\vec{x}_i)\lra S_{i,j}\lra 0,\quad j\in \Z/p_i\Z\]  characterize the $p_i$ pairwise non-isomorphic simple sheaves  $S_{i,j}$ supported at $\lambda_i$. The simple sheaf $S$ supported at an ordinary point satisfies $S(\vec{x})\cong S$ for any $\vec{x}\in L(\udp)$; the  simple sheaves $S_{i,j}$ supported at $\lambda_i$ satisfies $S_{i,j}(\vec{x})\cong S_{i,j+l_i}$ if $\vec{x}=\sum_{i=1}^nl_i \vec{x_i}$. In particular, $\tau S_{i,j}\cong S_{i,j-1}$. $S_{i,j}$ is an exceptional object iff $p_i>1$. 

\begin{rmk}
As a uniserial length abelian $k$-category whose Ext-quiver is an oriented cycle with $\w(\lambda)$ verticies,  $\coh_\lambda\X$ is equivalent to  the category  $\nilp k\tilde{\AA}_{\w(\lambda)-1}$ of nilpotent finite dimensional $k$-representations of the cyclic quiver $\tilde{\AA}_{\w(\lambda)-1}$ with $\w(\lambda)$ vertices. So the algebra $k\tilde{\AA}_{t-1}$ provides a local study of a weighted projective line. This accounts for the presence of \S\ref{sec: A_t}. 
\end{rmk}
Denote by $\pic\X$ the Picard group of $\X$, i.e., the group of isoclasses of line bundles under tensor product.
\begin{prop}
	\label{line bundle}
	\begin{enum2}
	\item  The mapping \[L(\udp)\lra \pic\X,\quad \vec{x}\mapsto \O(\vec{x})\] is an isomorphism. In particular, each line bundle over $\X$ is isomorphism to  $\O(\vec{x})$ for some $\vec{x}\in L(\udp)$. 
	
	\item Each nonzero bundle over $\X$ admits a line bundle filtration. That is, for a nonzero bundle $E$, there is a filtration \[0=E_0\subset E_1\subset \dots \subset E_n=E\] with line bundle factors $L_i=E_i/E_{i-1}$ ($0<i\leq n$).
	\end{enum2}
\end{prop} 
 
The Grothendieck group $K_0(\X)$ of $\coh\X$ (and thus the Grothendieck  group $K_0(\D^b(\X))$ of $\D^b(\X)$) is a finitely generated free abelian group of rank $\sum_{i=1}^t (p_i-1)+2$ with a basis $\{[\O(\vec{x})]\mid 0\leq \vec{x}\leq \vec{c}\}$. We have a linear form $\deg: K_0(\X)\ra \Z$, called \emph{degree}, such that $\deg \O(\vec{x})=\delta(\vec{x})$ for $\vec{x}\in L(\udp)$. 
The Euler form on $K_0(\X)$ is given by \[\chi(E,F)=\dim_k\hom(E,F)-\dim_k\ext^1(E,F)\] and  the averaged Euler form is defined by $\bar{\chi}(E,F)=\sum_{j=0}^{p-1}\chi(\tau^jE,F)$. 
\begin{thm}[Riemann-Roch Theorem]
	For $E,F\in \D^b(\X)$, we have  
	\[\bar{\chi}(E,F)=p(1-g_\X)\,\rk(E)\,\rk(F)+\deg(F)\rk(E)-\deg(E)\rk(F).\]
\end{thm}
Here $g_\X=1+\frac{1}{2}\delta(\vec{\omega})$ is the \emph{virtual genus} of $\X$. $\X$ is said to be of \emph{domestic} (resp. \emph{tubular}, resp. \emph{wild}) type if $g_\X<1$ (resp. $g_\X=1$, resp. $g_\X>1$), equivalently, $\delta(\vec{\omega})<0$ (resp. $\delta(\vec{\omega})=0$, resp. $\delta(\vec{\omega})>0$). $\X$ is of domestic type iff the weight sequence is $(1,p_1,p_2)$, $(2,2,n)\,(n\geq 2)$, $(2,3,3)$, $(2,3,4)$, $(2,3,5)$, up to permutation; $\X$ is of tubular type iff the weight sequence is $(2,2,2,2)$, $(3,3,3)$, $(2,3,6)$, $(2,4,4)$, up to permutation; weighted projective lines of wild type correspond to the remaining weight sequences.

A coherent sheaf $T$ over $\X$ is called a \emph{tilting sheaf} if it is 
a tilting object as an object in $\D^b(\X)$. A tilting sheaf $T$ yields a derived equivalence $\D^b(\X)\simeq \D^b(\End T)$ and induces a torsion pair $(\T,\F)$ in $\coh\X$, where 
\[\T=\{E\in\coh\X\mid \ext^1(T,E)=0\},\quad \F=\{E\in\coh\X\mid \hom(T,E)=0\}.\]  

\begin{thm}\label{der canonical algebra}
	There is a canonical tilting bundle $T=\oplus_{0\leq \vec{x}\leq \vec{c}}\O(\vec{x})$ over $\X$, whose endomorphism algebra is isomorphic to a canonical algebra $\Lambda$ with the same parameter $(\udp,\udl)$ in the sense of Ringel (\cite{Ringel}). In particular, we have a derived equivalence $\D^b(\Lambda)\simeq \D^b(\X)$.  
\end{thm}
Recall from \cite{Ringel} that a canonical algebra $\Lambda$ with  parameter $(\udp,\udl)$  is the path algebra of the quiver 
\[\xymatrix@R-25pt{& \vec{x}_1 \ar[r]^{x_1} & 2\vec{x}_1\ar[r]^{x_1} & \dots  \ar[r]^{x_1} & (p_1-2)\vec{x}_1 \ar[r]^{x_1} & (p_1-1)\vec{x}_1 \ar[dr]^{x_1} & \\
0 \ar[ur]^{x_1} \ar[r]^{x_2} \ar[ddr]^{x_t} & \vec{x}_2 \ar[r]^{x_2} & 2\vec{x}_2 \ar[r]^{x_2} & \dots  \ar[r]^{x_2} & (p_2-2)\vec{x}_2 \ar[r]^{x_2} & (p_2-1)\vec{x}_2 \ar[r]^{x_2} & \vec{c}\\
 & \vdots & \vdots & & \vdots & \vdots & \\
 & \vec{x}_t\ar[r]^{x_t} & 2\vec{x}_t\ar[r]^{x_t} & \dots \ar[r]^{x_t} & (p_t-2)\vec{x}_t\ar[r]^{x_t} & (p_t-1)\vec{x_t} \ar[uur]^{x_t} & 
}\]
with relations $x_i^{p_i}=x_2^{p_2}-\lambda_i x_1^{p_1}$ ($i=3,\dots, t$).

\subsection{A glimpse of Auslander-Reiten theory}\label{sec: AR theory}
Auslander-Reiten (=AR) theory  is introduced by Auslander and Reiten to study representations of artin algebras. The standard reference is \cite{ARS} (see also \cite{Aus}).  
The central concept (i.e. an almost split sequence, or an Auslander-Reiten sequence) makes sense in any Krull-Schmidt category with short exact sequences (in the sense of \cite[\S2.3]{Ringel}) but there is a problem of existence. 
Later Happel introduced in \cite{Happel} the notion of an Auslander-Reiten triangle,  a triangulated version of Auslander-Reiten sequence.  \cite{RVDB} investigated the close relationship between  Serre duality (in the sense of \cite{RVDB}) and Auslander-Reiten sequences (as well as Auslander-Reiten triangles). 

\def\irr{\textup{Irr}}
\def\rad{\textup{rad}}
Here we recall some basic definitions and we follow \cite{Ringel}. 
Let $\A$ be an essentially small Hom-finite abelian $k$-category.   
If $X$ and $Y$ are indecomposable, $\rad(X,Y)$ denotes the $k$-subspace of $\hom(X,Y)$ consisting of non-invertible morphisms. If $X=\oplus_{j=1}^mX_j, Y=\oplus_{i=1}^n Y_i$, where $X_j, Y_i$'s are indecomposable, then $\rad(X,Y)$ denotes the $k$-subspace of $\hom(X,Y)$  consisting of those $f=(f_{ij})$ with $f_{ij}\in \rad(X_j, Y_i)$.  $\rad^2(X,Y)$ denotes the $k$-subspace of $\hom(X,Y)$ consisting of morphisms of the form $gf$ with $f\in \rad(X,M)$, $g\in \rad(M,Y)$ for some  $M$. Let \[\irr(X,Y)=\rad(X,Y)/\rad^2(X,Y).\] 
A morphism $h: X\ra Y$  is called \emph{irreducible} if $h$ is neither a split monomorphism nor a split epimorphism and if $h=ts$ for some $s: X\ra Z$ and $t: Z\ra Y$, then $s$ is a split monomorphism or $t$ is a split epimorphism. $h: X\ra Y$ is irreducible iff $h\in\rad(X,Y)\backslash \rad^2(X,Y)$. 

A morphism $f: B\ra C$ in $\A$  is called a \emph{sink map} (or a \emph{minimal right almost split} morphism) if
\begin{enumi}
\item $f$ is right almost split, that is, $f$ is not an split epimorphism and   any morphism $X\ra C$ which is not a split epimorphism factors through $f$, and 
\item $f$ is right minimal, that is,  $\gamma\in \End(B)$ satisfying $f\gamma=f$ is an automorphism.
\end{enumi}
Dually, one defines  a \emph{source map} (or a \emph{minimal left almost split} morphism).  Sink (resp. source) maps
with a fixed target (resp. source), if they exist, are  obviously unique up to isomorphism.  If $f: B\ra C$ is a sink (resp. source) map then $C$ (resp. $B$) is indecomposable.  
An exact sequence $0\ra A\overset{g}{\ra} B\overset{f}{\ra} C\ra 0$ in $\A$ is called  an \emph{AR sequence} (or an \emph{almost split sequence})   if $g$ is a source map, equivalently, if $f$ is a sink map (see \cite[\S2.2, Lemma 2]{Ringel} for the equivalence). 
If such an AR sequence exists, then each irreducible map $f_1: A\ra B_1$ (or $g_1: B_1\ra C$) fits into an AR sequence \[0\lra A\overset{(f_1,f_2)^t}{\lra} B_1\oplus B_2\overset{(g_1, g_2)}{\lra} C\lra 0.\]
We say that $\A$ has sink (resp. source) maps if for each  indecomposable object $A\in \A$, there exists  a sink map $B\ra A$ (resp. a source map $A\ra C$).  We say that $\A$ has AR sequences (or almost split sequences) if $\A$ has both sink and source maps.

If $\A$ has AR sequences then  the \emph{AR quiver}
$(\Gamma_\A,\sigma)$ of $\A$, which turns out to be a translation quiver, is defined as follows.  The vertex set of $\Gamma_\A$ is in bijection with a complete set of representatives of isoclasses  of indecomposable objects in $\A$. Denote the vertex corresponding to an indecomposable object $M$ by $[M]$. 
The number of arrows from a vertex $[M]$ to another vertex $[N]$ is $\dim_k\irr(M,N)$.  By \cite[\S2.2, Lemma 3]{Ringel}, if $A\ra B$ is a source map then   there are $d$ arrows from $[A]$ to $[D]$ iff the multiplicity of $D$ as a direct summand of $B$ is $d$. There is a dual fact for a sink map.
So if $0\ra A\ra B\ra C\ra 0$ is an AR sequence then there are $d$ arrows from  $[A]$ to $[D]$ iff there are $d$ arrows from $[D]$ to $[C]$. 
The translation $\sigma$, called the \emph{AR translation} of $\A$, is such that $\sigma [C]=[A]$ if $0\ra A\ra B\ra C\ra 0$ is an AR sequence.

The existence of AR sequences as well as the existence of AR triangles is closely related to the existence of a Serre functor. We refer the reader to \cite{RVDB} and here we only record the following fact (see \cite[Theorem I.3.3]{RVDB}): if $\A$ is a hereditary abelian $k$-category with finite dimensional morphism and extension spaces, then the existence of a Serre functor of $\D^b(\A)$  implies the existence of  AR sequences in $\A$. Consequently, if $\X$ is a weighted projective line then $\coh\X$  admits AR sequences. 
\begin{prop}[{\cite[Corollary 2.3]{GL}}]
	Let $\X$ be a weighted projective line. $\coh\X$ has AR sequences with  AR translation given by $[M]\dashrightarrow [\tau M]$.
\end{prop}
 AR sequences are obtained in the following way. 
For each indecomposable sheaf $E$ over $\X$, we have a distinguished exact sequence $\eta_E: 0\ra \tau E\ra F\ra E\ra 0$ whose class in $\ext^1(E,\tau E)$ corresponds to $\id_{\tau E}$ under Serre duality $D\ext^1(E,\tau E)\cong \hom(\tau E, \tau E)$. The exact sequence $\eta_E$ is an AR sequence.   
Since $\tau$ is an autoequivalence of $\coh\X$, $0\ra E\ra \tau^{-1} F\ra \tau^{-1}E\ra 0$ is also an  AR sequence.  

An additive subcategory $\C$ of $\coh\X$ closed under direct summand is said to be closed under the formation of AR sequences if for any AR sequence $0\ra \tau E\ra F\ra  E\ra 0$, $E\in \C$ implies $F\in \C$ and $\tau^iE\in \C$ for all $i\in\Z$.   In this case, we can talk about the AR quiver of $\C$ and the AR quiver of $\C$ is a union of certain components of the AR quiver of $\coh\X$. 
For each $\lambda\in \P^1$, $\coh_\lambda\X$ is closed under the formation of AR sequences and the AR quiver of $\coh_\lambda\X$ is a tube of rank $\w(\lambda)$, where $\w$ is the weight function of $\X$, and thus the AR quiver of $\coh_0\X$ is a family of tubes parametrized by $\P^1$.  $\vect\X$ is also closed under the formation of AR sequences. We will see in the next subsection the shape of the AR quiver of $\vect\X$ for a domestic or tubular weighted projective line $\X$. We mention that  for a wild weighted projective line $\X$, each AR component of $\vect\X$   has the shape $\Z\AA_\infty$ \cite{LP}.

We introduce more definitions for the sake of the next subsection. Let $E$ be an indecomposable object in $\coh\X$ lying in a component which is  a tube of finite rank. 
The \emph{quasi-length} of $E$ is the largest integer $l$ such that there exists a sequence $E=A_l\epic A_{l-1}\epic \dots \epic A_2\epic A_1=A$ of irreducible epimorphisms, equivalently, there exists a sequence $B=B_1\monic B_2\monic \dots \monic B_{l-1}\monic B_l=E$ of irreducible monomorphisms. In this case, we say $A$ (resp. $B$) is the \emph{quasi-top} (resp. \emph{quasi-socle}) of $E$. 
$E$ is called \emph{quasi-simple} if $E$ is of quasi-length one, i.e., $E$ lies at the bottom of the tube. Note that the quasi-length of an indecomposable finite length sheaf coincides with its length and a quasi-simple torsion sheaf is just a simple sheaf. The $\tau$-period of $E$ is the minimal positive integer $n$ such that $\tau^nE\cong E$, which equals  the rank of the tube.

\subsection{Vector bundles over a domestic or tubular weighted projective line}\label{sec: wpl vect}

We  first recall the notion of stability of a vector bundle. For a nonzero bundle $F$ over a weighted projective line $\X$, its slope $\mu(F)$ is defined as $\mu(F)=\deg(F)/\rk(F)$.

\begin{lem}[{\cite[Lemma 2.5]{Lenzing}}]\label{twist slope}
	We have $\mu(F(\vec{x}))=\mu(F)+\delta(\vec{x})$. In particular, $\mu(\tau F)=\mu(F)+\delta(\vec{\omega})$.  
\end{lem}

$F$ is called \emph{semistable} (resp. \emph{stable}) if $\mu(E)\leq$ (resp. $<$) $\mu(F)$ for any subbundle $E$ of $F$ with $\rk(E)<\rk(F)$. 
For $\mu\in \Q$,  denote by $\coh^\mu\X$ the subcategory of $\coh\X$ consisting of semistable bundles of slope $\mu$. $\coh^{\mu}\X$ is a length abelian category  whose simple objects are precisely stable bundles of slope $\mu$. 
For a torsion sheaf $T$, we define $\mu(T)=\infty$ and denote $\coh^\infty \X=\coh_0\X$.   We have $\hom(\coh^{\mu}\X, \coh^{\mu'}\X)=0$ for $\mu >\mu'$.

As in the case of smooth projective curves, the maximal destabilizing subsheaf exists in our case, and thus each nonzero bundle admits a Harder-Narasimhan filtration, that is, a sequence \[0=F_0\subset F_1\subset \dots \subset F_m=F\] such that all the factors $A_i=F_i/F_{i-1}$ ($0<i\leq m$) are semistable bundles and \[\mu(A_1)>\mu(A_2)>\dots >\mu(A_m).\]  Such a filtration is unique up to isomorphism. $A_i$ are called the semistable factors of $F$. We will denote \[\mu^+(F)=\mu(A_1),\quad \mu^-(F)=\mu(A_m).\]
Let $\mu\in\bar{\R}=\R\cup \{\infty\}$. Denote \[\coh^{\geq \mu}\X=\{E\in\coh\X\mid \mu^-(E)\geq \mu\},\quad \coh^{< \mu}\X=\{E\in\coh\X\mid \mu^+(E)< \mu\}.\] Similarly one defines $\coh^{>\mu}\X, \coh^{\leq\mu}\X$. Then we have  torsion pairs \[(\coh^{\geq \mu}\X, \coh^{<\mu}\X),\quad (\coh^{>\mu}\X,\coh^{\leq \mu}\X)\] for each $\mu\in\bar{\R}$.

Suppose $\X$ is a weighted projective line  of domestic type with weight sequence $(p_1,p_2,p_3)$. Then up to permutation, \[(p_1,p_2,p_3)=(1,p_2,p_3), (2,2,n) (n\geq 2), (2,3,3), (2,3,4), \text{or}\, (2,3,5).\] Let $\Delta=\Delta(p_1,p_2,p_3)$ be  
the Dynkin diagram 
 \[\xymatrix@R-20pt@C-5pt{
	 \overset{(1, p_1-1)}{\tinybullet}\ar@{-}[r] & \overset{(1, p_1-2)}{\tinybullet} \ar@{.}[r] & \overset{(1, 2)}{\tinybullet} \ar@{-}[r] & \overset{(1, 1)}{\tinybullet}\ar@{-}[dr] & \\
	 \overset{(2, p_2-1)}{\tinybullet}\ar@{-}[r] & \overset{(2, p_2-2)}{\tinybullet} \ar@{.}[r] & \overset{(2,2)}{\tinybullet} \ar@{-}[r] & \overset{(2,1)}{\tinybullet}\ar@{-}[r] & \tinybullet \\
	 \overset{(3, p_3-1)}{\tinybullet}\ar@{-}[r] & \overset{(3, p_3-2)}{\tinybullet} \ar@{.}[r] & \overset{(3, 2)}{\tinybullet} \ar@{-}[r] & \overset{(3,1)}{\tinybullet}\ar@{-}[ur] &}\]
 Let $\tilde{\Delta}$ be the extended Dynkin diagram attached to $\Delta$. We collect well-known and basic properties of vector bundles over a domestic weighted projective line in the following theorem.

\begin{thm}\label{domestic bundle} Let $\X$ be a weighted projective line of domestic type with weight sequence $(p_1,p_2,p_3)$. 

	\begin{enum2}
	\item Each indecomposable bundle over $\X$ is stable and exceptional. The rank function $\rk$ is bounded on indecomposable bundles over $\X$. If some $p_i$ equals $1$ then each indecomposable bundle is a line bundle. 

\item The direct sum of a complete set of indecomposable bundles with slope in the interval $(\delta(\vec{\omega}),0]$ is a tilting bundle and its endomorphism algebra is  the path algebra $k\vec{\tilde{\Delta}}$ of 
an extended Dynkin quiver $\vec{\tilde{\Delta}}$ with underlying graph $\tilde{\Delta}$. 
In particular, we have a derived equivalence $\D^b(\X)\simeq \D^b(k\vec{\tilde{\Delta}})$. If each $p_i\geq 2$, then $\vec{\tilde{\Delta}}$ has a bipartite orientation. 

\item The Auslander-Reiten quiver of $\vect\X$ consists of a single component having the form $\Z\tilde{\Delta}$. 
\end{enum2}
\end{thm}
\begin{proof}
	The first statement in (1) is \cite[Proposition 5.5(i)]{GL}. The last statement in (1) is \cite[Corollary 3.8]{Lenzing}.  (2) and (3) are due to \cite{H89} (see also \cite[Theorem 3.5]{Lenzing}, \cite[Proposition 5.1]{KLM}).  It remains to show the second statement in (1). 
In fact,  the underlying graph $\Omega$ of the AR quiver   of $\vect\X$  is determined by the following  observations: 
\begin{enumi}
\item $\rk$ is an additive function on the full  sub-graph $\Omega_0$ of $\Omega$ consisting of vertices corresponding to indecomposable bundles with slope in $(\delta(\vec{\omega}), 0]$; 
\item the number of vertices of $\Omega_0$ is equal to the rank $\sum_{i=1}^3(p_i-1)+2$ of $K_0(\X)$ (since the direct sum of pairwise non-isomorphic indecomposable bundles with slope in the interval $(\delta(\vec{\omega}), 0]$ is a tilting bundle);
\item the number of line bundles with slope in the interval $(\delta(\vec{\omega}), 0]$ is $[L(\udp):\Z\vec{\omega}]$ (by Proposition~\ref{line bundle}(1)), which is equal to $p_2+p_3$ ($4$, $3$,  $2$, $1$, respectively) if $(p_1,p_2,p_3) = (1,p_2,p_3)$ ($(2,2,n)$ ($n\geq 2$),  $(2,3,3)$,  $(2,3,4)$,  $(2,3,5)$, respectively). 
\end{enumi}
In particular, rank of indecomposable bundles are explicitly known and form a bounded set since $\tau$ preserves rank. 

\end{proof}
\begin{rmk}
	\begin{enum2}
		\item To show that the endomorphism algebra $\End(T)$ of the tilting bundle $T$ given in Theorem~\ref{domestic bundle}(2) is a hereditary algebra, instead of using the argument in \cite{KLM}, we can also argue as follows. By Proposition~\ref{length torsion pair},  there are a bounded t-structure with heart $\B\subset \coh\X[1]*\coh\X$ and  an equivalence $\B\simeq \mod\, \End(T)$. Clearly we have $\hom_{\D^b(\X)}^2(\B, \B)=0$. Since there is a monomorphism $\ext^2_{\B}(X,Y)\monic \hom^2_{\D^b(\X)}(X, Y)$ for $X, Y\in\B$, we have $\ext^2_{\B}(\B,\B)=0$, that is, $\B$ is hereditary. So $\End(T)$ is a hereditary algebra.

\item We remark why $\vec{\tilde{\Delta}}$ has a bipartite partition if each $p_i\geq 2$. This is obtained via  a case-by-case analysis using AR-sequences and starting from line bundles with slope in the interval $(\delta(\vec{\omega}), 0]$. For example, if $(p_1,p_2,p_3)=(2,3,4)$, then the full subquiver of the AR quiver of $\vect\X$ consisting of those indecomposable bundles with slope in $(\delta(\vec{\omega}), 0]$ can be depicted as follows
	\[\xymatrix@=1.3em{
		& & & [E_2] & & & \\
		[\O]  & [E_1]\ar[l]\ar[r] & [F] & \ar[l] [G]\ar[u]\ar[r] & [F(\vec{x_1}-2\vec{x_3})] & [E_1(\vec{x_1}-2\vec{x_3})]\ar[l]\ar[r] & [\O(\vec{x_1}-2\vec{x_3})].}
	\]
It follows that $\vec{\tilde{\Delta}}$ has a bipartite partition.
\end{enum2}

\end{rmk}

Now suppose $\X$ is of tubular type. We have an interesting and extremely useful class of  exact autoequivalences of $\D^b(\X)$, called \emph{telescopic functors}. These functors are 
introduced in \cite{LM3} as equivalences between subcategories of $\coh\X$  and extended in \cite{Meltzer2} as exact autoequivalences of $\D^b(\X)$. \cite{Meltzer} is a good reference for these functors. 
\begin{thm} Let $\X$ be a weighted projective line of tubular type.  
	For each $q,q'\in \bar{\Q}$, there is an exact autoequivalence $\Phi_{q,q'}$ of  $\D^b(\X)$, called a telescopic functor, such that $\Phi_{q,q'}(\coh^{q'}\X)=\coh^{q}\X$. Moreover, these functors satisfy the conditions $\Phi_{q'',q}=\Phi_{q'',q'}\circ \Phi_{q',q}$ and $\Phi_{q,q}=\id$.
\end{thm}

Denote $\coh^\mu_\lambda\X=\Phi_{\mu,\infty}(\coh_\lambda\X)$.
The next theorem summarizes well-known and basic properties of vector bundles over a tubular weighted projective line. 
\begin{thm}\label{tubular bundle}
	Let $\X$ be a weighted projective line of tubular type.

	\begin{enum2}
	\item We have $\coh_\lambda^\mu\X\simeq \coh_\lambda\X$ and $\coh^\mu\X$ decomposes as $\coh^\mu\X=\coprod_{\lambda\in\P^1}\coh^\mu_\lambda\X$. In particular, each $\coh^\mu_\lambda\X$ as well as $\coh^\mu\X$ is a uniserial abelian category.
	
	\item Each indecomposable bundle over $\X$ is semistable. $\coh^\mu_\lambda\X$ is closed under the formation of Auslander-Reiten sequences and the Auslander-Reiten quiver of $\coh^\mu_\lambda\X$ is a tube of rank $\w(\lambda)$, where $\w$ is the weight function of $\X$. In particular, the Auslander-Reiten quiver of $\vect\X$ is a family of tubes parametrized by $\Q\times \P^1$.

	\item  An indecomposable bundle in $\coh^\mu_\lambda\X$ is exceptional iff its quasi-length is less than $\w(\lambda)$. An indecomposable bundle over $\X$ is stable iff  it is quasi-simple. A stable bundle in $\coh^\mu_\lambda\X$ has $\tau$-period $\w(\lambda)$. 
	\end{enum2}
\end{thm}
\begin{proof}
	The assertion that each indecomposable bundle is semistable is \cite[Proposition 5.5(ii)]{GL}.  The remaining assertions follow from facts on $\coh_0\X$  by applying a suitable telescopic functor. We remark that a telescopic functor commutes with $\tau$ since any exact autoequivalence commutes with a Serre functor. 
\end{proof}

Here we make an observation needed in the following two lemmas. 
Let $(p_1,\dots, p_t)$ be the weight sequence of $\X$. Recall that we denote by $p=\lcm(p_1,\dots, p_t)$. Since $\X$ is of tubular type, there is some $p_i$ equal to $p$. So there exists a simple sheaf  $S$ with $\tau$-period $p$. 

For $F\in \coh(\X)$ and $n\in\Z$, we define the slope $\mu(F[n])$ of the object $F[n]\in \D^b(\X)$ to be $\mu(F[n])=\mu(F)$. We will need to know the effect of the telescopic functor $\Phi_{\infty,q}$ on slope and the essential image of $\coh^\mu\X$ under $\Phi_{\infty,q}$.  
\begin{lem}\label{fractional linear}
	\begin{enum2}
	\item There is a fractional linear map 
\begin{equation}\label{phi_q}
	\phi_q: \bar{\R}\ra \bar{\R},\,\, \mu\mapsto \frac{a\mu+b}{c\mu+d}, 
	\end{equation}
	where $\left(\begin{matrix} a & b\\c & d\end{matrix}\right)\in SL(2,\Z)$, such that \[\mu(\Phi_{\infty,q}(E))=\phi_q(\mu(E))\] for a sheaf $E$. 

	\item  For $\mu\in \bar{\Q}$, we have 

	\begin{equation}
		\Phi_{\infty,q}(\coh^\mu \X)=\left\{\begin{array}{ll} \coh^{\phi_q(\mu)}\X & \text{if}\,\, \mu\leq q,\\
	\coh^{\phi_q(\mu)}\X[1] & \text{if}\,\,\mu>q.\end{array}\right.
	\end{equation}
\end{enum2}
\end{lem}
\begin{proof}
	Recall from  \cite[Chapter 5]{Meltzer} that for an indecomposable coherent sheaf $E$ over $\X$ with $\tau$-period $p_E$,  the tubular mutation functor $T_{\tau^\bullet E}$ with respect to the $\tau$-orbit of $E$, which is an exact autoequivalence of $\D^b(\X)$, fits into a triangle \[\oplus_{j=0}^{p_E-1}\hom^\bullet(\tau^jE,-)\otimes \tau^jE\ra \id\ra T_{\tau^\bullet E}\cra .\]
	Define an action of $SL(2,\Z)$ on $\bar{\Q}$ by \[\begin{pmatrix} a & b \\ c & d\end{pmatrix}.q=\frac{aq+b}{cq+d}.\] 
	By \cite[Corollary 5.2.3]{Meltzer},  $T_{\tau^\bullet \O}(\coh^q\X)$  is a shift of $\coh^{\frac{q}{1-q}}\X$ for each $q\in \bar{\Q}$.
	Let $S$ be a simple sheaf with $\tau$-period $p$. 
	From the triangle \[\oplus_{j=0}^{p-1}\hom^\bullet(\tau^j S, -)\otimes \tau^j S\ra \id\ra T_{\tau^\bullet S}\cra,\] we see that  $T_{\tau^\bullet S}(\coh^q\X)=\coh^{1+q}\X$ for $q\in \bar{\Q}$.  So $T_{\tau^\bullet S}$ ($T_{\tau^\bullet S}^{-1}$,  $T_{\tau^\bullet \O}$, respectively) acts on slopes by $\begin{pmatrix} 1 & 1 \\ 0 & 1\end{pmatrix}$ ($\begin{pmatrix} 1 & -1 \\ 0 & 1\end{pmatrix}$,  $\begin{pmatrix} 1 & 1 \\ 0 & 1\end{pmatrix}$, respectively). 
	By definition, $\Phi_{q,\infty}=\Phi_{\infty, q}^{-1}$ is a composition of a sequence of the functors $T_{\tau^\bullet S}, T_{\tau^\bullet S}^{-1}, T_{\tau^\bullet \O}$  (see \cite[Theorem 5.2.6]{Meltzer}).  
	So we have a unique function $\phi_q: \bar{\Q}\ra \bar{\Q}$ such that $\phi_q(\mu)=\frac{aq+b}{cq+d}$ for some $\begin{pmatrix}a & b\\c& d\end{pmatrix}\in SL(2,\Z)$ and such that $\Phi_{\infty,q}(\coh^\mu\X)$ is a shift of $\coh^{\phi_q(\mu)}\X$ for each $\mu\in \bar{\Q}$. We extend $\phi_q$ to be the function \[\phi_q: \bar{\R}\ra \bar{\R},\quad r\mapsto \frac{ar+b}{cr+d}.\]  
	By Riemann-Roch Theorem, we have $\hom(\coh^\mu\X, \coh^{\mu'}\X)\neq 0$ for $\mu< \mu'$. 
 Now that $\Phi_{\infty,q}(\coh^q\X)=\coh^{\infty}\X$, (2) follows immediately. 
	
\end{proof}

	It's well-known that a stable bundle over an elliptic curve defined over an algebraically closed field has coprime rank and degree. We have the following analogue\footnote{Prof. Lenzing informed me of this fact as an answer to my question.} for a stable bundle over a tubular weighted projective line, which is implicit in \cite{LM3}. Actually, there is a parallel proof for an elliptic curve. 
\begin{lem}\label{stable gcd}
	Let $\X$ be a weighted projective line of tubular type and $E$ a stable vector bundle over $\X$ with $\tau$-period $p_E$. Then \[\gcd(\rk(E), \deg(E))=\frac{p}{p_E}.\] 
\end{lem}

\begin{proof}
Let $S$ be a simple sheaf with $\tau$-period $p$. 
 By Riemann-Roch Theorem, the linear form $\deg: K_0(\X)\ra \Z$ coincides with $\bar{\chi}(\O,-)$ and the linear form $\rk: K_0(\X)\ra \Z$ with $\bar{\chi}(-,S)$. 
 So we have  \[\deg(E)=\bar{\chi}(\O,E)=\frac{p}{p_E}\sum_{j=0}^{p_E-1}\chi(\tau^i\O, E),\quad \rk(E)=\bar{\chi}(E,S)=\frac{p}{p_E}\sum_{j=0}^{p_E-1}\chi(\tau^iE, S),\] whence  $\frac{p}{p_E}\mid \gcd(\deg(E), \rk(E))$. 
Let $S'=\Phi_{\infty, \mu(E)}(E)$. $S'$ is a simple sheaf with $\tau$-period $p_{S'}=p_E$. Observe that there exists $\vec{x}\in L(\udp)$ such that $\bar{\chi}(\O(\vec{x}), S')=\frac{p}{p_E}$. Take $F=\Phi_{\mu(E), \infty}(\O(\vec{x}))$. 
	Then we have 
	\[\deg(F)\rk(E)-\deg(E)\rk(F)  =\bar{\chi}(F,E)=\bar{\chi}(\O(\vec{x}),S')=\frac{p}{p_E}.\] Hence $\gcd(\rk(E), \deg(E))=\frac{p}{p_E}$.
\end{proof}
\subsection{Perpendicular categories}\label{sec: wpl perp}
Let $\X=\X(\udp, \udl)$ be a weighted projective line with weight sequence $\udp=(p_1,\dots, p_t)$. For convenience, we will denote $\A=\coh\X, \D=\D^b(\X)$. 
For a collection $\SSS$ of objects in $\coh\X$,  we have $\SSS^{\perp_\A}={}^{\perp_\A}\tau \SSS$ by Serre duality. So it sufficies to describe right perpendicular categories. We are   concerned about perpendicular categories of an exceptional sequence.

A (possibly empty) collection  of  simple sheaves over $\X$ is called \emph{proper} if it does not contain  a complete set of simple sheaves supported at $\lambda$ for each $\lambda\in \P^1$ and simple sheaves in the collection are pairwise non-isomorphic. In particular, it contains only exceptional simple sheaves.
\begin{thm}[{\cite{GL2}}]\label{simple perp}
Let $\SSS=\bigcup_{i=1}^t\SSS_i$ be a collection  of simple sheaves, where $\SSS_i$ is a  proper collection of simple sheaves  supported at $\lambda_i$. 

\begin{enum2}
\item We have an equivalence $\SSS^{\perp_\A}\simeq \coh\X'$ preserving rank, where $\X'=\X(\udp',\udl)$ is a weighted projective line with weight sequence \[\udp'=(p_1-\sharp\SSS_1,\dots, p_i-\sharp\SSS_i,\dots, p_t-\sharp\SSS_t).\]
	
\item The inclusion of the exact subcategory $\SSS^{\perp_\A}$ into $\A=\coh\X$ admits an exact left adjoint and an exact right adjoint, both of which preserve rank. 
\end{enum2}
\end{thm}

\begin{lem}\label{extor perp}
	Let $E$ be an exceptional torsion sheaf. Denote 
\begin{equation}\label{extor simple}
	\SSS_E=\{\tau^i\top(E)\mid 0\leq i<l(E)\},\quad \SSS_E'=\SSS_E\backslash \{\top(E)\}.
\end{equation} Then $E^{\perp_\A}$ decomposes as 
	\[E^{\perp_\A}=\SSS_E^{\perp_\A}\coprod \pair{\SSS_E'}_{\A},\]
	and we have an equivalence  $\SSS_E^{\perp_\A}\simeq \coh\X'$ preserving rank,  where $\X'=\X(\udp', \udl)$ is a weighted projective line with weight sequence \[\udp'=(p_1,\dots, p_i-l(E), \dots, p_t),\] and an equivalence $ \pair{\SSS_E'}_{\A}\simeq \mod k\vec{\AA}_{l(E)-1},$ where $k\vec{\AA}_{l}$ is the path algebra of the equi-oriented $\AA_{l}$-quiver. 
\end{lem}
Note that if $\X$ is of tubular type then $\X'$ is of domestic type.
\begin{proof}
	Suppose $E$ is supported at $\lambda$. We have a decomposition  \[E^{\perp_\A}\cap \coh_\lambda\X=E^{\perp_{\coh_\lambda\X}}=(\SSS_E^{\perp_\A}\cap \coh_\lambda\X)\coprod \pair{\SSS_E'}_\A.\] The argument for showing this is similar to that in showing $N^{\perp}=\A_1\coprod \A_2$ in the proof of Lemma~\ref{simple ext-proj A}. 
	For $\lambda\neq \lambda'\in \P^1$, since $\hom(\coh_\lambda\X, \coh_{\lambda'}\X)=0$, we have \[E^{\perp_\A}\cap \coh_{\lambda'}\X=\coh_{\lambda'}\X=\SSS_E^{\perp_\A}\cap\coh_{\lambda'}\X.\] 
	We continue to show  \[E^{\perp_\A}\cap \vect\X=\SSS_E^{\perp_\A}\cap \vect\X.\] It sufficies to show that  each nonzero bundle  $F$ lying in $E^{\perp_\A}$ lies in $\SSS_E^{\perp_\A}$. Assume for a contradiction that $F\notin \SSS_E^{\perp_\A}$. Then for some $S\in \SSS_E$, $\ext^1(S, F)\neq 0$, whence $\hom(F,\tau S)\neq 0$ by Serre duality. Since $\tau S$ is a composition factor of $\tau E$ and since $\hom(F,-): \coh_\lambda\X\ra \mod k$ is an exact functor,   $\hom(F,\tau S)\neq 0$ implies $\hom(F, \tau E)\neq 0$. Hence $\ext^1(E, F)\neq 0$, a contradiction to $F\in E^{\perp_\A}$. So indeed we have \[E^{\perp_\A}\cap \vect\X=\SSS_E^{\perp_\A}\cap \vect\X.\] By Serre duality, this implies $\hom(E^{\perp_\A}\cap \vect\X, \pair{\SSS'_E}_\A)=0$. 
	Now that  each coherent sheaf over $\X$ is a direct sum of a bundle and a torsion sheaf and that  $\coh_0\X=\coprod_{\lambda\in \P^1}\coh_\lambda\X$, we can conclude   \[E^{\perp_\A}=\SSS_E^{\perp_\A}\coprod \pair{\SSS_E'}_{\A}.\]
	One easily sees  $\pair{\SSS_E'}_{\A}\simeq \mod k\vec{\AA}_{l(E)-1}$. 
 By Theorem~\ref{simple perp}, we have an equivalence $\SSS_E^{\perp_\A}\simeq \coh\X'$ preserving rank, where $\X'$ has a weight sequence as asserted. 
\end{proof}

\begin{thm}\label{bundle perp} \label{line bundle perp}
	\begin{enum2}
	\item \textup{(\cite{HL}; see also \cite[Kapitel 5]{H96})} Let $E$ be an exceptional bundle over $\X$. Then $E^{\perp_\A}\simeq \mod \Lambda$ for some finite dimensional hereditary algebra $\Lambda$.
	
	\item \textup{(\cite{HL}; see also \cite[Proposition 2.14]{Lenzing})}  Let $L$ be a line bundle in $\coh\X$. Then \[L^{\perp_\A} \simeq \mod k[p_1,\dots, p_t],\] where $k[p_1,\dots,p_t]$ is the path algebra of the equioriented star quiver $[p_1,\dots,p_t]$.
	\end{enum2}
 \end{thm}
Here, 
an equioriented star quiver  $[p_1, \dots, p_t]$
refers to the quiver
 \[\xymatrix@R-25pt@C-10pt{
	 \overset{(1, p_1-1)}{\tinybullet}\ar[r] & \overset{(1, p_1-2)}{\tinybullet} \ar@{.}[r] & \overset{(1, 2)}{\tinybullet} \ar[r] & \overset{(1, 1)}{\tinybullet}\ar[ddr] & \\
	 \overset{(2, p_2-1)}{\tinybullet}\ar[r] & \overset{(2, p_2-2)}{\tinybullet} \ar@{.}[r] & \overset{(2,2)}{\tinybullet} \ar[r] & \overset{(2,1)}{\tinybullet}\ar[dr] & \\
	 \vdots & & \vdots & \vdots & \tinybullet \\
	 \underset{(t, p_t-1)}{\tinybullet}\ar[r] & \underset{(t, p_t-2)}{\tinybullet} \ar@{.}[r] & \underset{(t, 2)}{\tinybullet} \ar[r] & \underset{(t,1)}{\tinybullet}\ar[ur] &}\]

In certain cases,  forming a perpendicular category can yield the module category of a representation-finite finite dimensional  hereditary algebra.
 \begin{lem}\label{wpl perp}
	 \begin{enum2}
	 \item If $\X$ is of domestic type and $E$ is an indecomposable bundle then $E^{\perp_\A}$ is equivalent to $\mod \Lambda$ for a representation-finite finite dimensional  hereditary algebra $\Lambda$.

	 \item If $\X$ is of tubular type and $(E,F)$ is an exceptional pair in $\coh\X$ with $\mu(E)\neq \mu(F)$ then $\{E,F\}^{\perp_\A}$ is equivalent to $\mod \Lambda$ for a representation-finite  finite dimensional hereditary algebra $\Lambda$. 
	 \end{enum2}
 \end{lem}
 \begin{proof}
	 (1) Let $(p_1,p_2,p_3)$ be the weight sequence of $\X$. If some $p_i=1$, say $i=1$, then $E$ is a line bundle and by Theorem~\ref{line bundle perp}(2) we have $E^{\perp_\A}\simeq \mod k[p_2,p_3]$. Otherwise $p_i\geq 2$ for all $i$.  Up to the action of some power of $\tau$, we can suppose $\delta(\vec{\omega})<\mu(E)\leq 0$. Let $T$ be the direct sum of a complete set of indecomposable bundles with slope in the interval $(\delta(\vec{\omega}), 0]$ and suppose $T=T_1\oplus E$. Recall that $T$ is a tilting bundle and its endomorphism algebra $\Gamma=\End (T)$ is a tame hereditary algebra whose quiver has a bipartite orientation.  
	 Hence $\Gamma_1=\End(T_1)$ is a representation-finite hereditary algebra. We already know $E^{\perp_\A}\simeq \mod \Lambda$ for a finite dimensional hereditary algebra $\Lambda$. Now that $T_1$ is a tilting object in $E^{\perp_\D}$, we have exact equivalences $\D^b(\Lambda)\simeq \D^b(E^{\perp_\A})=E^{\perp_\D}\simeq \D^b(\Gamma_1)$.  Hence $\Lambda$ is a representation-finite hereditary algebra, the underlying graph of whose quiver is the same as that of the quiver of $\Gamma_1$. 

	 (2) 
By applying Lemma~\ref{extor perp}, we have an equivalence \[F^{\perp_\D}\simeq \Phi_{\infty,\mu(F)}(F)^{\perp_\D}\simeq \D^b(\X')\coprod \D^b(k\vec{\AA}_{l(F)-1}),\] under which  $E\in F^{\perp_\A}$ corresponds to $E'[m]$ for some exceptional bundle $E'$ over  $\X'$ and some $m\in\Z$. 
	 Thus there are exact equivalences \[\D^b(\{E,F\}^{\perp_\A})=\{E,F\}^{\perp_\D}\simeq E'^{\perp_{\D^b(\coh\X')}}\coprod \D^b( k\vec{\AA}_{l(F)-1})\simeq \D^b(\Gamma)\] for a representation-finite  finite dimensional hereditary algebra $\Gamma$. It follows that $\{E,F\}^{\perp_\A}$ is equivalent to $\mod \Lambda$ for a representation-finite  finite dimensional hereditary algebra $\Lambda$.
 \end{proof}

\begin{rmk} 
	\begin{enum2}
	\item There is a more direct proof of (1) using Theorem~\ref{bundle hom nonzero}. The current proof has the advantage that it gives us  additional information on the quiver of $\Lambda$.
		
	\item It can be shown that if $\X$ is of tubular type and $E$ is an exceptional bundle with quasi-length $l$ then $E^{\perp_\A}\simeq \mod \Lambda\coprod \mod k\vec{\AA}_{l-1}$ for a tame hereditary algebra $\Lambda$ and an equioriented $\AA_{l-1}$-quiver. 
	\end{enum2}
\end{rmk}

 \subsection{Some nonvanishing Hom spaces}\label{nonvanish hom}
The following two lemmas are well-known.

\begin{lem}\label{bundle hom subtube-simples} \label{bundle hom torsion}
	Let $E$ be a nonzero bundle over $\X$ and $F$ an non-exceptional indecomposable torsion sheaf. Then $\hom(E,F)\neq 0$, $\ext^1(F,E)\neq 0$. 
\end{lem}

\begin{proof}
	 Suppose $F$ is supported at $\lambda\in\P^1$.
	  Take a line bundle $L$ such that there is an epimorphism $E\epic L$ and also a simple sheaf $S$ supported at $\lambda$ such that $\hom(L, S)\neq 0$. Then $\hom(E,S)\neq 0$. Since $F$ is  a non-exceptional indecomposable sheaf supported at $\lambda$,   $S$ is a composition factor of $F$. 
	  Then there exist two exact sequences \[0\ra F_1\ra F\ra F_2\ra 0,\quad 0\ra S\ra F_2\ra F_3\ra 0,\] where $F_i\in \coh_\lambda\X$ ($i=1,2,3$).
	   Applying $\hom(E,-)$, one has $\hom(E,S)\monic \hom(E,F_2)$ and $\hom(E,F)\epic \hom(E,F_2)$ therefore $\hom(E,F)\neq 0$. 
	   Note that $\tau F$ is also a non-exceptional indecomposable sheaf and thus $\hom(E, \tau F)\neq 0$. This gives $\ext^1(F,E)\neq 0$ by Serre duality. 
\end{proof}

\begin{lem}\label{tubular bundle hom}
	Let $\X$ be of tubular type. Suppose $E,F$ are two nonzero bundles with $\mu(E)<\mu(F)$. Then $\hom(E,\tau^i F)\neq 0$ for some $i$. If $E$ or $F$ is a non-exceptional indecomposable bundle, $\hom(E,F)\neq 0$ always holds.
\end{lem}

\begin{proof}
	By Riemann-Roch Theorem, we have 
	\[\sum_{j=0}^{p-1}(\dim_k\hom(\tau^jE,F)-\dim_k\ext^1(\tau^jE,F)) =\bar{\chi}(E,F)=\rk(E)\rk(F)(\mu(F)-\mu(E))>0.\]
Since $\ext^1(\tau^j E,F)=0$ for each $j$, $\hom(\tau^m E, F)\neq 0$ for some $0\leq m<p$, whereby $\hom(E, \tau^i F)\neq 0$ for some $i$. If $E$ is non-exceptional indecomposable bundle then $E$ has a filtration with factors $\tau^i G$ ($0\leq i<p_E$), where $G$ is the quasi-top of $E$ and $p_E$ is the $\tau$-period of $E$. Now that $\hom(\tau^i G, F)\neq 0$ for some $i$, $\hom(E,F)\neq 0$. Similar argument applies to the case when $F$ is a non-exceptional indecomposable bundle.   
\end{proof}

Using stability argument, \cite{LP} showed the following  fact. 
\begin{thm}[{\cite[Theorem 2.7]{LP}}]\label{bundle hom nonzero}
	Let $F,G$ be nonzero bundles on $\X$ with $\mu(G)-\mu(F)>\delta(\vec{c}+\vec{\omega})=p+\delta(\vec{\omega})$ then $\hom(F,G)\neq 0$.
\end{thm}

For $E[n]\in \D^b(\X)$ ($E\in \coh\X$), we defined the slope of $E[n]$ by $\mu(E[n])=\mu(E)$. 
For  a nonzero subcategory $\CC$ of $\D$ closed under nonzero direct summands, define  
\begin{equation}\label{def of mu}
	\mu(\CC)  =\{\mu(E)\mid E\,\,\text{an indecomposable object in}\,\, \CC\}.
\end{equation} 
We emphasize that we only count in indecomposables. We will talk about limit points of subsets of $\mu(\CC)$. In doing so, we will deem $\mu(\CC)$ as a subspace of $\bar{\R}$, where $\bar{\R}$ is equipped with the topology obtained via one point compactification of $\R$. 

If $\X$ is of  tubular type, by Lemma~\ref{fractional linear}, for each $q\in\bar{\Q}$, 
 there is  a fractional linear function $\phi_q$ on $\bar{\R}$ with integer coefficients such that  $\mu(\Phi_{\infty, q}(E))=\phi_q(\mu(E))$, where $\Phi_{\infty, q}$ is a telescopic functor. Evidently, $\phi_q$ is a homeomorphism of $\bar{\R}$ and restricts to a homeomorphism of the subspace $\bar{\Q}$.

 \begin{lem}\label{tubular limit point}
	 Suppose $\X$ is of tubular type and let $E$  be an exceptional sheaf over $\X$. Then $\mu(E)$ is the unique limit point of $\mu (E^{\perp_\A})$ (and $\mu ({}^{\perp_\A} E)$).
	\end{lem}
	
\begin{proof}
	First suppose that $E$ is an exceptional torsion sheaf. By Lemma~\ref{extor perp} (and with the notation there), we have  \[E^{\perp_\A}=\SSS_E^{\perp_\A}\coprod \pair{\SSS_E'}_\A \simeq \coh\X'\coprod \mod k\vec{\AA}_{l(E)-1},\] where $\X'$ is a weighted projective line of domestic type, and the equivalence $\SSS_E^{\perp_\A}\simeq \coh\X'$ preserves rank. By Theorem~\ref{domestic bundle}, the rank function $\rk$  is bounded on indecomposable sheaves in $E^{\perp_\A}$. Moreover, $L(n\vec{c})\in E^{\perp_\A}$ for a line bundle $L\in E^{\perp_\A}$ and $n\in\Z$. Thus $\infty$ is the unique limit point of $\mu (E^{\perp_\A})$.

	Now consider  an exceptional bundle $E$ with slope $q$.  Since $\Phi_{\infty,q}(E)$ is an exceptional torsion sheaf, $\infty$ is the unique limit point of $\mu(\Phi_{\infty,q}(E)^{\perp_\A})$. Now that \[\mu(E^{\perp_\A})=\mu(E^{\perp_\D})=\phi_q^{-1}(\mu(\Phi_{\infty,q}(E)^{\perp_\D}))=\phi_q^{-1}(\mu(\Phi_{\infty,q}(E)^{\perp_\A})),\]  $q=\phi_q^{-1}(\infty)$ is the unique limit point of $\mu(E^{\perp_\A})$. 

Recall that ${}^{\perp \A} E=(\tau^{-1} E)^{\perp_\A}$. Hence $\mu(E)=\mu (\tau^{-1} E)$ is the unique limit point of $\mu({}^{\perp_\A}E)=\mu((\tau^{-1} E)^{\perp_\A})$.
\end{proof}

\begin{cor}\label{tubular nonvanish hom}
	Suppose $\X$ is of tubular type. Let $E$ be an indecomposable sheaf and $\E=\{E_i\mid i\in I\}$ a collection of indecomposable sheaves with $\mu(\E)$ a bounded subset of $\R$.  Suppose $\mu$ is a limit point of $\mu(\E)$. If $\mu<\mu(E)$ then there is some $E_i$ with $\hom(E_i,E)\neq 0$; if $\mu>\mu(E)$ then there is some $E_i$ with $\hom(E,E_i)\neq 0$.
\end{cor}
\begin{proof}
	We will consider the case $\mu<\mu(E)$ and the other case is similar. If $E$ is non-exceptional then our assertion follows from Lemma~\ref{bundle hom torsion} and Lemma~\ref{tubular bundle hom}. So  we consider exceptional $E$. We can assume that $\mu(E_i)<\mu(E)$ for all $i$ by dropping the other $E_i$'s. Then $\ext^1(E_i,E)=0$ for all $i$. If $\hom(E_i,E)=0$ for all $i$ then $E_i\in {}^{\perp_\A} E$ for all $i$ and thus $\mu$ is limit point of $\mu({}^{\perp_\A}E)$. This is a contradiction to Lemma~\ref{tubular limit point}. Thus we have 
	$\hom(E_i,E)\neq 0$ for some $i$.
\end{proof}

\subsection{Full exceptional sequences in $\coh\X$}
It's well-known that if a $k$-linear  essentially small triangulated category  $\D$ of finite type contains an exceptional sequence of length $n$ then the rank $\rk K_0(\D)$ of the Grothendieck group $K_0(\D)$ of $\D$ satisfies $\rk K_0(\D)\geq n$. In general, the exceptional sequence is not full even if $n=\rk K_0(\D)$. But this is the case in our setup.

\begin{lem}\label{full exseq rk}
	An exceptional sequence $(E_1,\dots, E_n)$ in $\D^b(\X)$ is full iff $n=\rk K_0(\X)$.
\end{lem}
\begin{proof}
We always have $n\leq \rk K_0(\D^b(\X))=\rk K_0(\X)$. \cite[Lemma 4.1.2]{Meltzer} showed that an exceptional sequence in $\D^b(\X)$ of length $\rk K_0(\X)$  generates $\D^b(\X)$.  So an exceptional sequence $(E_1,\dots, E_n)$ in $\D^b(\X)$ is  full iff $n=\rk K_0(\X)$.
\end{proof}

Observe that by Serre duality,  if $(E_1,\dots,E_n)$ is a full exceptional sequence in $\coh\X$ then \[(\tau E_{i+1},\dots, \tau E_n,E_1,\dots, E_i)\] is also a full exceptional sequence. 
We show that a full exceptional sequence in $\coh\X$ can possess certain nice term. 
\begin{lem}\label{exseq simple}
	If a full exceptional sequence  in $\coh\X$ contains a torsion sheaf then it contains a simple sheaf.
\end{lem}
\begin{proof}
	Let $(E_1,\dots,E_n)$ be a full exceptional sequence with $E_i$ a torsion sheaf. We can suppose $i=n$. Note that $(E_1,\dots, E_{n-1})$ is a full exceptional sequence in $E_n^{\perp_\A}$. If $E_n$ is already simple then there is nothing to prove. Suppose $l(E_n)>1$. Then by Lemma~\ref{extor perp}, we have an equivalence
	\begin{equation}\label{torsionsimple}
		E_n^{\perp_\A}\simeq \coh\X'\coprod \mod k\vec{\AA}_{l(E_n)-1}
	\end{equation}
	for some weighted projective line $\X'$ and an equioriented $\AA_{l(E_n)-1}$-quiver. Via this equivalence, a subsequence of $(E_1,\dots, E_{n-1})$ yields a full exceptional sequence in $\mod k\vec{\AA}_{l(E_n)-1}$, which contains a simple module by Corollary~\ref{A_s exseq}. 
	Note that a simple $k\vec{\AA}_{l(E_n)-1}$-module maps to a simple sheaf under the equivalence \eqref{torsionsimple}, which is clear from Lemma~\ref{extor perp}. So some $E_i$ is a simple sheaf.
\end{proof}

\begin{prop}\label{domestic exceptional seq}\label{domestic exseq}
	For $\X$ of domestic type, each full exceptional sequence  in $\coh\X$ contains a line bundle.
	\end{prop}
	\begin{proof}
		Let  $(E_1,\dots,E_n)$ be a full exceptional sequence in $\coh\X$. 
		We use induction to show our assertion. Consider the weight type $(1,p_1,p_2)$, in which case each indecomposable bundle over $\X$ is a line bundle. Since $(E_1,\dots, E_n)$ classically generates $\D^b(\X)$, some $E_i$ is an indecomposable bundle and thus a line bundle. We continue to consider a domestic weight type different than $(1,p_1,p_2)$ even up to permutation. 
		We claim that if each $E_i$ is a bundle then the assertion holds, which is proved later. So consider the case that some $E_i$ is a torsion sheaf.  We can assume that $i=n$.  
		Moreover, $(E_1,\dots, E_{n-1})$ is a full exceptional sequence in $E_n^{\perp_{\A}}$. By Lemma~\ref{extor perp} (and with the notation there), we have \[E_n^{\perp_{\A}}=\SSS_{E_n}^{\perp_\A}\coprod \pair{\SSS_{E_n}'}\simeq \coh\X'\coprod \mod k\vec{\AA}_{l(E_n)-1},\] where $\X'$ is a weighted projective line with a weight function dominated by the weight function of $\X$ (in the sense of \cite{GL2}),  and the equivalence $\SSS_{E_n}^{\perp_\A}\simeq \coh\X'$ preserve rank. By induction, we know that some $E_i$ ($i\in \{1,\dots, n-1\}$) is a line bundle.

	It remains to prove our claim that if each $E_i$ is a bundle then some $E_i$ is a line bundle. 
		The proof is inspired by the proof of \cite[Proposition 4.3.6]{Meltzer}. As in \cite[\S 4.3.6]{Meltzer}, for an exceptional sequence  
		$\underline{E}=(E_1,\dots, E_n)$, define \[\|\underline{E}\|=(\rk( E_{\pi(1)}), \dots, \rk( E_{\pi(n)})),\] where $\pi$ is a permutation on $\{1,\dots,n\}$ such that $\rk( E_{\pi(1)})\geq \dots \geq \rk( E_{\pi(n)})$.  
		
		Suppose for a contradiction that $\rk(E_i)\geq 2$ for each $i$. In particular, $\oplus E_i$ is not a tilting bundle since each tilting bundle contains a line bundle summand for $\X$ of domestic type by \cite[Corollary 3.7]{Lenzing} (reproved with Corollary~\ref{tilting bundle}(1)). Hence for some $i<j$, $\ext^1(E_i, E_j)\neq 0$. 
		We can assume that $\ext^1(E_k, E_l)=0$ for $i\leq k<l\leq j$. By \cite[Lemma 3.2.4]{Meltzer}, $\hom(E_i, E_j)=0$. 
		
		Consider $i<k<j$ such that $\hom(E_i,E_k)\neq 0$. Let $f: E_i\ra E_k$ be a nonzero morphism, which is either a monomorphism or an epimorphism by Happel-Ringel Lemma (see Proposition~\ref{Happel-Ringel lemma}). $f$ being a monomorphism implies 
		\[0=\ext^1(E_k,E_j)\epic \ext^1(E_i,E_j)\neq 0,\] a contradiction. Hence $f$ is an epimorphism. Thus $\hom(E_i,E_j)=0$ implies $\hom(E_k,E_j)=0$.  
		
		Let $P$ be the subsequence of $(E_{i+1},\dots, E_{j-1})$ consisting of those $E_k$ satisfying $\hom(E_i,E_k)\neq 0$. Then for each term $E_k$ in $P$, we have an epimorphism in $\hom(E_i,E_k)$ and $\hom(E_k,E_j)=0$. Let  $Q$ be the subsequence of $(E_{i+1},\dots, E_{j-1})$ consisting of the remaining terms, i.e., those $E_l$ satisfying $\hom(E_i,E_l)=0.$ 
We want to show that $\hom(E_k,E_l)=0$ for $E_k\in P, E_l\in Q$.  
Each nonzero morphism $g: E_k\ra E_l$ is either a monomorphism or an epimorphism by Happel-Ringel Lemma. 
If $g$ is a monomorphism then  $\hom(E_i,E_k)\neq 0$ implies $\hom(E_i, E_l)\neq 0$, a contradiction to $\hom(E_i,E_l)=0$;  
if  $g$ is an epimorphism then composing with an epimorphism in $\hom(E_i, E_k)$ yields  an epimorphism in $\hom(E_i,E_l)$, again a contradiction to $\hom(E_i,E_l)=0$.  
These show that $\hom(E_k,E_l)=0$ for $E_k\in P, E_l\in Q$. 
Moreover, $\hom(E_k,E_j)=0$ for $E_k\in P$. Therefore the sequence \[(E_1,\dots, E_{i-1}, Q, E_i,E_j, P, E_{j+1}, \dots, E_n)\] 
is a full exceptional sequence. This gives us a full exceptional sequence $(F_1,F_2,\dots, F_n)$ with $\rk(F_i)\geq 2$, $\ext^1(F_1,F_2)\neq 0$ and $\hom(F_1,F_2)=0$. 

Now we use mutation of an exceptional sequence. Let $L_{F_1}F_2$ be the universal extension: \[0\ra F_2\ra L_{F_1}F_2\ra \ext^1(F_1,F_2)\otimes F_1\ra 0.\] Then \[\underline{F'}=(L_{F_1}F_2, F_1, F_3,\dots, F_n)\] is a full exceptional sequence with $\|\underline{F'}\|>\|\underline{F}\|$. As before,  since each bundle in the sequence has rank $\geq 2$, the direct sum of bundles in $\underline{F'}$ is not a tilting bundle. 
		This allows us to repeat the argument above. Successive repeating will give us indecomposable bundles with arbitrary large rank.  
	This is  a contradiction to the fact that the rank function is bounded on indecomposable bundles over a weighted projective line of domestic type. We have thus shown our claim that each full exceptional sequence $(E_1,\dots, E_n)$ with each $E_i$ a bundle indeed contains a line bundle. 
\end{proof}

\begin{cor}\label{tubular exseq}
	Suppose  $\X$ is of tubular type. If a full exceptional sequence  in $\coh\X$ contains a torsion sheaf then it contains a line bundle and a simple sheaf.
\end{cor}
\begin{proof}
	Let $(E_1,\dots,E_n)$ be a full exceptional sequence in $\coh\X$. By Lemma~\ref{exseq simple}, if some $E_i$ is torsion then some $E_j$ is simple. Suppose $j=n$. Since $(E_1,\dots, E_{n-1})$ is a full exceptional sequence in $E_n^{\perp_\A}\simeq \coh\X'$, where $\X'$ is a weighted projective line of domestic type and the equivalence preserves rank, it follows from Proposition~\ref{domestic exseq} that some $E_k$  is a line bundle.
\end{proof}

\subsection{Torsion pairs in $\coh\X$}\label{sec: torsion pair}
In this subsection, we discuss some properties of torsion pairs in $\coh\X$ and also give some preparatory descriptions of  torsion pairs (see \S\ref{sec: torsion final} for the final description).
 We first describe two simple classes of torsion pairs in $\coh\X$. Obviously, any  torsion pair in $\coh\X$ restricts to a torsion pair in $\coh_\lambda\X$ for each $\lambda\in\P^1$.

 \begin{lem}\label{torsion pair contained in torsion part} Let $(\T, \F)$  be  a pair of subcategories of $\coh\X$. 

	 \begin{enum2}
	 \item  $(\T,\F)$ is a torsion pair in $\coh\X$ with $\T\subset \coh_0\X$ iff for each $\lambda \in \P^1$, there is a  torsion pair  $(\T_\lambda, \F_\lambda)$ in  $\coh_\lambda\X$ such that \[\T=\add\{\T_\lambda \mid \lambda \in \P^1\},\quad \F=\add\{\vect\X, \F_\lambda \mid \lambda \in \P^1\}.\]

	 \item $(\T,\F)$ is a torsion pair in $\coh\X$ with $\F\subset \coh_0\X$ iff  \[\F=\add\{\F_\lambda\mid \lambda\in \P^1\},\quad \T=\{E\in\coh\X\mid \hom(E,\F)=0\},\] where each $\F_\lambda$ is a torsion-free class in $\coh_\lambda\X$ without non-exceptional indecomposable object.
	 \end{enum2}
\end{lem}
\begin{proof}
	We prove (2) as  (1) is clear.

	\nec Suppose $\F\subset \coh_0\X$. $\F$ restricts to a torsion-free class $\F_\lambda$ in $\coh_\lambda\X$ for each $\lambda\in\P^1$. If $\F_\lambda$ contains a non-exceptional indecomposable sheaf then by Lemma~\ref{bundle hom torsion}, $\T$ contains no nonzero bundle and thus $\vect\X\subset \F$, a contradiction. Hence each $\F_\lambda$ contains no non-exceptional indecomposable sheaf. 

	\suf By the definition of $\T$, $\T$ is closed under quotient and extension. Therefore $\T$ is a torsion class in $\coh\X$ since $\coh\X$ is noetherian.  
Then $(\T,\T^{\perp_{0}})$ is a torsion pair in $\coh\X$ and thus we need to show  $\F=\T^{\perp_{0}}$. $\hom(\T,\F)=0$ implies $\F\subset \T^{\perp_{0}}$ and it remains to  show $\T^{\perp_0}\subset \F$. 
For each $\lambda\in \P^1$,  $\T\cap \coh_\lambda\X={}^{\perp_{0,\coh_\lambda\X}}\F_\lambda$ is the torsion class in $\coh_\lambda\X$ corresponding to the torsion-free class $\F_\lambda$, which implies $\T^{\perp_0}\cap \coh_\lambda\X\subset \F_\lambda$. Hence $\T^{\perp_0}\cap \coh_0\X\subset \F$. We claim that $\T^{\perp_0}$ contains no nonzero bundle, which implies $\T^{\perp_0}\subset \F$. 
	Suppose for a contradiction that $\T^{\perp_0}$ contains a nonzero bundle $E$. For each $\lambda\in \P^1$, by Lemma~\ref{order sequence orthogonal}, it is impossible that each simple sheaf in $\coh_\lambda\X$ occurs as a composition factor of some indecomposable sheaf in $\F_\lambda$. Hence we have a line bundle $L$ such that $L(n\vec{c})\in\T$ for all $n\in\Z$.  
	But $\hom(L(n\vec{c}), E)\neq 0 $ for $n\ll 0$, a contradiction. This shows our claim.
\end{proof}
\begin{rmk}
	For an ordinary point $\lambda$, either $\T_\lambda= 0$ or $\F_\lambda= 0$.
\end{rmk}

Recall that for each $\mu \in\bar{\R}$, we have torsion pairs \[(\coh^{\geq \mu}\X, \coh^{<\mu}\X),\quad (\coh^{>\mu}\X, \coh^{\leq \mu}\X).\] These are very useful for our analysis. 

A torsion pair in $\coh\X$ is either tilting or cotilting.

\begin{lem}\label{cotilting torsion theory}
	Let $(\T,\F)$ be a torsion pair in $\coh\X$. 

	\begin{enum2}
	\item If $\F$ contains  a nonzero bundle then $\F$ is a cotilting torsion-free class and $\coh^{\leq \mu}\X\subset \F$  for some $\mu\in \R$.
	
	\item If $\T$ contains a nonzero bundle then $\T$ is a tilting torsion class. If $\coh_0\X\subsetneq \T$ then  $\coh^{\geq \nu}\X\subset \T$ for some $\nu\in\R$. 
	\end{enum2}
\end{lem}

\begin{proof}

	 Suppose that $\F$ contains a nonzero  bundle  $A$.  If $\T$ contains no nonzero bundle, then $\vect\X\subset \F$. Now suppose that $\T$ contains a nonzero bundle $T$.  
	 Let $\mu=\mu(A)-\delta(\vec{c}+\vec{\omega})$. Then for each  bundle $B\in \T$, we have $\mu (B)>\mu$. Indeed, if $\mu(B)\leq \mu$ then $\mu(A)-\mu (B)\geq \delta(\vec{c}+\vec{\omega})$ and $\hom(B,A)\neq 0$ by Theorem~\ref{bundle hom nonzero}, a contradiction to $\hom(\T,\F)=0$. 
	 Since $\T$ is closed under quotient, for each nonzero bundle $E$ in $\T$, the last semistable factor of $E$ lies in $\T$ and hence $\mu^-(E)>\mu$. This shows $\vect\X\cap \T\subset \coh^{>\mu}\X$. Recall that a coherent sheaf over $\X$ decomposes as a direct sum of a torsion sheaf and a vector bundle. So we have $\T\subset \coh^{>\mu}\X$ and thus $\coh^{\leq \mu}\X\subset \F$. Similarly one shows that if $\T$ contains a nonzero bundle then $\vect\X\cap \F\subset \coh^{<\nu}\X$ for some $\nu\in\R$, which implies $\coh^{\geq \nu}\X\subset \T$ provided $\coh_0\X\subsetneq \T$. 

	 Now we  show that $\F$ is a cotilting torsion-free class  if $\F$ contains a nonzero bundle. That is, we need to show that for each  sheaf $E$, there is some sheaf $F\in\F$ and an epimorphism $F\epic E$. We do induction on $\rk(E)$. We already have $\coh^{\leq \mu}\X\subset \F$ for some $\mu\in \R$. If  $E$ is an indecomposable torsion sheaf then we can take a line bundle $L\in\F$ such that $L\epic E$. If $\rk(E)>0$, take a line bundle $L_1\in \F$ with $\mu(L)\ll \mu (E)$.  Then we have an exact sequence $0\ra L_1\ra E\ra E_1\ra 0$ with $\rk(E_1)< \rk(E)$. By the induction hypothesis, there is some $F_1\in F$ and an epimorphism $F_1\epic E_1$. The pullback diagram 
	\[\xymatrix@R-5pt{0 \ar[r] & L_1 \ar[r] \ar@{=}[d] & F\ar[r] \ar[d] & F_1\ar[r] \ar[d] & 0\\
0\ar[r] & L_1\ar[r] & E\ar[r] \ar[r] & E_1\ar[r] & 0}\]
gives us an object $F\in \F$ and an epimorphism $F\epic E$, as desired.

If $\T$ contains a nonzero bundle, we show that $\T$ is a tilting torsion class.  For each $\lambda\in \P^1$, consider the torsion pair $(\T_\lambda,\F_\lambda)=(\T\cap \coh_\lambda\X, \F\cap \coh_\lambda\X)$ in $\coh_\lambda\X$. 
By Lemma~\ref{bundle hom subtube-simples},  $\F_\lambda$ contains no non-exceptional object and thus $\T_\lambda$ contains a non-exceptional object. Then $S\in \T$ for a simple sheaf $S$ supported at an ordinary point. Moreover,  $\T_\lambda$ is a tilting torsion class in $\coh_\lambda\X$ by Lemma~\ref{A_t torsion pair}. Hence each indecomposable torsion sheaf in $\coh_\lambda\X$ is a subobject of some object in $\T_\lambda$.  
Since $\T$ is closed under quotient, $\T$ contains a line bundle $L$ by Proposition~\ref{line bundle}(2). $L, S\in \T$ implies $L(n\vec{c})\in \T$ for $n\geq 0$. 
By \cite[Corollary 2.7]{GL}, for each $E\in \vect\X$, $E$ is a subbundle of $\oplus_{i=1}^mL_i$ for some line bundles $L_1,\dots,L_m$. Now that $L_i$ is a subbundle of $L(n\vec{c})$ for $n\gg 0$, $E$ is a subbundle of $\oplus_{i=1}^mL(n\vec{c})\in \T$.   This shows that $\T$ is a tilting torsion class if $\T$ contains a nonzero bundle.
\end{proof}

\begin{lem}\label{no quasi-simple}
	Let $(\T,\F)$ be a torsion pair in $\coh\X$ with $\coh_0\X\subsetneq \T\subsetneq \coh\X$. 

	\begin{enum2}
	\item If $\X$ is of domestic type then the $\tau$-orbit of each line bundle contains some line bundle $L$ such that $L\in \T$ and $\tau L\in\F$.

	\item If $\X$ is of tubular type then exactly one of the following holds:
\begin{enumerate}[label={$(\alph*)$},leftmargin=0.6cm]
	\item there exists some quasi-simple bundle $E$ in $\T$ with $\tau E\in\F$;
	\item for some $\mu \in\R\backslash \Q$, $(\T,\F)=(\coh^{>\mu}\X, \coh^{<\mu}\X)$;
	\item for some $\mu\in \Q$ and some $P\subset \P^1$, \[(\T,\F)=(\add\{\coh^{>\mu}\X, \coh^{\mu}_\lambda\X\mid \lambda\in P\},\,\,\add\{\coh^\mu_\lambda\X,   \coh^{<\mu}\X\mid \lambda\notin P\}).\]
	\end{enumerate}
\end{enum2}
\end{lem}

\begin{proof}
	Note that $\coh_0\X\subsetneq \T\subsetneq \coh\X$ implies $\{0\}\subsetneq \F\subsetneq \vect\X$. By Lemma~\ref{cotilting torsion theory}, $\coh^{\leq \mu_0}\X\subset \F$ for some $\mu_0\in\R$ and $\coh^{\geq \nu_0}\X\subset \T$ for some $\nu_0\in\R$.

	(1)	By Lemma~\ref{twist slope}, $\mu(\tau^n L)=\mu(L)+n\delta(\vec{\omega})$. Since $\delta(\vec{\omega})<0$, for each line bundle $L$, $\tau^{n}L\in \F$ for $n\gg 0$  and $\tau^{n}L\in \T$ for $n\ll 0$. Moreover $\coh_0\X\subsetneq \T$ implies that each line bundle lies in $\T$ or $\F$ and therefore there must be a line bundle $\tau^n L\in \T$ with $\tau^{n+1} L\in \F$.

	(2) Obviously, the three types are disjoint. If $(\T,\F)$ is not of type $(a)$ then $\tau E\in \T$ for each quasi-simple $E\in \T$. For $\mu\in \Q$ and $\lambda\in \P^1$, let $X$ be an indecomposable bundle in $\coh_\lambda^\mu\X \cap \T$. Take the quasi-top $Y$ of $X$. 
Then $\tau$-orbit of $Y$ lies in $\T$, which implies $\coh_\lambda^\mu\X\subset \T$. Hence if $\T\cap \coh_\lambda^\mu\X\neq 0$ then $\coh_\lambda^\mu\X\subset \T$. Since $\coh^{\leq \mu_0}\X\subset \F$, we have $\T\subset \coh^{>\mu_0}\X$. Thus we can take $\mu_1=\inf\{T\in\T\mid \mu^-T\}\in \R$.
$\T\subset \coh^{\geq \mu_1}\X$ implies $\coh^{<\mu_1}\X\subset \F$. Let $E$ be any indecomposable bundle with $\mu(E)>\mu_1$ and take an indecomposable bundle  $T\in \T$ with $\mu_1\leq \mu(T)<\mu(E)$.  
 Then  
 Lemma~\ref{tubular bundle hom} implies that $\hom(\tau^j T,E)\neq 0$ for some $j$. Since $\tau^j T\in\T$,  $E\notin \F$. This shows that $\mu^+F\leq \mu_1$ for $F\in \F$. 
 Thus $\F\subset \coh^{\leq \mu_1}\X$ and $\coh^{>\mu_1} \X\subset \T$.  If $\mu_1\in \R\backslash \Q$ then $\T=\coh^{>\mu_1}\X$ and $\F=\coh^{<\mu_1}\X$. If $\mu_1\in \Q$ then $(\coh^{\mu_1}_\lambda \X\cap \T, \coh^{\mu_1}_\lambda \X\cap \F)$ is a torsion pair in $\coh^{\mu_1}_\lambda\X$. We already know that for $\lambda\in \P^1$, if $\T\cap \coh^{\mu_1}_\lambda\X\neq 0$ then $\coh_\lambda^{\mu_1}\X\subset \T$ and hence either $\coh_\lambda^{\mu_1}\X\subset \T$ or $\coh_\lambda^{\mu_1}\X\subset\F$. Consequently, for some $P\subset \P^1$, 
 \[(\T, \F) = (\add\{\coh^{>\mu_1}\X, \coh^{\mu_1}_\lambda\X \mid \lambda\in P\},\,\, \add\{\coh^{\mu_1}_\lambda\X,  \coh^{<\mu_1}\X\mid \lambda\in \P^1\backslash P\}).\]
\end{proof}

We establish bijective correspondences between tilting sheaves, certain bounded t-structures on $\D^b(\X)$ and certain torsion pairs in $\coh\X$.

\begin{prop}\label{length torsion pair}
	 Denote $\A=\coh\X$. There are bijective correspondences between 
	 \begin{enum2}	
	 \item torsion pairs $ (\T,\F)$ in $\A$ such that the tilted heart $\F[1]*\T$ is a length category;

	 \item bounded t-structures whose heart is a length category contained in $\A[1]*\A$;

	 \item  isomorphism classes of basic tilting sheaves in $\A$;

	 \item torsion pairs $(\T,\F)$ such that there is $n=\rk K_0(\X)$ pairwise non-isomorphic indecomposable sheaves $E_1,\dots, E_n$ in $\T$ with $\tau E_i\in \F$ for all $i$.
	 \end{enum2}
	Moreover, torsion pairs $(\T, \F)$ in $(1)$ with the additional assumption $\coh_0\X\subsetneq \T\subsetneq \coh\X$ are in bijection with isoclasses of basic tilting bundles.
\end{prop}
\begin{proof}
	The second assertion follows readily from the first one. We show the first assertion. 
	The bijection between (1) and (2) follows from Proposition~\ref{torsion pair t-structure}.  
	Note that for those $E_i$'s in (4), we have $\hom(\oplus E_i, \oplus \tau E_i)=0$. By Serre duality, we have $\ext^1(\oplus E_i, \oplus E_i)=0$. Thus $E_i$'s can be ordered to be a full exceptional sequence by Proposition~\ref{hereditary ext vanish} and Lemma~\ref{full exseq rk}. 
 So $\oplus E_i$ is a tilting sheaf. Then the obvious associations between (3) and (4) are evidently inverse to each other. 
	
	Now we establish the bijection between (2) and (3). 
	By Theorem~\ref{der canonical algebra}, $\A=\coh\X$ is derived equivalent to $\mod \Lambda$ for a canonical algebra $\Lambda$. Hence we can apply Theorem~\ref{silting t-str} to conclude   that bounded t-structures on $\D^b(\A)$ with length heart  are in bijection with equivalence classes of silting objects in $\D^b(\X)$. 
	Note that if a bounded t-structure $(\D^{\leq 0}, \D^{\geq 0})$ has heart $\B\subset  \A[1]*\A$ then $\D_\A^{\leq -1}\subset \D^{\leq 0}\subset \D_\A^{\leq 0}$ and thus the Serre functor $\tau(-)[1]$ of $\D^b(\A)$ is right t-exact with respect to $(\D^{\leq0 }, \D^{\geq 0})$. 
	By Lemma~\ref{silting tilting}, in this bijection, a bounded t-structure with length heart $\B\subset \A[1]*\A$ corresponds to some equivalence class of tilting objects in $\D^b(\X)$. It remains to show that such a tilting object $T$ is a sheaf. 
	By Lemma~\ref{silting tilting}, $T, \tau T[1]\in \B\subset \A[1]*\A$.  This forces $T$ to be a sheaf. 
\end{proof}
\begin{rmk}
	Recall that we have a torsion pair $(\T,\F)$ induced by a tilting sheaf $T$, where  \[\T=\{E\in \coh\X \mid \ext^1(T,E)=0\}, \quad \F=\{E\in \coh\X\mid \hom(T,E)=0\}.\] Since $T\in \T, \tau T\in \F$, this torsion pair is just the one corresponding to $T$.  

\end{rmk}

\begin{exm}
	Consider the torsion pair $(\T,\F)=(\coh^{\geq \mu}\X, \coh^{<\mu}\X)$ for $\mu\in \R$. If $\X$ is of domestic type, 
		similar argument to that in the proof of \cite[Theorem 3.5]{Lenzing} shows that the direct sum of  a  complete set of indecomposable bundles with slope in the interval $[\mu, \mu-\delta(\vec{\omega}))$ is a tilting bundle, whose endomorphism algebra turns out to be a tame hereditary algebra. 
	The induced torsion pair is exactly $(\coh^{\geq \mu}\X, \coh^{<\mu}\X)$. 
	If $\X$ is not of domestic type then $\T$ (resp. $\F$) is closed under $\tau$ (resp. $\tau^{-1}$) since $\delta(\vec{\omega})\geq 0$. Therefore $(\coh^{>\mu}\X, \coh^{\leq\mu}\X)$ cannot be induced by a tilting sheaf and the tilted heart $\coh^{\leq\mu}\X[1]*\coh^{> \mu}\X$ is not a length category. 
\end{exm}

We obtain the following known results as a corollary of Proposition~\ref{length torsion pair}.
\begin{cor}\label{tilting bundle}
	\begin{enum2}
	\item \textup{(\cite[Corollary 3.7]{Lenzing}).} If $\X$ is of domestic type then each tilting bundle $T$ contains at least $[L(\udp):\Z\vec{\omega}]$ pairwise nonisomorphic line bundles  as its direct summands. 

	\item \textup{({\cite[Corollary 3.5]{LM}}).} If $\X$ is of tubular type then each tilting bundle $T$ contains a quasi-simple bundle direct summand. For some $q\in\bar{\Q}$, $\Phi_{\infty, q}(T)$ is a tilting sheaf with an exceptional simple sheaf as its direct summand.
	\end{enum2}
\end{cor}
\begin{proof}
Let $(\T, \F)$ be the torsion pair corresponding to $T$. Since $T$ is a bundle, $\coh_0\X\subsetneq \T\subsetneq \coh\X$.

(1)  By Lemma~\ref{no quasi-simple}, each $\tau$-orbit of a line bundle contains a line bundle $L\in \T$ with $\tau L\in F$. Each such $L$ is a direct summand of $T$. By Proposition~\ref{line bundle}, we have precisely  $[L(\udp): \Z\vec{\omega}]$ $\tau$-orbits of line bundles. So $T$ contains at least $[L(\udp): \Z\vec{\omega}]$ pairwise nonisomorphic line bundles.

	(2) Note that in Lemma~\ref{no quasi-simple}, a torsion pair $(\U, \V)$ in $\coh\X$ of type  ~\ref{no quasi-simple}$(2b)$ or ~\ref{no quasi-simple}$(2c)$ contains no nonzero sheaf $F$ with $F\in \U$ and $\tau F\in\V$. So $(\T, \F)$ is of type ~\ref{no quasi-simple}$(2a)$, i.e., there exists a quasi-simple bundle $E$ with  $E\in \T, \tau E\in \F$. $E$ is then a  direct summand of $T$. Let $q$ be the maximal slope of indecomposable direct summands of $T$. Then $\Phi_{\infty,q}(T)$ is a tilting sheaf with a nonzero torsion direct summand. Since its indecomposable direct summands can be ordered to be a full exceptional sequence, by Lemma~\ref{exseq simple}, one of the direct summands is a simple sheaf. This finishes the proof.  
\end{proof}

We end this subsection by  determining whether certain torsion pairs yield a noetherian or artinian tilted heart. 
For $P\subset \P^1$, denote by $(\T_P, \F_P)$ the torsion pair in $\coh\X$
\begin{equation}\label{P torsion pair}
	(\add\{\coh_\lambda\X\mid \lambda\in P\},\,\,\add\{\vect\X, \coh_\lambda\X\mid \lambda\in \P^1\backslash P\}).
\end{equation}

\begin{lem}\label{tilted heart no-ar-le}
	Let $P\subset \P^1$.

	\begin{enum2}
	\item The tilted heart $\B=\F_P[1]*\T_P$ is noetherian resp. artinian iff $P=\emptyset$ resp. $P=\P^1$.

	\item Suppose $\X$ is of tubular type. 
	If $\mu\in \R\backslash\Q$ then the tilted heart $\B=\coh^{<\mu}\X[1]*\coh^{>\mu}\X$ 
	is neither noetherian nor artinian. 
	If $\mu\in\bar{\Q}$, the tilted heart $\B=\F[1]*\T$ is noetherian resp. artinian iff $P=\emptyset$ resp. $P=\P^1$, where 
	\[(\T,\F)=(\add\{\coh^{>\mu}\X, \coh^{\mu}_\lambda\X\mid \lambda\in P\},\,\,\add\{\coh^\mu_\lambda\X, \coh^{<\mu}\X\mid \lambda\notin P\}).\]
\end{enum2}
\end{lem}
\begin{proof}
	(1) If $P=\emptyset$ then $\B=\coh\X[1]$, which is noetherian. If $P=\P^1$ then $\B=\vect\X[1]*\coh_0\X\simeq (\coh\X)^{\text{op}}$ is artinian, where the equivalence is induced by  the duality functor $R\mathcal{H}om(-,\O)$. Otherwise, $\emptyset\neq P\neq \P^1$. Let $\lambda\in P, \lambda'\notin P$. Take a line bundle $L$ over $\X$. 
	We have $L(n\vec{c})[1]\in \F[1]*\T$ for all $n\in\Z$. Take an indecomposable torsion sheaf $F_1$ resp. $F_2$ supported at $\lambda$ resp. $\lambda'$ such that $F_i$ fits into an exact sequence $0\ra L(n\vec{c})\ra L((n+1)\vec{c})\ra F_i\ra 0$ in $\coh\X$.  Then for each $n\in \Z$, we have exact sequences in $\F[1]*\T$  
	\[0\ra F_1\ra L(n\vec{c})[1]\ra L((n+1)\vec{c})[1]\ra 0,\,\,\, 0\ra L(n\vec{c})[1]\ra L((n+1)\vec{c})[1]\ra F_2[1]\ra 0.\]  
	The first (resp. second) exact sequence implies the existence of a strict infinite ascending (resp. decending) chain of quotient  objects (resp. subobjects) of $L[1]$ in $\B=\F[1]*\T$. 
	Hence $\B$ is neither noetherian nor artinian in this case. 

	(2) The assertion for $\mu\in\bar{\Q}$ is reduced to (1) by using the telescopic functor $\Phi_{\infty,\mu}$. So we consider $\mu\in\R\backslash \Q$. 
	By applying the duality functor $R\mathcal{H}om(-,\O)$, we know that \[\B=\coh^{< \mu}\X[1]*\coh^{>\mu}\X\simeq  (\coh^{<-\mu}\X[1]*\coh^{>-\mu}\X)^{\text{op}}.\]
		To show that  $\B$ is neither noetherian nor artinian, it sufficies to show that $\B$ is not artianian, which in turn follows readily from our claim that each indecomposable bundle $F$ of slope $>\mu$ fits into an exact sequence $0\ra E\ra G\ra F\ra 0$, where $E\in \coh^{<\mu}\X$ and $G\in \coh^{>\mu}\X$.  

	Let us show our claim. 
	For a  quasi-simple bundle  $A\in \coh^\nu_\lambda\X$ with $\nu<\mu(F)$, consider the evaluation map \[\ev:  \bigoplus_{i=0}^{p_\lambda-1}\hom(\tau^i A, F)\otimes \tau^i A\lra F.\] By \cite[Theorem 5.1.3]{Meltzer}, 
	$\ev$ is either a monomorphism or an epimorphism. By \cite[Theorem 1.7]{N}, there  exists a pair of coprime integers $(h,k)$ such that \[k>\rk(F),\quad \frac{h}{k}<\mu(F),\quad 0<\frac{h}{k}-\mu<\frac{1}{k^2}.\] By Lemma~\ref{stable gcd}, there is  
	a quasi-simple bundle $A\in\coh^{\frac{h}{k}}\X$ with coprime rank and degree. 
	In particular, we have $\rk(A)=k, \deg(A)=h$. Then $\ev$ is an epimorphism and $E:=\ker\,\ev$ is indecomposable.
	Moreover, we have 
	\[
		\begin{aligned}
			\mu(E) & =\frac{\bar{\chi}(A,F)\deg(A)-\deg(F)}{\bar{\chi}(A,F)\rk(A)-\rk(F)}\\
			&= \frac{(\mu(F)-\mu(A))\mu(A)-\frac{1}{\rk(A)^2}\mu(F)}{(\mu(F)-\mu(A))-\frac{1}{\rk(A)^2}} \quad\text{(by Riemann-Roch theorem)}\\
			& <\mu.
		\end{aligned}
		\]
		Hence \[0\lra E\lra \bigoplus_{i=0}^{p-1}\hom(\tau^i A, F)\otimes\tau^i  A\lra F\lra 0\] is the desired exact sequence. We are done.
\end{proof}

\section{Bounded t-structures on $\D^b(\X)$}\label{sec: X t-str}
Throughout this section,
 $\X$ will denote a weighted projective line, $\A=\coh\X$ the category of coherent sheaves over $\X$ and $\D=\D^b(\X)$ the bounded derived category of $\coh\X$. Moreover,  $(\D^{\leq 0},\D^{\geq 0})$ will denote a bounded t-structure  on $\D$ and its heart will be denoted by $\B$. The standard t-structure on $\D^b(\X)$ is denoted by $(\D_\A^{\leq 0}, \D_\A^{\geq 0})$.  

\begin{lem}\label{all width bounded}
	Each bounded t-structure on $\D^b(\X)$ is  width-bounded with respect to the standard  t-structure. In particular, $\B\subset \D_\A^{[m,n]}$ for some $m,n\in \Z$.
\end{lem}
\begin{proof} Recall that for each $\X$, there is a canonical algebra $\Lambda$ such that $\D^b(\X)\simeq \D^b(\Lambda)$. Henceforth we have a bounded t-structure on $\D^b(\X)$ with heart equivalent to $\mod \Lambda$. 
	So bounded t-structures are width-bounded with respect to each other (see Example~\ref{fin simple width bounded}). 
\end{proof}

\subsection{Bounded t-structures which restrict to a t-structure on $\D^b(\coh_0\X)$} In this subsection, we characterize when a bounded t-structure  on $\D^b(\X)$ restricts to a t-structure on $\D^b(\coh_0\X)$ and then describe this class of t-structures. 

The following fact is very useful in analyzing direct summands of truncations of an object.
\begin{lem}\label{direct summand nonzero}
	Let $\T$ be a triangulated category. Assume that $A\overset{f}{\ra} B\overset{g}{\ra} C\cra $ is a  triangle in $\T$ with $\hom^{-1}(A,C)=0$. If $A=A_1\oplus A_2$ and correspondingly $f=(f_1,f_2)$ then $f_1 =0$ implies $A_1=0$. If $C=C_1\oplus C_2$ and $g=(g_1,g_2)^t$ then $g_1= 0$ implies $C_1=0$.
\end{lem}
\begin{proof}
	$f_1=0$ implies $C\iso \cone(f_2)\oplus A_1[1]$ and then $\hom(A_1,A_1)\subset \hom^{-1}(A,C)= 0$ thus $A_1= 0$. Similarly one shows the second assertion.
\end{proof}

\begin{lem}\label{trun bundle}
If $\D^{[m,n]}$ contains a nonzero bundle then for some $m\leq l\leq n$, $\B[-l]$ contains a nonzero bundle.
\end{lem}
\begin{proof}
	We use induction on $n-m$. If $n=m$ then there is nothing to prove. Assume $n>m$. Let $E$ be a nonzero bundle  lying in $\D^{[m,n]}$.  Consider the triangle $E_1\ra E\ra E_2\cra $, where $E_1=\tau_{\leq n-1} E\in \D^{[m,n-1]}, E_2=\tau_{\geq n}E\in \B[-n]$. 
	Recall that	 since $\coh\X$ is hereditary, each object $X$ in $\D^b(\X)$ decomposes as $X\cong \oplus \cH^i(X)[-i]$, where $\cH^i(X)$ is the $i$-th cohomology of $X$. Since $\hom^{-1}(E_1, E_2)=0$, by Lemma~\ref{direct summand nonzero}, $\cH^i(E_1)=0$ for $i\neq 0,1$ and $\cH^j(E_2)=0$ for $j\neq 0,-1$. Hence $E_1$ decomposes as a direct sum $A\oplus B[-1]$ and $E_2$ as a direct sum $C\oplus D[1]$, where $A,B,C,D$ are sheaves. Taking cohomology yields a long exact sequence \[0\lra D\lra A\lra E\lra C\lra B\lra 0.\] If $A=0$ then $D=0$ and thus $\rk(C)>0$, that is, $C$ contains a nonzero bundle direct summand. Since $C\in \B[-n]$, such a direct summand gives a desired   bundle. 
Since $\hom(\coh_0\X, \vect\X)=0$, if $A\neq 0$ then $A$ cannot be a torsion sheaf  by Lemma~\ref{direct summand nonzero}. Thus $A$ contains a nonzero bundle direct summand $F$. Now that $F\in \D^{[m,n-1]}$, the induction hypothesis assures the existence of the desired bundle.
\end{proof}   

Let us make our basic observation on bounded t-structures on $\D^b(\X)$.
\begin{lem}\label{restrict to coh0}\label{cycle of ind torsion}
	The following are equivalent:
	\begin{enum2}
	\item $ \{i\mid \vect\X[i]\cap \B\neq 0\}\subset \{j,j+1\}$ for some $j\in \Z$;

	\item   $(\D^{\leq 0},\D^{\geq 0})$ restricts to a bounded t-structure on $\D^b(\coh_\lambda\X)$ for each $\lambda\in \P^1$;

	\item  $\B$ contains a shift of some non-exceptional indecomposable torsion sheaf. 
	\end{enum2}

\end{lem}

\begin{proof}

(2) $\Ra$ (3)  
Take an ordinary point $\lambda$. The induced bounded t-structure on $\D^b(\coh_\lambda\X)$ has heart $\B_\lambda=\B\cap \D^b(\coh_\lambda\X)$. Since $\lambda$ is ordinary, each bounded t-structure on $\D^b(\coh_\lambda\X)$ is a shift of the standard one by Proposition~\ref{A_t t-str}.  Hence a shift of the simple torsion sheaf $S$ supported at $\lambda$ lies in $\B_\lambda\subset \B$.  

(3) $\Ra$ (1) Suppose $T$ is a non-exceptional indecomposable torsion sheaf such that $T[j]\in \B$.  
	 By Lemma~\ref{bundle hom subtube-simples}, for each nonzero bundle $E$, $\ext^1(T,E)\neq 0$ and $\hom(E,T)\neq 0$. Now that $T[j]\in \B$, if $E[i]\in\B$ then  $m\neq j,j+1$ will yield a contradiction to $\hom^n(\B,\B)=0$ for $n<0$. Hence $\{i\mid \vect\X[i]\cap \B\neq 0\}\subset \{j,j+1\}$.
	
(1) $\Ra$ (2) We will show that (1) implies that $(\D^{\leq 0}, \D^{\geq 0})$ restricts to a bounded t-structure on $\D^b(\coh_0\X)$. Then (2) follows since $\coh_0\X=\coprod_{\lambda\in\P^1}\coh_\lambda\X$. 
	Suppose that $(\D^{\leq 0}, \D^{\geq 0})$ does not restrict to a t-structure on $\D^b(\coh_0\X)$.  Then for some torsion sheaf $T$ and some $l\in\Z$, $\tau_{\leq l} T\notin \D^b(\coh_0\X)$. 
By Lemma~\ref{direct summand nonzero}, $\tau_{\leq l} T$ decomposes as $A\oplus B[-1]$ with $A\in \coh\X, B\in \coh_0\X$ and $\tau_{> l} T$ decomposes as $C\oplus D[1]$ with $C\in \coh_0\X, D\in \coh\X$. $\tau_{\leq l}T\notin \D^b(\coh_0\X)$ implies that $A$ contains a nonzero bundle $E$ as its direct summand. Since $\rk(A)=\rk(D)$, $D$ also contains such a direct summand $F$.  
	Now that $E\in \D^{\leq l}, F\in \D^{\geq l+2}$ and the t-structure is bounded, by Lemma~\ref{trun bundle}, both $\B[-r]$ and $\B[-s]$ contain  nonzero bundles for some $r\leq l, s\geq l+2$. It is then impossible that $\{i\mid \vect\X[i]\cap \B\neq 0\}\subset \{j,j+1\}$ for some $j$.
\end{proof}
	
We are going to give a description of  bounded t-structures on $\D^b(\X)$ satisfying the conditions in the above lemma. 
Recall the definition of a proper collection  of simple sheaves in \S\ref{sec: wpl perp}. Two such collections are said to be equivalent if they yield the same isoclasses of simple sheaves. 
Recall also that for $P\subset\P^1$, the pair $(\T_P,\F_P)$ denotes the torsion pair ~\eqref{P torsion pair} in $\coh\X$. Moreover, we have a split torsion pair $(\SSS^{\perp_\A}\cap \T_P, \SSS^{\perp_\A}\cap \F_P)$ in $\SSS^{\perp_\A}$.

\begin{prop}\label{restrict t-str}

	Suppose  $\{i\in \Z\mid \vect\X[i]\cap \B\neq 0\}=\{j\}$ or $\{j-1,j\}$ for some $j\in\Z$. Then there is a unique (up to equivalence) proper collection $\SSS$ of  simple sheaves such that 
	\begin{itemize}
		\item $(\D^{\leq 0},\D^{\geq 0})$ is compatible with the recollement
			\[ \xymatrix{ \D^b(\SSS^{\perp_\A})=\SSS^{\perp_\D}  \ar[rr]|{i_*}   & &\ar@/_1pc/[ll]|{i^*} \ar@/^1pc/[ll]|{i^!}\D^b(\X) \ar[rr]|{j^*} & &\ar@/_1pc/[ll]|{j_!} \ar@/^1pc/[ll]|{j_*} \pair{\SSS}_\D,}\] where $i_*, j_!$ are the inclusion functors;
		\item if $\{i\mid \vect\X[i]\cap \B\neq 0\}=\{j\}$ then for a unique $P\subset \P^1$, the corresponding t-structure on $\SSS^{\perp_\D}$ is a shift of the HRS-tilt with respect to the torsion pair $(\SSS^{\perp_\A}\cap \T_P,\SSS^{\perp_\A}\cap \F_P)$ in $\SSS^{\perp_\A}$;  

\item if $\{i\mid \vect\X[i]\cap \B\neq 0\}=\{j-1,j\}$ then the corresponding t-structure on $\SSS^{\perp_\D}$ is a shift of the HRS-tilt with respect to some torsion pair $(\T,\F)$ in $\SSS^{\perp_\A}$ with $\SSS^{\perp_\A}\cap \coh_0\X\subsetneq  \T\subsetneq \SSS^{\perp_\A}$.
	\end{itemize}
			
\end{prop}
 
\begin{proof}
 By Lemma~\ref{restrict to coh0}, $(\D^{\leq 0}, \D^{\geq 0})$ restricts to a bounded t-structure on $\D^b(\coh_\lambda\X)$ for each $\lambda\in \P^1$.
 Denote \[\A_\lambda=\coh_\lambda\X,\quad \D_\lambda=\pair{\coh_\lambda\X}_\D=\D^b(\coh_\lambda\X),\] \[\D_\lambda^{\leq 0}=\D^{\leq 0}\cap \D_\lambda,\quad \D_\lambda^{\geq 0}=\D^{\geq 0}\cap \D_\lambda,\quad \B_\lambda=\B\cap \D_\lambda.\] Then $(\D^{\leq 0}_\lambda, \D^{\geq 0}_\lambda)$ is a bounded t-structure on $\D_\lambda$ with heart $\B_\lambda$.  Observe that each  Ext-projective object in $\D^{\leq 0}_\lambda$ is Ext-projective in $\D^{\leq 0}$. Indeed, if $X\in \D_\lambda^{\leq 0}\subset \D^{\leq 0}$ is $\D_\lambda^{\leq 0}$-projective then  $\tau X[1] \in \D_\lambda^{\geq 0}\subset \D^{\geq 0}$, which implies $X$ is $\D^{\leq 0}$-projective. 

	For each $\lambda\in \P^1$, by Proposition~\ref{A_t t-str}, there is a unique proper collection $\SSS_\lambda$ of  simple sheaves supported at $\lambda$ such that $(\D_\lambda^{\leq 0}, \D_\lambda^{\geq 0})$ is compatible with
	\[ \xymatrix{ \SSS_\lambda^{\perp_{\D_\lambda}}  \ar[rr]|{f_*}   & &\ar@/_1pc/[ll]|{f^*} \ar@/^1pc/[ll]|{f^!}\D_\lambda=\D^b(\coh_\lambda\X) \ar[rr]|{g^*} & &\ar@/_1pc/[ll]|{g_!} \ar@/^1pc/[ll]|{g_*} \pair{\SSS_\lambda}_{\D_\lambda},}\] where $f_*, g_!$ are the inclusion functors,
	and the corresponding t-structure on $\SSS_\lambda^{\perp_{\D_\lambda}}$ has heart $\B_\lambda\cap \SSS_\lambda^{\perp_{\D_\lambda}}=\SSS_\lambda^{\perp_{\A_\lambda}}[m_\lambda]$ for some $m_\lambda$. 
	If $\SSS_\lambda=\emptyset$ (say when $\lambda$ is an ordinary point), let $T_\lambda=0$. Otherwise, $\pair{\SSS_\lambda}_{\D_\lambda}$ is triangle equivalent to $\D^b(\coprod_{i=1}^{n_\lambda} \mod k\vec{\AA}_{l_{i,\lambda}})$ for some positive integers $n_\lambda,l_{i,\lambda}$, where $k\vec{\AA}_l$ is the path algebra of the equioriented $\AA_l$-quiver. By Theorem~\ref{silting t-str},  the  t-structure $(g^*\D_\lambda^{\leq 0}, g^*\D_\lambda^{\geq 0})$ on $\pair{\SSS_\lambda}_{\D_\lambda}$ corresponds to a basic silting object $T_\lambda$ in $\pair{\SSS_\lambda}_{\D_\lambda}$ so that $\pair{T_\lambda}_{\D_\lambda} =\pair{\SSS_\lambda}_{\D_\lambda}$ and $T_\lambda$ is $g^*\D_\lambda^{\leq 0}$-projective. 
	By Lemma~\ref{ext-proj from recollement}, $T_\lambda=g_! T_\lambda$ is $\D^{\leq 0}_\lambda$-projective  and hence $T_\lambda$ is $\D^{\leq 0}$-projective. By Proposition~\ref{order silting}, the indecomposable direct summands of $T_\lambda$ can be ordered to form an exceptional sequence.  Let $T=\oplus_\lambda T_\lambda$, $\SSS=\cup_\lambda\SSS_\lambda$. We have $\pair{T}_\D =\pair{\SSS}_\D$ and the indecomposable direct summands of $T$ can be ordered to form an exceptional sequence.  
	Then by Lemma~\ref{ex collection compatible filt}, $(\D^{\leq 0}, \D^{\geq 0})$ is compatible with the recollement 
	\[\xymatrix{ \SSS^{\perp_{\D}}=T^{\perp_\D}  \ar[rr]|{i_*}   & &\ar@/_1pc/[ll]|{i^*} \ar@/^1pc/[ll]|{i^!}\D \ar[rr]|{j^*} & &\ar@/_1pc/[ll]|{j_!} \ar@/^1pc/[ll]|{j_*} \pair{T}_\D =\pair{\SSS}_{\D},}\] where $i_*, j_!$ are the inclusion functors.

	Now let us show that the corresponding t-structure on $\SSS^{\perp_\D}$ takes the asserted form. 	 Let $\B_1=\B\cap \SSS^{\perp_\D}$ be its heart.
 We have $\SSS_\lambda^{\perp_{\A_\lambda}}[m_\lambda]=\B_1\cap \D_\lambda\subset \B_1$. 
	Hence for each $\lambda\in \P^1$, there is  a nonexceptional indecomposable torsion sheaf $F_\lambda$ such that $F_\lambda[m_\lambda]\in \B$.  Up to a shift of $\B$, we can suppose $\{i\mid \vect\X[i]\cap \B\neq 0\}=\{1\}$ or $\{0,1\}$. If $\{i\mid \vect\X[i]\cap \B\neq 0\}=\{1\}$, let $E$ be a nonzero bundle such that $E[1]\in\B$. $\hom(E,F_\lambda)\neq 0$ and $\ext^1(F_\lambda,E)\neq 0$ imply that $m_\lambda\in \{0,1\}$. 
	If $\{i\mid \vect\X[i]\cap \B\neq 0\}=\{0,1\}$ then we have nonzero bundles $E_1,E_2$ with $E_1, E_2[1]\in \B$. $\hom(E_i,F_\lambda)\neq 0$ and $\ext^1(F_\lambda, E_i)\neq 0$ ($i=1,2$) imply $m_\lambda=0$. Consequently, in either case, we have $\B_1\subset \SSS^{\perp_\A}[1]*\SSS^{\perp_\A}$ and thus $\B_1=\F[1]*\T$ for some torsion pair $(\T,\F)$ in $\SSS^{\perp_\A}$. 
	Moreover, if $\{i\mid \vect\X[i]\cap \B\neq 0\}=\{1\}$ then $\T=\add\{\SSS_\lambda^{\perp_{\A_\lambda}}\mid \lambda \in P\}=\SSS^{\perp_\A}\cap \T_P$,   
	where $P=\{\lambda\in\P^1\mid m_\lambda=0\}$; if $\{i\mid \vect\X[i]\cap \B\neq 0\}=\{0,1\}$ then $\SSS^{\perp_{\A}}\cap \coh_0\X\subsetneq \T\subsetneq \SSS^{\perp_\A}$.   

	Finally, the uniqueness of $\SSS$ follows from  the uniqueness of $\SSS_\lambda$; the uniqueness of $P$ follows from Lemma~\ref{t-str recollement bijection}.
\end{proof}
\begin{rmk} Actually, for each bounded t-structure $(\D^{\leq 0}, \D^{\geq 0})$ on $\D$, there exists a unique maximal proper collection $\SSS$ of simple sheaves such that $(\D^{\leq0 },\D^{\geq 0})$ is compatible with the admissible subcategory $\SSS^{\perp_\D}$.  The crucial point to show this is that $\pair{T}_\D =\pair{\SSS}_\D$, where $T$ is the direct sum of a complete set of  indecomposable $\D^{\leq 0}$-projectives of the form $E[n]$ with $E$ a torsion sheaf.
\end{rmk}

\begin{rmk}\label{perp torsion pair}
	Recall from Theorem~\ref{simple perp} that we have an equivalence $\SSS^{\perp_\A}\simeq \coh\X'$ for some weighted projective line $\X'$.
	Via such an equivalence,  the torsion pair $(\SSS^{\perp_\A}\cap \T_P,\,\, \SSS^{\perp_\A}\cap \F_P)$ in $\SSS^{\perp_\A}$ corresponds to the torsion pair $(\T'_P,\F'_P)$ in $\coh\X'$; 
	a torsion pair $(\T,\F)$ in $\SSS^{\perp_\A}$ with $\SSS^{\perp_\A}\cap \coh_0\X\subsetneq \T\subsetneq \SSS^{\perp_\A}$ corresponds to a torsion pair $(\T',\F')$ in $\coh\X'$ with $\coh_0\X'\subsetneq \T'\subsetneq \coh\X'$.
\end{rmk}
Here we characterize when the heart of a bounded t-structure just described is noetherian, artinian or of finite length.
\begin{cor}\label{restrict heart}
	With the notation in Proposition~\ref{restrict t-str}, in the  case  $\{i\mid \vect\X[i]\cap \B\neq 0\}=\{j\}$, the heart $\B$ is not of finite length and $\B$ is noetherian resp. artinian iff $P=\emptyset$ resp. $P=\P^1$; 
	in the case $\{i\mid \vect\X[i]\cap \B\neq 0\}=\{j-1,j\}$, the heart $\B$  is noetherian (artianian or of finite length) iff so is the tilted heart $\F[1]*\T$. 
\end{cor}
\begin{proof}
	Recall that there exist  integers $n, l_1,\dots, l_n$ such that $\pair{\SSS}_\A\simeq  \coprod_{i=1}^n\mod k\vec{\AA}_{l_i}$. By Lemma~\ref{finite rep length heart}, each bounded t-structure on $\pair{\SSS}_\D=\D^b(\pair{\SSS}_\A)\simeq \D^b(\coprod_{i=1}^n\mod k\vec{\AA}_{l_i})$ has length heart. So the assertion for the case $\{i\mid \vect\X[i]\cap \B\neq 0\}=\{j-1,j\}$ follows from Lemma~\ref{no-ar-le}. For the case $\{i\mid \vect\X[i]\cap \B\neq 0\}=\{j\}$, by virtue of the equivalence $\SSS^{\perp_\A}\simeq \coh\X'$ in Theorem~\ref{simple perp}, the assertion follows from Lemma~\ref{tilted heart no-ar-le}(1) and Lemma~\ref{no-ar-le}. 
\end{proof}

\subsection{Bounded t-structures which do not even up to action of $\aut\D^b(\X)$}
Now we deal with bounded t-structures on $\D^b(\X)$ which does not satisfy the condition considered above even up to the action of $\Aut \D^b(\X)$. We only have results for the domestic and tubular cases and 
we rely heavily on the telescopic functors in the tubular case.    

The key feature of this class of t-structures is given in the following lemma.  

\begin{lem}\label{exceptional ind}  
	\begin{enum2}
	\item  If $\X$ is of domestic type then each indecomposable object in $\B$ is exceptional iff $\{i\mid \vect\X[i]\cap \B\neq 0\}\nsubseteq \{j,j+1\}$ for any $j\in\Z$. 

	\item $\X$ is of tubular type then each indecomposable object in $\B$ is exceptional iff $\{i\mid \vect\X[i]\cap \Phi_{\infty,q}(\B)\neq 0\}\nsubseteq \{j,j+1\}$ for any $q\in \bar{\Q}$ and $j\in\Z$.
	\end{enum2}
\end{lem}
\begin{proof}	Each indecomposable object in $\B$ is of the form $E[n]$ for some $n\in\Z$ and some indecomposable bundle  or some indecomposable torsion sheaf $E$.

	(1) By Theorem~\ref{domestic bundle}, if $\X$ is of domestic type then each indecomposable bundle is exceptional. So $\B$ contains a non-exceptional indecomposable object iff $\B$ contains a shift of a non-exceptional torsion sheaf, which is equivalent to  $\{i\mid \vect\X[i]\cap \B\neq 0\}\subseteq \{j,j+1\}$ for some $j\in\Z$ by Lemma~\ref{restrict to coh0}. So our assertion holds.
	
	(2) By Theorem~\ref{tubular bundle}, if $\X$ is of tubular type then each indecomposable sheaf is semistable and thus lies in $\coh^\mu\X$ for some $\mu\in \bar{\Q}$. $\B$ contains a non-exceptional indecomposable object $E[n]$, where $E$ is a sheaf with slope $q$, iff the heart $\Phi_{\infty,q}(\B)[-n]$ contains the non-exceptional torsion sheaf $\Phi_{\infty,q}(E)$. Thus our assertion  follows from Lemma~\ref{restrict to coh0}.
\end{proof}

We show that $\B$ contains no cycle if each indecomposable object in $\B$ is exceptional.

\begin{lem}\label{ordered ind}
	Suppose $\X$ is of dometic or tubular type. If each indecomposable object in $\B$ is exceptional then a complete set of pairwise non-isomorphic indecomposable objects in $\B$ can be totally ordered as $\{X_i\}_{i\in I}$ such that $\hom(X_i,X_j)=0$ if $i<j$.
\end{lem}
\begin{proof}
Each indecomposable object in $\B$ is of the form $E[n]$ for some  indecomposable sheaf $E$. Since $\hom(E[n], F[m])=0$ for $E,F\in \A$ and $n>m$, it sufficies to order indecomposables in $\B\cap \A[n]$, or rather, indecomposables in $\B[-n]\cap \A$. For $\X$ of domestic or tubular type, each idecomposable sheaf is semistable and $\hom(E,F)=0$ for indecomposable sheaves $E,F$ with $\mu(E)>\mu(F)$. Thus we only need to consider indecomposable sheaves with the same slope, i.e., those in $\B[-n]\cap \coh^\mu\X$. We have assumed these indecomposables to be exceptional.
	
We consider $\mu=\infty$ at first.  
	If indecomposables in $\B[-n]\cap \coh^\infty\X=\B[-n]\cap \coh_0\X$ cannot be totally ordered as desired then   $\B[-n]\cap \coh_0\X$ will contain a cycle  of indecomposables in some  $\coh_\lambda\X$. 
	By Lemma~\ref{order sequence orthogonal}, $\B$ contains a non-exceptional object, a contradiction. Hence indecomposables in $\B[-n]\cap \coh^\infty\X$ can be totally ordered as desired.  
	Now we consider $\mu\in\Q$. If $\X$  is of domestic type then   indecomposable bundles in $\coh^\mu\X$ are stable  and thus the morphism spaces between each other vanish, whence any  order is satisfying.  If $\X$ is of tubular type then using the telescopic functor $\Phi_{\infty, \mu}$, we know from the conclusion for $\mu=\infty$ that the desired ordering also exists.
\end{proof}

Recall the definition of $\mu(\B)$ from ~\eqref{def of mu}. Observe that each limit point in $\mu(\B)$ is a limit point of $\mu(\B[l]\cap \A)$ for some $l$ since $\A$ is hereditary and since $\B\subset \D_\A^{[m,n]}$ for some $m,n\in\Z$ by Lemma~\ref{all width bounded}.

\begin{lem}\label{unbounded slope}\label{inf limit point}
	 $\infty$ is a limit point of $\mu(\B)$  iff $\{i\mid \vect\X[i]\cap \B\neq 0\}\subset \{j,j+1\}$ for some $j\in\Z$.
 \end{lem}
 \begin{proof}
	 \nec If $\infty$ is a limit point of $\mu(\B)$ then there is a sequence $(E_i)_{i=1}^\infty$ of objects in some $\B[l]$, where $E_i$'s are indecomposable bundles, such that \[\mu(E_i)\ra +\infty\,\, \text{or}\,\, \mu(E_i)\ra -\infty\,\,\text{as}\,\,i\ra +\infty.\] If $\mu(E_i)\ra +\infty$ then by Theorem~\ref{bundle hom nonzero}, for each nonzero bundle $F$, $\hom(F,E_i)\neq 0$ and $\ext^1(E_i,F)\neq 0$ for $i\gg 1$. Consequently, $F[k]\in \B$ implies $k\in\{l,l+1\}$. Similar arguments apply to the case $\mu(E_i)\ra -\infty$.  

 \suf Suppose $\{i\mid \vect\X[i]\cap \B\neq 0\}=\{j\}$ or $\{j,j+1\}$. By Proposition~\ref{restrict t-str}, we can take a line bundle $L$ such that $L\in \B[-j]$. Moreover, for a simple sheaf   $S$ supported at an ordinary point, i) $S\in \B[-j]$ or ii) $S[1]\in \B[-j]$. If case i) happens, $L(n\vec{c})\in \B[-j]$ for $n\geq 0$; if case ii) happens, $L(n\vec{c})\in \B[-j]$ for $n\leq 0$. In either case, $\infty$ is  a limit point of $\mu(\B)$.
\end{proof}

The following  lemma allows us to apply a telescopic functor in the next proposition. 
\begin{lem}\label{rational limit point}
	If $\X$ is of tubular type and $\mu(\B)$ has an irrational number as its limit point then for some $q\in \bar{\Q}$, $\Phi_{\infty,q}(\B)$ coincides with a shift of the tilted heart with respect to some torsion pair in $\A$.
\end{lem}

\begin{proof} Suppose that for some $l\in\Z$, $\mu(\B\cap \A[l])$ has an irrational number $r$ as its limit point.  
	Then there is a sequence $(E_i)_{i=1}^\infty$ of indecomposable bundles  such that $E_i\in \B[-l]$ and  $\mu(E_i)$ converges to $r$. 
	Let $E$ be an indecomposable sheaf with $\mu(E) < r$. By Corollary~\ref{tubular nonvanish hom}, there are some $E_i$ with $\hom(E, E_i)\neq 0$  and some $E_j$ with $\hom(\tau^{-1}E, E_j)\neq 0$, which implies $\ext^1(E_j, E)\neq 0$. 
  Thus for $h\in\Z$, $E[h]\in \B$ implies $h\in \{l,l+1\}$. 
  Similarly, if $F$ is an indecomposable sheaf with $\mu(F)>r$, then for some $E_i, E_j$, $\hom(E_i, F)\neq 0,$ $\ext^1(F, E_j)\neq 0$. 
For $h\in\Z$,  $F[h]\in \B$ implies $h\in \{l, l-1\}$. 
Consequently, if $\mu(\B\cap \A[l])$ has an irrational limit point $r$ then \[\{k\in\Z\mid \B\cap \A[k]\neq 0\}\subset \{l-1, l, l+1\}\] and an indecomposable sheaf in $\B[-1-l]$ (resp. $\B[1-l]$) has slope $<r$ (resp. $>r$).   
  
  If $\mu(\B\cap \A[l+1])$ also has an irrational number as its limit point then similar arguments as before show that $\{k\mid \B\cap \A[k]\neq 0\}\subset \{l, l+1\}$, that is, $\B\subset \A[l+1]*\A[l]$. Thus $\B$ is a shift of the tilted heart with respect to some torsion pair in $\A$. 
	Consider the case that the set of  limit points  of $\mu(\B\cap \A[l+1])$ is contained  in $\bar{\Q}$. 
  Since each indecomposable sheaf in $\B[-l-1]\cap \A$ has slope less than $r$, there is some rational number $q<r$ such that each indecomposable sheaf $E\in \B[-l-1]\cap \A$ has slope  $\mu(E)\leq q$. 
  Then $\Phi_{\infty,q}(\B\cap \A[l+1])\subset \A[l+1]$. Since an indecomposable object $E\in \B[1-l]\cap \A$ has slope $\mu(E)>r>q$, we have $\Phi_{\infty, q}(\B\cap \A[l-1])\subset \A[l]$. 
  It follows that \[\Phi_{\infty,q}(\B)=\Phi_{\infty,q}(\add\{\B\cap \A[l-1],\B\cap \A[l],\B\cap \A[l+1]\})\subset \A[l+1]*\A[l],\] as desired. 
\end{proof}

 The class of bounded t-structures on $\D^b(\X)$ under consideration  is reminiscent of  bounded t-structures on $\D^b(\Lambda)$, where $\Lambda$ is a representation-finite finite dimensional hereditary algebra, as the following proposition indicates.
\begin{prop}\label{not restrict to coh0}
	If one of the following cases occurs:
	\begin{itemize}
		\item $\X$ is of domestic type and $\{i\mid \vect\X[i]\cap \B\neq 0\}\nsubseteq \{j,j+1\}$ for any $j\in\Z$,
		\item $\X$ is of tubular type and $\{i\mid \vect\X[i]\cap \Phi_{\infty,q}(\B)\neq 0\}\nsubseteq \{j,j+1\}$ for any $q\in \bar{\Q}$ and $j\in\Z$,
	\end{itemize}
	then $\B$ is a length category with finitely many (isoclasses of) indecomposables and each indecomposable object in $\B$ is exceptional. 
\end{prop}
\begin{proof}
	It has been shown in Lemma~\ref{exceptional ind} that each indecomposable object in $\B$ is exceptional under the given condition. We show that $\B$ contains finitely many indecomposables.  If $\B$ contains infinitely many indecomposables then for some $n$, $\B[n]\cap \A$ contains infinitely many indecomposables. 
	But for each $\mu\in \bar{\Q}$, $\coh^\mu \X$ contains finitely many exceptional indecomposables. Thus $\mu(\B[n]\cap \A)$ has a limit point in $\bar{\R}$. Note that an indecomposable object in $\A$ is either a torsion sheaf or a vector bundle. For $\X$ of domestic type, since rank on indecomposables is bounded, $\infty$ is the unique limit point of $\mu(\B[n]\cap \A)$. By Lemma~\ref{unbounded slope}, $\{i\mid \vect\X[i]\cap \B\neq 0\}\subset \{j,j+1\}$ for some $j$, a contradiction. For $\X$ of tubular type, under the given assumption, by Lemma~\ref{rational limit point}, $\mu(\B)$ contains at most limit points in $\bar{\Q}$.
	If $q\in \bar{\Q}$ is a limit point of $\mu(\B)$, $\infty$ is  a limit point of $\mu(\Phi_{\infty, q}(\B))$, whereby yielding a contradiction to our assumption by Lemma~\ref{inf limit point}. Thus in either case, $\B$ contains only finitely many indecomposables. It remains to show that $\B$ is of finite length.
	Let $\{X_1,\dots,X_n\}$ be a complete set of indecomposable objects in $\B$. We have $\End(X_i)=k$. Moreover, by Lemma~\ref{ordered ind}, we can suppose $\hom(X_i,X_j)=0$ for $i<j$. Then one sees that if $\oplus_{i=1}^n X_i^{\oplus s_i}$ is a proper subobject of $\oplus_{i=1}^nX_i^{\oplus t_i}$ then $(s_1,\dots, s_n)<(t_1,\dots, t_n)$, where $<$ refers to the lexicographic order. 
	If follows immediately that $\B$ must be of finite length. This finishes the proof. 
\end{proof}

As a corollary, we obtain a characterization of when a bounded t-structure on $\D^b(\X)$, where $\X$ is of domestic type, has length heart. 
\begin{cor}\label{domestic length heart} 
	If $\X$ is of domestic type then $\B$ is of finite length iff $\sharp\{i\mid \vect\X[i]\cap \B\neq 0\}>1$. 
\end{cor}
\begin{proof}
	Proposition~\ref{not restrict to coh0} tells us that if $\{i\mid \vect\X[i]\cap \B\neq 0\}\nsubseteq \{j,j+1\}$ for any $j$ then $\B$ is of finite length. So consider those $\B$ with $\{i\mid \vect\X[i]\cap \B\neq 0\}= \{j\}$ or $\{j-1,j\}$ for some $j$.  
	By Corollary~\ref{restrict heart}, if $\{i\mid \vect\X[i]\cap \B\neq 0\}=\{j\}$ then $\B$ is not of finite length. So consider the case $\{i\mid \vect\X[i]\cap \B\neq 0\}=\{j-1,j\}$. We shall apply Proposition~\ref{restrict to coh0} and keep the notation there. By Theorem~\ref{simple perp}, we have an equivalence $\SSS^{\perp_\A}\simeq \coh\X'$, where $\X'$ is also a weighted projective line of domestic type.  
	By Remark~\ref{perp torsion pair}, the corresponding t-structure on $\D^b(\X')\simeq \D^b(\SSS^{\perp_\A})$ has up to shift the tilted heart $\F'[1]*\T'$ for some torsion pair $(\T',\F')$ in $\coh\X'$ with $\coh_0\X'\subsetneq  \T'\subsetneq  \coh\X'$. 
	By Lemma~\ref{no quasi-simple}(1), we have a line bundle $L\in \T'$ with $\tau L\in \F'$. Let $(\D_1^{\leq 0}, \D_1^{\geq 0})$ be the bounded t-structure on $\D_1:=\D^b(\X')$ with heart $\F'[1]*\T'$. By Lemma~\ref{ext-proj}, $L$ is $\D_1^{\leq 0}$-projective.  
	By Lemma~\ref{ext-proj recollement},  $(\D_1^{\leq 0}, \D_1^{\geq 0})$ is compatible with the admissible subcategory $L^{\perp_{\D_1}}=\D^b(L^{\perp_{\coh\X'}})$ of $\D_1=\D^b(\X')$.  
We know from Lemma~\ref{wpl perp}(1) that $L^{\perp_{\coh\X'}}\simeq \mod \Lambda$ for a  representation-finite  finite dimensional hereditary algebra. 
	Then by Lemma~\ref{finite rep t-str}, each bounded t-structure of $L^{\perp_{\D_1}}=\D^b(L^{\perp_{\coh\X'}})$ has length heart. Moreover, ${}^{\perp_{\D_1}}(L^{\perp_{\D_1}})=\pair{L}_{\D_1}\simeq \D^b(k)$. Thus the tilted heart $\F'[1]*\T'$ is of finite length by Lemma~\ref{no-ar-le}. So is $\B$. 
	In conclusion, $\B$ is of finite length iff $\sharp \{i\mid \vect\X[i]\cap \B\neq 0\}>1$. 
\end{proof}

\subsection{Some properties of silting objects}
Recall K\"onig-Yang correspondence (see Theorem~\ref{silting t-str}) that equivalent classes of silting objects in $\D^b(\X)$ are in bijective correspondence with bounded t-structures on $\D^b(\X)$ with length heart.
 So we  continue to describe some properties of silting objects in $\D^b(\X)$, which in turn give information on bounded t-structures with length heart. 

 By Proposition~\ref{hereditary ext vanish}, the direct summands of a basic silting object $T$ in $\D^b(\X)$ can be ordered to form a full exceptional sequence. We obtain the following information on directs summands of $T$ from our previous conclusion on full exceptional sequences. This holds particularly for a tilting object in $\D^b(\X)$.
\begin{prop}\label{silting summand}
	Let $T$ be a  silting object in $\D^b(\X)$. 

	\begin{enum2}
	\item If $T$ contains a shift of a torsion sheaf as its direct summand then $T$ contains a shift of an exceptional simple sheaf as its direct summand. 
	
	\item If $\X$ is of domestic  type  then $T$ contains a  shift of some line bundle as its direct summand. 
 
	\item If $\X$ is of tubular type then  for a suitable $q\in \bar{\Q}$, $\Phi_{\infty, q}(T)$ contains a shift of some exceptional simple torsion sheaf  and a shift of a line bundle as its direct summands.
	\end{enum2}
 \end{prop}
 \begin{proof}
	 (1) follows immediately from Lemma~\ref{exseq simple}, (2) from Proposition~\ref{domestic exseq} and (3) from Corollary~\ref{tubular exseq}.
\end{proof}
 
 A silting object $T$ in $\D^b(\X)$ is called \emph{concentrated} if $T$ contains nonzero direct summands in $\vect\X[m]$ for a unique $m$. This is a generalization of the notion of a concentrated tilting complex (\cite[Definition 9.3.3]{Meltzer}). 	
 \begin{lem}\label{concentrated silting}
		A silting object $T$ in $\D^b(\X)$ is concentrated iff the corresponding bounded t-structure $(\D^{\leq 0},  \D^{\geq 0})$ satisfies the  property $\{i\in\Z\mid \vect\X[i]\cap \B\neq 0\}\subset \{j,j+1\}$ for some $j\in\Z$. 
	\end{lem}

	\begin{proof}
		Recall that in K\"onig-Yang correspondence, the t-structure $\tstr$ corresponding to  $T$ has heart \[\B=\{X\in \D^b(\X)\mid \hom^{\neq 0}(T,X)=0\}.\] Let $T$ be a concentrated silting object, say $T=T_1\oplus T_2$ with $T_1\in \vect\X[l]$ and $T_2\in\D^b(\coh_0\X)$. By Happel-Ringel Lemma (see Proposition~\ref{Happel-Ringel lemma}), the indecomposable direct summands of $T_2$ are exceptional.  Hence $T_2$ is supported at exceptional points. For a simple sheaf $S$ supported at an ordinary point, we have $\hom^{\neq 0}(T_1, S[l])=0$ and $\hom^k(T_2,S[l])=0$ for any $k\in\Z$ and thus $S[l]$ lies in $\B$. If follows from Lemma~\ref{restrict to coh0} that $\{i\in\Z\mid \vect\X[i]\cap \B\neq 0\}\subset \{j,j+1\}$ for some $j$.  

	Conversely, suppose $\{i\mid \vect\X[i]\cap \B\neq 0\}\subset \{j,j+1\}$ for some $j$. 
		By Proposition~\ref{restrict t-str} and Remark~\ref{perp torsion pair}, there is a proper collection $\SSS$ of simple sheaves such that the t-structure $(\D^{\leq 0}, \D^{\geq 0})$  is compatible with the admissible subcategory $\D^b(\SSS^{\perp_\A})$ of $\D^b(\X)$ and  up to shift the corresponding t-structure $(\D_1^{\leq 0}, \D_1^{\geq 0})$ on $ \D^b(\X')\simeq \D^b(\SSS^{\perp_\A})$ has heart $\F'[1]*\T'$ for  some torsion pair $(\T', \F')$ in $\coh\X'$ with $\coh_0\X'\subsetneq \T'\subsetneq \coh\X'$. 
		Since $\B$ is of finite length, so is $\F'[1]*\T'$ by Lemma~\ref{no-ar-le} and thus $(\T', \F')$ is induced by a tilting bundle in $\coh\X'$ by Proposition~\ref{length torsion pair}. Hence indecomposable Ext-projectives in $\D_1^{\leq 0}$ are bundles. If $X[n]$ is an indecomposable direct 
		summand of $T$ with $X$ a bundle then $X[n]$ is $\D^{\leq 0}$-projective and thus $i^* X[n]$ is nonzero $\D_1^{\leq 0}$-projective by Lemma~\ref{ext-proj from recollement}, where $i^*$ is the left adjoint of the composition $\D^b(\X')\overset{\sim}{\ra} \D^b(\SSS^{\perp\A})\monic \D^b(\X)$. This implies that $i^* X[n]$ is a nonzero bundle.
		By Theorem~\ref{simple perp}(2),   $i^*$ is t-exact with respect to the standard t-structures. So we have $n=0$. Hence $T$ is concentrated.
\end{proof}

We now give some properties of the endomorphism algebra of a silting object in $\D^b(\X)$. This generalizes parts of \cite[Theorem 9.4.1, 9.5.3]{Meltzer}. 

\begin{prop}\label{silting property}
	Let $T$ be a silting object in $\D^b(\X)$ and $\Gamma=\End(T)$. 

	\begin{enum2}
	\item The quiver of $\Gamma$ has no oriented cycle. In particular, $\Gamma$ has finite global dimension. 
 
	\item If $\X$ is of domestic or tubular type then $\Gamma$ is either representation infinite or representation directed. 
	
	\item For $\X$ of domestic type, $\Gamma$ is representation infinite iff $T$ is concentrated. 
	
	\item For $\X$ of tubular type, $\Gamma$ is representation infinite iff $\Phi_{\infty, q}(T)$ is concentrated for some $q\in \bar{\Q}$. 
	\end{enum2}
\end{prop}

\begin{proof}
 Let $(\D^{\leq 0}, \D^{\geq 0})$ be the bounded t-structure corresponding to $T$ in K\"onig-Yang correspondence. Its heart $\B$ is equivalent to $\mod\, \Gamma$. 

  (1)  We can assume $T$ is basic. Then by Proposition~\ref{hereditary ext vanish}, indecomposable direct summands of $T$ can be ordered to form an exceptional sequence. Hence the quiver of $\Gamma=\End(T)$ has no oriented cycle. 

  (2)  If $\Gamma$ is not representation infinite then $\B\simeq \mod\,\Gamma$ contains finitely many indecomposables. Thus $\B$ contains no non-exceptional object by Lemma~\ref{restrict to coh0} (for the tubular case, we may need an additional application of a telescopic functor to apply Lemma~\ref{restrict to coh0}.). By Lemma~\ref{ordered ind}, each object in $\mod\,\Gamma\simeq \B$ is directed. So $\Gamma$ is representation directed.

  (3)  Suppose  $T$ is concentrated. By Lemma~\ref{concentrated silting}, we have $\{i\in\Z\mid \vect\X[i]\cap \B\neq 0\}\subset \{j,j+1\}$ for some $j\in\Z$. By Lemma~\ref{unbounded slope}, $\B$ contains infinitely many  indecomposables. Since $\mod\,\Gamma\simeq \B$, $\Gamma$ is representation infinite. Conversely, suppose $\Gamma$ is representation infinite, then $\B$ contains infinitely many indecomposables. By Proposition~\ref{not restrict to coh0}, we have  $\{i\in\Z\mid \vect\X[i]\cap \B\neq 0\}\subset \{j,j+1\}$ for some $j\in\Z$. Then Lemma~\ref{concentrated silting} implies that $T$ is concentrated.  

  (4) The argument is similar to that for (3), except that we need to take into account the action of a suitable telescopic functor $\Phi_{\infty,q}$. We remark that $\Phi_{\infty,q}(T)$ corresponds to the bounded t-structure with heart $\Phi_{\infty, q}(\B)$. 
\end{proof}

\subsection{Description of bounded t-structures on $\D^b(\X)$}\label{sec: thm}
We are in a position to formulate our description of bounded t-structures on $\D^b(\X)$ using HRS-tilt and recollement.  
 Recall once again that for $P\subset \P^1$, $(\T_P,\F_P)$ denotes the torsion pair ~\eqref{P torsion pair} in $\coh\X$.

We begin with the domestic case.
\begin{thm}\label{thm: domestic}
	Let $\X$ be a weighted projective line of domestic type. Suppose $(\D^{\leq 0}, \D^{\geq 0})$ is a bounded t-structure on $\D^b(\X)$ with heart $\B$. Then exactly one of the following holds: 

	\begin{enum2}
	\item up to the action of $\pic\X$, $(\D^{\leq 0}, \D^{\geq 0})$ is compatible with the recollement \[ \xymatrix{\O^{\perp_\D}  \ar[rr]|{i_*}   & &\ar@/_1pc/[ll] \ar@/^1pc/[ll]\D=\D^b(\X) \ar[rr] & &\ar@/_1pc/[ll]|{j_!} \ar@/^1pc/[ll] \pair{\O}_\D,}\] 
	where $i_*,j_!$ are the inclusion functors, in which case $\B$ is of finite length;

\item for a unique (up to equivalence) proper collection $\SSS$ of  simple sheaves and a unique $P\subset \P^1$, $(\D^{\leq 0}, \D^{\geq 0})$ is compatible with the recollement
	\[ \xymatrix{ \D^b(\SSS^{\perp_\A})=\SSS^{\perp_\D}  \ar[rr]|{i_*}   & &\ar@/_1pc/[ll] \ar@/^1pc/[ll]\D=\D^b(\X) \ar[rr] & &\ar@/_1pc/[ll]|{j_!} \ar@/^1pc/[ll] \pair{\SSS}_\D,}\]
	where $i_*,j_!$ are the inclusion functors, 
	such that the corresponding t-structure on $\SSS^{\perp_\D}$ is a shift of the HRS-tilt with respect to the torsion pair $(\SSS^{\perp_\A}\cap \T_P, \SSS^{\perp_\A}\cap \F_P)$ in $\SSS^{\perp_\A}$, 
	in which case $\B$ is not of finite length and $\B$ is noetherian resp. artinian iff $P=\emptyset$ resp. $P=\P^1$.
	\end{enum2}	
\end{thm}

\begin{proof}

If $\B$ is of finite length then the corresponding basic silting object is the direct sum of  a complete set of indecomposable Ext-projectives in $\D^{\leq 0}$. 
	By Lemma~\ref{silting summand}, $\D^{\leq 0}$ has an Ext-projective object which is a shift of  a line bundle and thus up to the action of $\pic\X$, $(\D^{\leq 0}, \D^{\geq 0})$ is compatible with the recollement given in (1). 
	Conversely, if $(\D^{\leq 0}, \D^{\geq 0})$ is compatible with the recollement in (1), then $\B$ is of finite length since bounded t-structures on $\O^{\perp_\D}$ and $\pair{\O}_\D$ have length heart. 
	If $\B$ is not of finite length then by Proposition~\ref{not restrict to coh0},  $\B$ satisfies the assumption of Proposition~\ref{restrict t-str}. 
	By Corollary~\ref{domestic length heart}, $\B$ is not of finite length iff $\{i\mid \vect\X[i]\cap \B\neq 0\}=\{j\}$ for some $j\in\Z$ and thus $(\D^{\leq 0}, \D^{\geq 0})$ fits into type (2) by Proposition~\ref{restrict t-str}. The assertion on the noetherianness or artianness of $\B$ in this case is shown in  Corollary~\ref{restrict heart}.
\end{proof}

For the tubular case, we need one more lemma characterizing when the heart $\B$ is of finite length.

\begin{lem}\label{tubular length ext-proj}
	Suppose $\X$ is of tubular type. Then $\B$ is of finite length iff there are two indecomposable sheaves $E,F$ with $\mu(E)\neq \mu(F)$ for which $E[m],F[n]$ are $\D^{\leq 0}$-projectives for some $m,n$.
\end{lem}
\begin{proof}
	\nec Let $T$ be  a corresponding silting object. Then by Proposition~\ref{silting summand}, for some $q\in \bar{\Q}$, $\Phi_{\infty,q}(T)$ contains a shift of some simple sheaf and a shift of some line bundle as its direct summands. The assertion follows immediately.

	\suf  By Proposition~\ref{order silting}, either $(E,F)$ or $(F,E)$ is an exceptional pair. We only consider the case that $(F,E)$ is an exceptional pair since the other case is similar. 
	By Corollary~\ref{ex collection compatible filt}, $(\D^{\leq 0},\D^{\geq 0})$ is compatible with the admissible filtration \[\D^b(\{E,F\}^{\perp_\A})=\{E,F\}^{\perp_\D}\subset E^{\perp_\D}\subset \D.\] If $\mu(E)\neq \mu(F)$ then by Lemma~\ref{wpl perp}(2), $\{E,F\}^{\perp_\A}\simeq \mod \Lambda$ for some representation-finite  finite dimensional hereditary algebra $\Lambda$. It follows from Corollary~\ref{filt no-ar-le} and Lemma~\ref{finite rep length heart} that $\B$ is of finite length.
\end{proof}

Here comes our description of bounded t-structures in the tubular case.
\begin{thm}\label{thm: tubular}
	Let $\X$ be a weighted projective line of tubular type. Suppose $(\D^{\leq 0}, \D^{\geq 0})$ is a bounded t-structure on $\D^b(\X)$ with heart $\B$. Then exactly one of the following holds:

	\begin{enum2}
	\item for a unique $\mu \in\R\backslash \Q$, $\tstr$ is a shift of the HRS-tilt with respect to the torsion pair $(\coh^{>\mu}\X, \coh^{<\mu}\X)$ in $\coh\X$, in which case $\B$ is neither noetherian nor artinian;

	\item for a unique $\mu\in\bar{\Q}$ and a unique $P\subset \P^1$, $\tstr$ is a shift of the HRS-tilt with respect to the torsion pair 
		\[(\add\{\coh^{>\mu}\X, \coh_\lambda^\mu\X\mid \lambda\in P\},\,\,\add\{\coh_\lambda^\mu\X, \coh^{<\mu}\X\mid \lambda\in \P^1\backslash P\})\] in $\coh\X$, in which case $\B$ is not of finite  length and $\B$ is noetherian resp. artinian iff $P=\emptyset$ resp. $P=\P^1$;
	
		\item for a unique $q\in \bar{\Q}$, a unique (up to equivalence) nonempty proper collection $\SSS$ of simple sheaves and a unique $P\subset \P^1$, $\Phi_{\infty,q}((\D^{\leq 0}, \D^{\geq 0}))$ is compatible with the recollement  
			\[ \xymatrix{ \D^b(\SSS^{\perp_\A})=\SSS^{\perp_\D}  \ar[rr]|{i_*}   & &\ar@/_1pc/[ll] \ar@/^1pc/[ll]\D=\D^b(\X) \ar[rr] & &\ar@/_1pc/[ll]|{j_!} \ar@/^1pc/[ll] \pair{\SSS}_\D,}\]
	where $i_*,j_!$ are the inclusion functors, such that the corresponding t-structure on $\D^b(\SSS^{\perp_\A})$ is a shift of the HRS-tilt with respect to the torsion pair $(\SSS^{\perp_\A}\cap \T_P, \SSS^{\perp_\A}\cap \F_P)$ in $\SSS^{\perp_\A}$, 
	in which case $\B$ is not of finite length and $\B$ is noetherian resp. artinian iff $P=\emptyset$ resp. $P=\P^1$;
	
	\item for some $q\in \bar{\Q}$ and some exceptional simple sheaf $S$,  $\Phi_{\infty,q}((\D^{\leq 0},\D^{\geq 0}))$ is compatible with the recollement  
		\[ \xymatrix{ \D^b(S^{\perp_\A})=S^{\perp_\D}  \ar[rr]|{i_*}   & &\ar@/_1pc/[ll] \ar@/^1pc/[ll]\D=\D^b(\X) \ar[rr] & &\ar@/_1pc/[ll]|{j_!} \ar@/^1pc/[ll] \pair{S}_\D,}\]
	where $i_*,j_!$ are the inclusion functors, such that the corresponding t-structure on $\D^b(S^{\perp_\A})$ has length heart, in which case $\B$ is of finite length.
\end{enum2}
\end{thm}
\begin{proof}
	If $\B$ is of finite  length then by Proposition~\ref{silting summand}, for some $q\in \bar{\Q}$, there is some exceptional simple sheaf $S$ which is $\Phi_{\infty,q}(\D^{\leq l})$-projective for some $l$. Hence $\Phi_{\infty,q}((\D^{\leq 0},\D^{\geq 0}))$ is compatible with the recollement of the form in (4). 
	The corresponding t-structure on $\D^b(S^{\perp_\A})$ has length heart by Lemma~\ref{no-ar-le}. Suppose $\B$ is not of finite length.
	By Proposition~\ref{not restrict to coh0}, for some $q\in \bar{\Q}$ and some $j\in\Z$, $\{i\mid \vect\X[i]\cap \Phi_{\infty,q}(\B)\neq 0\}\subset \{j,j+1\}$.  Thus Proposition~\ref{restrict t-str} applies.  
	Moreover, by Lemma~\ref{tubular length ext-proj}, either (I) $\Phi_{\infty,q}(\D^{\leq 0})$ contains no nonzero Ext-projective or (II) all indecomposable $\Phi_{\infty,q}(\D^{\leq 0})$-projectives has the same slope.  
	
	First consider the case (I): $\Phi_{\infty,q}(\D^{\leq 0})$ contains no nonzero Ext-projective. Then the asserted collection $\SSS$ of simple sheaves in Proposition~\ref{restrict t-str} is empty by Lemma~\ref{ext-proj from recollement}. Hence up to shift we have two cases: 1) $\Phi_{\infty,q}(\B)=\F_P[1]*\T_P$ for some $P\subset \P^1$, or 2) $(\T,\F)$ is a torsion pair in $\coh\X$ with $\coh_0\X\subsetneq \T\subsetneq \coh\X$. 
	Moreover, for case 2), there exists no nonzero sheaf $E\in \T$ with $\tau E\in \F$ since $\Phi_{\infty,q}(\D^{\leq 0})$ contains no nonzero Ext-projective. By Lemma~\ref{no quasi-simple}, we have either 2.1) $(\T,\F)=(\coh^{>\mu}\X, \coh^{<\mu}\X)$ for some $\mu\in\R\backslash\Q$, or 2.2) for some $\mu\in\Q$ and some $P\subset \P^1$, \[(\T,\F)=(\add\{\coh^{>\mu}\X, \coh^{\mu}_\lambda\X\mid \lambda\in P\},\,\, \add\{\coh^\mu_\lambda\X, \coh^{<\mu}\X\mid \lambda\notin P\}).\] 
	If case 2.1) occurs then $\Phi_{\infty,q}(\tstr)$ is of type (1); if 1) or 2.2) occurs, $\Phi_{\infty,q}(\tstr)$ is of type (2). Observe that  the class of t-structures of type (1) or (2) is closed under the action of the telescopic functor $\Phi_{q,\infty}=\Phi_{\infty,q}^{-1}$.   
	Hence $(\D^{\leq 0},\D^{\geq 0})$ is of type (1) or (2). It is evident that types (1) and (2) are  disjoint and the assertion on uniqueness is also obvious. The assertion on noetherianness or artinianness is proved in Lemma~\ref{tilted heart no-ar-le}.

	Now consider the case (II): all indecomposable $\Phi_{\infty,q}(\D^{\leq 0})$-projectives has the same slope, which we denote by $\mu$. By Lemma~\ref{ext-proj from recollement}, the compatibility of $\Phi_{\infty,q}(\tstr)$ with the recollement in Proposition~\ref{restrict t-str} implies that there is a torsion sheaf which is Ext-projective in some $\Phi_{\infty,q}(\D^{\leq l})$. Thus $\mu=\infty$. It follows that if an indecomposable sheaf  $E$ is Ext-projective in some $\D^{\leq l}$ then $\mu(E)=q$. This enforces the uniqueness of $q$.   The uniqueness of $\SSS$ and $P$ is then asserted in Proposition~\ref{restrict t-str}. 
	To show that $(\D^{\leq 0}, \D^{\geq 0})$ is of type (3), we will show that it is impossible that $\{i\mid \vect\X[i]\cap \Phi_{\infty,q}(\B)\neq 0\}=\{j,j+1\}$. It sufficies to show that the corresponding t-structure on $\D^b(\X')\simeq \D^b(\SSS^{\perp_\A})$ is not a shift of HRS-tilt with respect to any torsion pair $(\T',\F')$ in $\coh\X'$ with $\coh_0\X'\subsetneq \T'\subsetneq \coh\X'$ (see Remark~\ref{perp torsion pair}). 
	Assume for a contradiction that it was. Since $\X'$ is a weighted projective line of domestic type, by Corollary~\ref{domestic torsion pair}, $\F'[1]*\T'$ would be of finite length. Then so would $\Phi_{\infty,q}(\B)$, a contradiction. This finishes the proof.  
\end{proof}

In light of Lemma~\ref{t-str recollement bijection}, we can already see certain bijective correspondence from our theorems for bounded t-structures whose heart is not of finite length.  In the following corollary, we identify $\Z$ as the group of autoequivalences of $\D^b(\X)$ generated by the translation functor, which acts freely on the set of bounded t-structures on $\D^b(\X)$.
\begin{cor} \label{bijection for not length heart}
	\begin{enum2}
	 \item If $\X$ is of domestic type then there is a bijection
	\begin{multline}\label{bijection for not length}
	\{\text{bounded t-structures on $\D^b(\X)$ whose heart is not of finite  length}\}/\Z \longleftrightarrow\\
	\bigsqcup_\SSS\left( \{P\mid P\subset \P^1\}\times \{\text{bounded t-structures on $\pair{\SSS}_\D$}\}\right),
	\end{multline}
	where 
$\SSS$ runs through all equivalence classes of proper collections of simple sheaves.  

	\item If $\X$ is of tubular type then there is a bijection
	\begin{multline}
		\{\text{bounded t-structures on $\D^b(\X)$ whose heart is not of finite length}\}/\Z \longleftrightarrow\\
		\R\backslash \Q\bigsqcup \left(\bar{\Q}\times \bigsqcup_{\SSS}\left( \{P\mid P\subset \P^1\}\times \{\text{bounded t-structures on $\pair{\SSS}_\D$}\}\right)\right),
	\end{multline}
where  
$\SSS$ runs through all equivalence classes of proper collections of simple sheaves.  
\end{enum2}
\end{cor}

Suppose $\X$ is  of domestic or tubular type. 
Corollary~\ref{bijection for not length heart} reduces the classification of bounded t-structure on $\D^b(\X)$ whose heart is not of finite length  to the classification of bounded t-structures on $\pair{\SSS}_\D=\D^b(\pair{\SSS}_\A)$. Recall that if $\SSS\neq \emptyset$ then there are positive integers $m, k_1,\dots, k_m$ such that $\pair{\SSS}_\A\simeq \coprod_{i=1}^m\mod k\vec{\AA}_{k_i}$. By Lemma~\ref{finite rep length heart}, each bounded t-structure on $\D^b(\mod k\vec{\AA}_l)$ has length heart. So we can achieve the latter classification by calculating silting objects or simple-minded collections in $\D^b(\mod k\vec{\AA}_{k_i})$ by virtue of K\"onig-Yang correspondences.

For bounded t-structures on $\D^b(\X)$ with length heart, there is no obvious  bijective correspondence from the recollement in Theorem~\ref{thm: domestic}(1) or Theorem~\ref{thm: tubular}(4). Recall that $\D^b(\X)$ is triangle equivalent to $\D^b(\Lambda)$ for a canonical algebra $\Lambda$, whose global dimension is at most $2$.  
So  the powerful K\"onig-Yang correspondences  are still applicable. We can try to compute the collections of simple objects in the heart from the recollements using Proposition~\ref{simple in heart}. Instead, we can try to compute silting objects in $\D^b(\X)$ from these recollements using \cite[Corollary 3.4]{LVY}.  
 
 Anyway, for $\X$ of tubular type, since $S^{\perp_\A}\simeq \coh\X'$ for some weighted projective line of domestic type, Theorem~\ref{thm: tubular}(4) reduces  the combinatorics in classification of bounded t-structures on $\D^b(\X)$ with length heart to that in the classification of bounded t-structures on $\D^b(\X')$ with length heart; 
 for $\X$ of domestic type with weight seqence $(p_1,p_2,p_3)$, Theorem~\ref{thm: domestic}(1) reduces the combinatorics in the classification of bounded t-structures on $\D^b(\X)$ with length heart  to that in the classification of bounded t-structures on $\O^{\perp_\D}=\D^b(\O^{\perp_\A})\simeq \D^b(k[p_1,p_2,p_3])$ (by Theorem~\ref{line bundle perp}(2)), where $k[p_1,p_2,p_3]$ is the path algebra of the equioriented star quiver $[p_1,p_2,p_3]$ (a Dynkin quiver here). 

All in all, for $\X$ of domestic or tubular type, the combinatorics in the classifiction of bounded t-structures on $\D^b(\X)$ can be reduced to that in the classification of bounded t-structures on $\D^b(\Lambda)$ for representation-finite finite dimensional hereditary algebras $\Lambda$.    
	
	The following example  recovers the description of bounded t-structures on $\D^b(\P^1)$ in \cite[\S 6.10]{GKR}.
	\begin{exm}\label{t-str on P^1}
		Let $\X$ be of trivial weight type $(p_1,\dots, p_t)$, that is, each $p_i=1$, and thus $\coh\X\simeq \coh\P^1$.  Then each indecomposable object in $\A=\coh\X$ is isomorphic to either a torsion sheaf $S^{[m]}$ supported at some point $\lambda\in \P^1$ for some $m\in \Z_{\geq 1},$ or a line bundle $\O(n\vec{c})$ for some $n\in\Z$.
		By Theorem~\ref{thm: domestic}, a bounded t-structure whose heart is not a length category  is up to shift of the form $(\D_\A^{\leq -1}*\T_P, \F_P[1]*\D_\A^{\geq 0})$ for some $P\subset \P^1$, where \[(\T_P, \F_P)= (\add\{ \coh_\lambda \X\mid \lambda\in P\},\,\, \add\{\O(n\vec{c}), \coh_\lambda\X\mid  n\in\Z, \lambda\notin P \}).\]
To obtain bounded t-structures with length heart, it is easy enough to compute  silting objects directly. Each basic silting object is up to shift of the form $\O(n\vec{c})\oplus \O((n+1)\vec{c})[l]$ for some $n\in \Z, l\geq 0$. Such an object is a tilting object iff $l=0$.
The t-structure corresponding to the silting object $\O(n\vec{c})\oplus \O((n+1)\vec{c})[l]$ has heart \[\left\{\begin{array}{ll} \add\{\O(n\vec{c}), \O((n-1)\vec{c})[l+1]\}\simeq \mod k\coprod \mod k& \text{if}\,\, l>0,\\ \add\{\coh_0\X\cup \{\O(q\vec{c})[1], \O(m\vec{c})\mid q<n, m\geq n \}\}\simeq \mod k(\bullet\rightrightarrows \bullet) & \text{if}\,\,l= 0.\end{array}\right.\]
\end{exm}

\subsection{Torsion pairs in $\coh\X$ revisited}\label{sec: torsion final} We can now give a more clear description of torsion pairs in $\coh\X$ since  torsion pairs are in bijective correspondence with certain t-structures.

\begin{prop}\label{domestic torsion pair}
	Suppose $\X$ is of domestic type. Each torsion pair $(\T,\F)$ in $\coh\X$ fits into exactly one of the following types:
	\begin{enum2}
	\item $(\T,\F)$ is induced by some tilting sheaf, that is, there is a tilting sheaf $T$ such that \[\T=\{E\in \coh(\X)\mid \ext^1(T,E)=0\},\quad \F=\{E\in \coh(\X) \mid \hom(T,E)=0\}.\]

	\item either $\T\subset \coh_0\X$ or $\F\subset \coh_0\X$, and thus $(\T,\F)$ is of the form given in Lemma~\ref{torsion pair contained in torsion part}.
	\end{enum2}
\end{prop}

\begin{proof}
	Note that $\T\nsubseteq \coh_0\X$ and $\F\nsubseteq \coh_0\X$ iff both $\T$ and $\F$ contain nonzero bundles. So in this case the tilted heart $\B=\F[1]*\T$ satisfies $\{i\mid \vect\X[i]\cap \B\neq 0\}=\{0,1\}$. By Corollary~\ref{domestic length heart}, $\B$ is of finite length. 
	Then by  Proposition~\ref{length torsion pair}, $(\T, \F)$ corresponds to a tilting sheaf $T$, which is exactly the one induced by $T$. 
\end{proof}

\begin{prop}\label{tubular torsion pair}
Suppose $\X$ is of tubular type. Each torsion pair $(\T,\F)$  in $\coh\X$ fits into  exactly one of the following types:

	\begin{enum2}
	\item $(\T,\F)$ is induced by a tilting sheaf, that is, there is a tilting sheaf $T$ such that \[\T=\{E\in \coh(\X)\mid \ext^1(T,E)=0\},\quad \F=\{E\in \coh(\X) \mid \hom(T,E)=0\}.\]

	\item for some $\mu\in \R\backslash \Q$, $(\T,\F)=(\coh^{>\mu}\X, \coh^{<\mu}\X)$;

	\item for some $\mu \in \bar{\Q}$, there exists a torsion pair $(\T_\lambda,\F_\lambda)$ in $\coh^\mu_\lambda\X$ for each $\lambda\in\P^1$ such that \[\T=\add\{\coh^{>\mu}\X, \T_\lambda\mid \lambda\in\P^1\}, \quad \F=\add\{\F_\lambda, \coh^{<\mu}\X\mid \lambda\in\P^1\};\]

\item $\F\subset \coh_0\X$ and thus $(\T,\F)$ is of the form given in Lemma~\ref{torsion pair contained in torsion part}(2).
	\end{enum2}
\end{prop}
\begin{proof}
	Consider the HRS-tilt $(\D^{\leq 0}_\B, \D^{\geq 0}_\B)$ with heart $\B=\F[1]*\T$. 
	 Obviously types (2), (3) and (4) form disjoint classes. If $(\T,\F)$ is a torsion pair of type (2) or (3) or (4) then either there is no nonzero $\D^{\leq 0}_\B$-projective or all indecomposable $\D_\B^{\leq 0}$-projectives have the same slope and hence $\F[1]*\T$ is not of finite length by Lemma~\ref{tubular length ext-proj}. Thus types (2), (3) and (4) are disjoint from type (1) by Proposition~\ref{length torsion pair}. 
	 Conversely, suppose that $(\T,\F)$ is a torsion pair in $\coh\X$ such that  $\B$ is not of finite length. We want to show that $(\T,\F)$ is of type (2), (3) or (4).  

	 We apply Theorem~\ref{thm: tubular}. If $(\D_\B^{\leq 0},\D_\B^{\geq 0})$ is a t-structure of type Theorem~\ref{thm: tubular}(1) resp. Theorem~\ref{thm: tubular}(2) then obviously $(\T,\F)$ is of type (2) resp. (3). Otherwise, $(\D_\B^{\leq 0}, \D_\B^{\geq 0})$ is of  type Theorem~\ref{thm: tubular}(3). 
	 Denote \[\widetilde{\B}=\Phi_{\infty,q}(\B)=\Phi_{\infty,q}(\F[1]*\T),\] where $q$ is the unique element in $\bar{\Q}$ asserted in Theorem~\ref{thm: tubular}(3). 
	 From the proof of Theorem~\ref{thm: tubular}, we see that $\{i\mid \vect\X[i]\cap \widetilde{\B}\neq 0\}=\{j\}$ for some $j$. 
	 If $\F\subset \coh_0\X$ then $(\T,\F)$ is of type (4).  
	 Suppose $\F$ contains nonzero bundles. Then by Lemma~\ref{cotilting torsion theory}, $\coh^\mu\X\subset \F$ for $\mu\ll q$.  
	 Now that $\coh^\mu\X[1]\subset \F[1]\subset \B$, we have $\vect\X[1]\cap \widetilde{\B}\neq 0$ by Lemma~\ref{fractional linear}(2). Hence $j=1$. Moreover, an indecomposable sheaf $E$ such that $\Phi_{\infty,q}(E)\in \D^b(\coh_0\X)$ has slope $\mu(E)=q$.    
	 It follows that $\widetilde{\B}\subset\A[1]*\coh_0\X\subset \A[1]*\A$. Thus $\widetilde{\B}= \widetilde{\F}[1]*\widetilde{\T}$, where $(\widetilde{\T},\widetilde{\F})$ is the torsion pair \[(\add\{\widetilde{\T}_\lambda\mid \lambda\in  \P^1\},\,\, \add\{\vect\X, \widetilde{\F}_\lambda\mid \lambda\in \P^1\})\] for some torsion pair $(\widetilde{\T}_\lambda, \widetilde{\F}_\lambda)$ in $\coh_\lambda\X$.  
	 Let  \[(\T_\lambda, \F_\lambda)=(\Phi_{q,\infty}(\tilde{\T_\lambda}), \Phi_{q,\infty}(\tilde{\F_\lambda})),\] which is a torsion pair in $\coh^q_\lambda\X$. Then we have \[(\T,\F)=(\add\{\coh^{>q}\X, \T_\lambda\mid \lambda\in\P^1\}, \,\,\add\{\F_\lambda, \coh^{<q}\X\mid \lambda\in\P^1\}),\] which is of type (3). We are done.
\end{proof}

\section{Derived equivalence}\label{sec: der equiv}
\subsection{Serre functor and derived equivalence}

The main theorem of \cite{SR} states that given a finite dimensional hereditary algebra $\Lambda$ and a bounded t-structure  $(\D^{\leq 0},\D^{\geq 0})$ with heart $\B$ on $\D^b(\Lambda)$, the inclusion $\B\monic \D^b(\Lambda)$ extends to a derived equivalence $\D^b(\B)\simeq \D^b(\Lambda)$ iff the Serre functor of $\D^b(\Lambda)$ is right t-exact with respect to the t-structure $(\D^{\leq 0},\D^{\geq 0})$. This motivates us to consider the following 
\begin{assertion}\label{ass}
	For a Hom-finite $k$-linear triangulated category $\D$ with a Serre functor and a bounded t-structure $\tstr$ on $\D$ with heart $\B$, the inclusion of $\B$ into $\D$ extends to an exact equivalence $\D^b(\B)\simeq \D$ iff the Serre functor is right t-exact.
\end{assertion}	
	
The necessity of Assertion~\ref{ass} always holds by \cite[Corollary 4.13]{SR} whereas \cite[Example 9.4, Example 9.5]{SR} show that the sufficiency does not hold in general. We put it in the form only to stress the role of the Serre functor. Hopefully there would exist more classes of triangulated categories such that Assertion~\ref{ass} hold. Observe that if $\T$ is a  $k$-linear triangulated category that is triangle equivalent to $\D$ then Assertion~\ref{ass} holds for $\T$ iff it holds for $\D$. 

	To give an application of our results on bounded t-structures on the bounded derived category $\D^b(\X)$ of coherent sheaves over a weighted projective line $\X$, we will prove the following 

\begin{thm}\label{der equiv}
	If $\X$ is  of domestic or tubular type then Assertion~\ref{ass} holds for $\D=\D^b(\X)$. 
\end{thm}

Since the result of \cite{SR} embraces the wild case, it is tempting to make the following 

\begin{conj}\label{der equiv all case}
Given an arbitrary weighted projective line $\X$, Assertion~\ref{ass} holds for $\D=\D^b(\X)$.
\end{conj}

We will see in Lemma~\ref{concentrated equiv} that this does hold for a certain class of t-structures on $\D^b(\X)$.

Recall that for $\X$ of domestic type, $\coh\X$ is derived equivalent to $\mod\, \Gamma$ for a tame hereditary algebra $\Gamma$. Thus the conclusion for this case is already covered by \cite{SR}. The new part of Theorem~\ref{der equiv} is for the tubular case. Recall that a tubular algebra $\Lambda$, introduced by Ringel in \cite{Ringel}, can be realized as  the endomorphism algebra of a tilting sheaf over a weighted projective line of tubular type. In particular,  $\D^b(\Lambda)$ is triangle equivalent to $\D^b(\X)$ for some weighted projective line $\X$ of tubular type. So Theorem~\ref{der equiv} yields the following  
\begin{cor}\label{tubular serre}
	Assume that $k$ is an algebraically closed field. Assertion~\ref{ass} holds for $\D=\D^b(\Lambda)$ where $\Lambda$ is a tubular algebra over $k$. 
\end{cor}

Here let us review some necessary background. Let $\D$ be a triangulated category equipped with a bounded t-structure whose heart is denoted by $\B$. An exact functor $F: \D^b(\B)\ra \D$ is called \emph{a realization functor} if $F$ is t-exact and the restriction $F_{\mid \B}: \B\ra \B$ is isomorphic to the identity functor of $\B$. This is a reasonable functor but the existence of such a functor is a problem.  
 By virtue of the filtered derived category, \cite[\S 3.1]{BBD} constructed a realization functor for arbitrary bounded t-structure on  a triangulated subcategory  of $\D^+(\A)$, where $\A$ is an abelian category with enough injectives. \cite{B} abstracted this theme and  introduced the notion of a filtered triangulated category.  
 Given a triangulated category $\D$ with  a filtered triangulated category over it (see \cite[Appendix]{B} for the precise definition),  \cite[Appendix]{B} constructed a  realization functor for arbitrary bounded t-structure on $\D$. Recently, \cite[\S 3]{CR} showed that an algebraic triangulated category indeed admits a filtered triangulated category over it and so  generally we have 
 \begin{prop}[{\cite{CR}}]
A realization functor exists
for any bounded t-structure on an algebraic triangulated category.
 \end{prop}
 A realization functor is not necessarily an equivalence. For example, Example~\ref{t-str on P^1} tells us  that there is a bounded t-structure on $\D^b(\P^1)$ with heart equivalent to $\mod k\coprod \mod k$ but definitely $\mod k\coprod \mod k$ is not derived equivalent to $\coh\P^1$.  The following lemma helps us determine when a realization functor is an equivalence.  

 \begin{lem}[{\cite[Lemma 1.4]{B}}]\label{heart equiv}
	Let $\D_1, \D_2$ be two triangulated categories with bounded t-structures. Suppose $\A_1, \A_2$ are the hearts respectively.  Let $F:\D_1\ra \D_2$ be an exact functor such that $F$ is t-exact and $F_{|\A_1}:\A_1\ra \A_2$ is an equivalence.
	The following are equivalent:
	\begin{enum2}
	\item $F: \D_1\ra \D_2$ is an equivalence;

	\item For each $A,B\in \A_1$, the map $F : \hom^n_{\D_1}(A,B)\ra \hom^n_{\D_2}(F(A),F(B))$ is an isomorphism.

	If $\D_1=\D^b(\A_1)$ then there is an additional equivalent condition:
	
	\item For any $A,B\in \A_1$, $n>0$ and $f\in \hom^n_{\D_2}(F(A),F(B))$, there exists a monomorphism $B\monic B'$ in $\A_1$ effacing $f$.
	\end{enum2}
\end{lem}
As remarked in \cite[Remarque 3.1.17]{BBD}, we have always \[\hom_{\D^b(\A_1)}^n(A,B)\overset{\sim}{\ra} \hom_{\D_2}^n(F(A),F(B))\] for  $A,B\in\A_1$ and $n\leq 1$.  

Although we don't know the uniqueness of a realization functor, if some realization functor $F_1: \D^b(\B)\ra \D$ is an equivalence then any realization functor $F_2: \D^b(\B)\ra \D$ is an equivalence by Lemma~\ref{heart equiv}. 
So it makes sense to say that the inclusion $\B\monic \D$ extends to an exact equivalence $\D^b(\B)\simeq \D$ if some realization functor $F: \D^b(\B)\ra \D$ is an equivalence. 

If there exists an exact equivalence $H: \D^b(\B)\overset{\sim}{\ra} \D$ which is moreover t-exact then any realization functor $F: \D^b(\B)\ra \D$ is an equivalence; given an exact autoequivalence $\Phi$ of $\D$, there exists a realization functor $F:\D^b(\B)\ra \D$  iff there exists a realization functor $G: \D^b(\Phi(\B))\ra \D$ and $F$ is an equivalence iff so is $G$. We will use these trivial facts implicitly.

A remarkable instance of a realization functor being an equivalence is  that for a tilted heart with respect to a (co-)tilting torsion theory introduced in \cite{HRS}.  

\begin{prop}\label{tilt-co-tilt equiv} Suppose  that $\A$ is an abelian category and $(\T,\F)$ a torsion pair in $\A$. 
	If $\T$ is a tilting torsion class or $\F$ is a co-tilting torsion-free class then the inclusion of the tilted heart $\F[1]*\T$ into $\D^b(\A)$ extends to an exact equivalence $\D^b(\F[1]*\T)\overset{\sim}{\ra} \D^b(\A)$.  
\end{prop}
\begin{rmk}
	\begin{enum2}
	\item Proposition~\ref{tilt-co-tilt equiv} is proved originally in \cite{HRS} requiring enough projectives or enough injectives in $\A$ (see \cite[Theorem 3.3]{HRS}). The additional condition is removed in \cite{BB} using the derived category of an exact category (see \cite[Proposition 5.4.3]{BB}). See also \cite{C1} for a short proof via an explicit construction of the equivalence functor. 
	
	\item Generalizing Proposition~\ref{tilt-co-tilt equiv},  \cite{CHZ}  contains a characterization of when the inclusion of the tilted heart $\F[1]*\T$  into $\D^b(\A)$ extends to an exact equivalence for a torsion pair $(\T, \F)$ in $\A$. 
	\end{enum2}
\end{rmk}

\subsection{Reduction via Ext-projectives}

In \cite{SR}, one step of the proof of the main theorem (i.e., Assertion~\ref{ass} holds for $\D^b(\Lambda)$ for a finite dimensional hereditary algebra $\Lambda$)  is reduction via Ext-projectives (more precisely,  the simple top of an Ext-projective). 
The reduction  relies on \cite[Proposition 8.6]{SR}, which seems to work only for  $\D^b(\Lambda)$, where $\Lambda$ is  a finite dimensional hereditary algebra. 
Our proof of Theorem~\ref{der equiv} also uses Ext-projectives to do reduction. In contrast, we will  rely on Proposition~\ref{der equiv reduction} to do reduction, which works for a more general class of triangulated categories, but we have additional assumption on our  Ext-projectives to do reduction and so we have to make efforts to assure the existence of such an Ext-projective object. 
 
 Let $\D$ be a  $k$-linear algebraic triangulated category of finite type  admiting a Serre functor $\SS$. Let $(\D^{\leq 0},\D^{\geq 0})$ be  a bounded t-structure on $\D$ with heart $\B$.   These hypothesis will be retained through this subsection.
 
Let $X\in \D^{\leq 0}$ be an exceptional object such that $\SS X\in \D^{\geq 0}$. 
 By Lemma~\ref{ext-proj},  $X$ is $\D^{\leq 0}$-projective. 
Denote $\D_1:=X^{\perp_\D}={}^{\perp_\D}\SS X$. 
By Lemma~\ref{ext-proj recollement}, $(\D^{\leq 0},\D^{\geq 0})$ is compatible with  the recollement 
\begin{equation}\label{exrec}
	\xymatrix{\D_1\ar[rr]|{i_*}   & &\ar@/_1pc/[ll]|{i^*} \ar@/^1pc/[ll]|{i^!}\D \ar[rr]|{j^*} & &\ar@/_1pc/[ll]|{j_!} \ar@/^1pc/[ll]|{j_*} \pair{X}_\D,}
\end{equation}
where $i_*, j_!$ are the inclusion functors. 
We have $j_*X=\SS X$; for $Y\in \D$, we have $j^*Y=\hom^\bullet(X,Y)\otimes X$.  There are triangles \[\hom^\bullet(X,Y)\otimes X\overset{\text{ev}}{\ra} Y\ra  i_*i^*Y\cra ,\quad i_*i^!Y\ra Y\overset{\text{co-ev}}{\ra} D\hom^\bullet(Y,\SS X))\otimes \SS X\cra .\] $\D_1$ has a Serre functor $\SS_1=i^!\SS i_*$ by Proposition~\ref{admissible Serre functor}. Moreover, 
we have an induced t-structure \[(\D_1^{\leq 0},\D_1^{\geq 0})=(\D_1\cap \D^{\leq 0},\D_1\cap \D^{\geq 0})\] on $\D_1$ with heart $\B_1=\D_1\cap \B$.  
We keep these notation	in the following proposition. 
\begin{prop}\label{der equiv reduction}
 Let $X\in \D^{\leq 0}$ be an exceptional object. Suppose that $X$ and $\SS X$ lie in $\B$ and that  either $X$ or $\SS X$ is simple  in $\B$.  Then 
 \begin{enum2}	
 \item $\SS$ is right t-exact with respect to $(\D^{\leq 0},\D^{\geq 0})$ iff  so is $\SS_1$ with respect to the  t-structure $(\D_1^{\leq 0},\D_1^{\geq 0})$ on $\D_1$;

	\item the inclusion $\B\monic \D$ extends to an exact equivalence $\D^b(\B)\simeq \D$ iff the inclusion $\B_1\monic \D_1$ extends to an exact equivalence $\D^b(\B_1)\simeq \D_1$.
	\end{enum2}	
\end{prop}

\begin{proof}
 Since $X$ is $\D^{\leq 0}$-projective, we have $\ext^1_\B(X,Y)\cong\hom^1_\D(X,Y)=0$ for all $Y\in\B$, and thus $X$ is a projective object in $\B$. Similarly, since $\SS X$ is Ext-injective in $\D^{\geq 0}$, $\SS X$ is an injective object in $\B$. 

	(1) First we show that the right t-exactness of $\SS_1$ implies that of $\SS$.  
	Let $Y\in \D^{\leq 0}$. Then \[i^*Y\in\D_1^{\leq 0},\quad i_*i^* Y\in\D^{\leq 0}\] and we have a triangle \[i_*i^! \SS i_*i^* Y\ra \SS i_*i^*Y\ra D\hom^\bullet(\SS i_*i^* Y, \SS X)\otimes \SS X\cra.\] Note that for $n<0$, we have $X[n]\in \D^{\geq 1}$ and \[\hom^n(\SS i_*i^* Y, \SS X)=\hom^n(i_*i^*Y, X)=0.\] Thus  \[D\hom^\bullet (\SS i_*i^* Y, \SS X)\otimes \SS X=\oplus_{n\geq 0}D\hom^n(\SS i_*i^*Y, \SS X)\otimes \SS X[n]\in \D^{\leq 0}.\] If $\SS_1$ is right t-exact then \[i_*i^!\SS i_*i^*Y=i_*\SS_1 i^*Y\in \D^{\leq 0}.\] Hence $\SS i_*i^*Y\in \D^{\leq 0}$. 
Since $X$ is $\D^{\leq 0}$-projective, we have
\[
	\begin{aligned}
		\SS(\hom^\bullet(X,Y)\otimes X) & =\SS(\oplus_{n\leq 0}\hom^n(X, Y)\otimes X[-n])\\
										& =\oplus_{n\leq 0}\hom^n(X,Y)\otimes \SS X[-n]\\
				&\in \D^{\leq 0}.
	\end{aligned}
\]Then using the triangle \[\SS(\hom^\bullet(X,Y)\otimes X\ra Y\ra i_*i^*Y\cra),\] one knows that $\SS Y\in \D^{\leq 0}$. This shows that $\SS$ is right t-exact.

Now we suppose $\SS$ is right t-exact and  deduce the equivalence between the right t-exactness of $\SS_1$ and the condition that for each $Y\in \B_1$, the co-evaluation map \[H^0(\SS i_*Y)\lra D\hom(H^0(\SS i_*Y),\SS X)\otimes \SS X\] is an epimorphism in $\B$. This equivalence will yield the desired implication, as we will see. 
For $Y\in \B_1$, $\SS_1 Y=i^!\SS i_* Y$ fits into the triangle \[i_*\SS_1 Y\ra \SS i_*Y\ra D\hom^\bullet(\SS i_*Y,\SS X)\otimes \SS X\cra .\] 
Since $i_*Y, X\in \B$, we have \[D\hom^\bullet(\SS i_*Y,\SS X)\otimes \SS X =\oplus_{m\geq 0}D\hom( i_*Y, X[m])\otimes \SS X[m] \in \D^{\leq 0};\] since $\SS$ is right t-exact, we have $\SS i_*Y\in \D^{\leq 0}$. 
	Consider the commutative diagram \[\xymatrix@R-5pt{Z_1\ar[r]\ar[d] & \tau_{\leq -1} \SS i_*Y\ar[r]\ar[d] & \oplus_{m>0}\hom(\SS i_*Y,\SS X[m])\otimes \SS X[m]\ar[d]\\
	i_*\SS_1 Y\ar[r]\ar[d] & \SS i_*Y \ar[r]^{\text{co-ev}\qquad\qquad}\ar[d] & D\hom^\bullet(\SS i_*Y,\SS X)\otimes \SS X\ar[d]\\
	Z_2\ar[r] & H^0(\SS i_*Y)\ar[r]^{\text{co-ev}\qquad\qquad} & D\hom(H^0(\SS i_*Y), \SS X)\otimes \SS X,}\] where rows and columns are distinguished triangles.  
	Then $Z_1\in \D^{\leq 0}$ and hence $i_*\SS_1 Y\in\D^{\leq 0}$ iff $Z_2\in \D^{\leq 0}$. By the triangle \[Z_2\ra H^0(\SS i_*Y)\ra D\hom(H^0(\SS i_*Y),\SS X)))\cra,\] we have $Z_2\in \D^{[0,1]}$. Taking cohomology tells us that $Z_2\in \B$ iff the morphism $H^0(\SS i_*Y)\ra D\hom(H^0(\SS i_*Y),\SS X)\otimes \SS X$ is epic in $\B$. Hence we have the claimed equivalence that $\SS_1$ is right t-exact iff for each $Y\in \B_1$, the co-evaluation map \[H^0(\SS i_*Y)\ra D\hom(H^0(\SS i_*Y),\SS X)\otimes \SS X\] is epic in $\B$.

	If $\SS X$ is simple in $\B$ then 
	clearly the co-evaluation map is an epimorphism. 
	If $X$ is simple in $\B$ then $X$ is a simple projective. Hence for $Y\in \B_1$, \[\hom(H^0(\SS i_*Y),\SS X)\cong \hom(\SS i_*Y, \SS X)\cong \hom(i_*Y,X)=0\] and so the co-evaluation map is  also an epimorphism. 

	(2) If $X$ is simple in $\B$ then for $Y\in \B$, the evaluation map $\hom(X,Y)\otimes X\ra Y$ is a monomorphism in $\B$. Therefore  \[i_*i^* Y  =\cone(\hom^\bullet(X,Y)\otimes X\ra Y) =\cone(\hom(X,Y)\otimes X\ra Y)\] coincides with  the cokernel of the evaluation map \[\hom(X,Y)\otimes X\ra Y\] in $\B$, whence $i^*Y\in \B_1$.
	It follows that $i^*$ is t-exact and restricts to an exact functor ${i^*}_{|\B}: \B\ra \B_1$, which is  left adjoint to the inclusion $\iota=i_{*|\B_1}: \B_1\monic \B$.
	This implies that the inclusion $\iota: \B_1\monic \B$ extends to a fully faithful exact functor $\D^b(\iota): \D^b(\B_1)\monic \D^b(\B)$. 
	Similarly, if $\SS X$ is simple in $\B$ then $i^!$ is t-exact and restricts to an exact functor ${i^!}_{|\B}: \B\ra \B_1$. This also implies that the inclusion $\iota=i_{*|\B_1}: \B_1 \monic \B$ extends to a fully faithful embedding  $\D^b(\iota): \D^b(\B_1)\monic \D^b(\B)$. 
  In either case, we have a fully faithful functor  $\D^b(\iota): \D^b(\B_1)\monic \D^b(\B)$.

  Let $F: \D^b(\B)\ra \D$ be a realization functor. Note that $F$ maps the essential image of $\D^b(\B_1)$ in $\D^b(\B)$ into $\D_1$ and $F_1:=F\circ \D^b(\iota): \D^b(\B_1)\ra \D_1$ is a realization functor. We now show our assertion.
	
  \nec If $F$ is an equivalence then for any $Y_1,Y_2\in \B_1$, we have \[\hom^n_{\D^b(\B_1)}(Y_1,Y_2) \overset{\sim}{\ra} \hom^n_{\D^b(\B)}(Y_1,Y_2) \overset{\sim}{\ra} \hom_\D^n(Y_1,Y_2) = \hom^n_{\D_1}(Y_1,Y_2).\]
	Hence $F_1$ is an equivalence. 

\suf  Assume that $F_1: \D^b(\B_1)\ra \D_1$ is an equivalence.  
Since both $\D^b(\B)$ and $\D$ are generated by $\{X\}\cup \B_1$ and also by $\{\SS X\}\cup \B_1$, to show that $F$ is an equivalence, it sufficies to show that $F$ induces an isomorphism \[(*)\quad \quad\hom^n_{\D^b(\B)}(Y_1, Y_2)\overset{\sim}{\ra}\hom_\D^n(Y_1,Y_2)\] for each $Y_1\in \{X\}\cup \B_1, Y_2\in \{\SS X\}\cup \B_1$. 
	$(*)$ always holds for $n\leq 1$ and so we need to show that $(*)$ holds for $n\geq 2$. 
	Since $F_1:\D^b(\B_1)\ra \D_1$ is an equivalence, $(*)$ holds for $Y_1,Y_2\in \B_1$. Since $X$ is Ext-projective in $\D^{\leq 0}$ and projective in $\B$, \[\hom^n_\D(X,Y_2)=0=\hom^n_{\D^b(\B)}(X,Y_2)\] for $Y_2\in \{\SS X\}\cup \B_1$ and $n\geq 1$; since $\SS X$ is Ext-injective in $\D^{\geq 0}$ and injective in $\B$, we have \[\hom^n_\D(Y_1,\SS X)=0=\hom^n_{\D^b(\B)}(Y_1, \SS X)\]  for $Y_1\in \{X\}\cup \B_1$ and $n\geq 1$. This finishes the proof. 

\end{proof}

We use the following fact to find an  object satisfying the assumption of Proposition~\ref{der equiv reduction}.  For an exceptional  object $X\in \D$,  denote $M_X=\cocone(X\overset{\eta}{\ra} \SS X)$, where $\eta$ is a nonzero morphism. Since $\hom(X,\SS X)\cong D\hom(X,X)=k$, $M_X$ is  up to isomorphism independent of the choice of $\eta$. 
\begin{lem}\label{simple in A_t heart}
	Let $X$ be an exceptional Ext-projective object  in $\D^{\leq 0}$. With the above notation, if $M_X\in\D^{\leq 0}$ then $\SS X$ is a simple object in $\B$; if $M_X\in \D^{\geq 1}$ then $X$ is a simple object in $\B$. In particular, if $M_X[l]\in \B$ for some $l$ then either $X$ or $\SS X$ is simple in $\B$.
\end{lem}
\begin{proof} 
	We will use the recollement \eqref{exrec}, with which the t-structure $(\D^{\leq 0}, \D^{\geq 0})$ is compatible. 
Denote by $(\D_2^{\leq 0}, \D_2^{\geq 0})$ the corresponding t-structure on $\pair{X}_\D\simeq \D^b(k)$. Since $j^*X=X\in \D_2^{\leq 0}$ and $j^*\SS X=X\in \D_2^{\geq 0}$, we know that the heart of $(\D_2^{\leq 0}, \D_2^{\geq 0})$ is $\add\,X$. 
Then by Proposition~\ref{simple in heart}, $j_{!*}X$ is simple in $\B$ and $j_{!*}X$ fits into the two triangles \[i_*\tau_{\leq 0}i^!j_!X\ra j_!X\ra j_{!*}X)\cra, \quad j_{!*}X\ra j_*X\ra i_*\tau_{\geq 0}i^*j_*X\cra .\]
If $M_X\in \D^{\geq 1}$ then \[M_X=i^!M_X\in \D_1^{\geq 1}\] thus \[i_*\tau_{\leq 0}i^!j_!X=i_*\tau_{\leq 0} M_X=0,\quad j_{!*}X\cong j_!X=X;\]  if $M_X\in \D^{\leq 0}$ then \[M_X=i^*M_X\in \D_1^{\leq 0}\]  thus \[i_*\tau_{\geq 0}i^*j_*X=i_*\tau_{\geq 0}(M_X[1])=0,\quad j_{!*} X=\SS X.\]
	These show our first assertion and the second assertion follows easily.
\end{proof}

\begin{rmk}
	If $X,\SS X$ lie in $\B$ then by the definition of $j_{!*}$, we have $j_{!*}(X)=\im(\eta: X\ra \SS X)$, which is 
	the simple top (resp. socle) of $X$ (resp. $\SS X$).
\end{rmk}

\subsection{Proof of Theorem~\ref{der equiv}}
 We prove Theorem~\ref{der equiv} in this subsection. At first, we consider again the category $\A_t$ of finite dimensional nilpotent $k$-representations of the  cyclic quiver $\tilde{\AA}_{t-1}$ with $t$ vertices. The following  lemma refines Lemma~\ref{simple Ext-proj} and makes feasible our induction process.
\begin{lem}\label{simple simple Ext-proj}
	For a bounded t-structure $(\D^{\leq 0},\D^{\geq 0})$ on $\D^b(\A_t)$ with heart $\B$, which is not a shift of the standard one,  there exists a simple object $X\in \A_t$   such that for some some $n\in\Z$, $X[n]$ is $\D^{\leq 0}$-projective and either $X[n]$ or $\SS X[n]$ is a simple object in $\B$, where $\SS$ is the Serre functor of $\D^b(\A_t)$.
\end{lem}
\begin{proof}
	We will use freely the notation introduced at the start of \S\ref{sec: A_t}.  
	Let $\SSS$ be the  proper  collection  of simple objects in $\A_t$ asserted in Proposition~\ref{A_t t-str}. Then for some $S\in \SSS$,  $(\D^{\leq 0}, \D^{\geq 0})$ is compatible with the recollement \[ \xymatrix{S^{\perp_\D}  \ar[rr]|{i_*}   & &\ar@/_1pc/[ll]|{i^*} \ar@/^1pc/[ll]|{i^!}\D=\D^b(\A_t) \ar[rr]|{j^*} & &\ar@/_1pc/[ll]|{j_!} \ar@/^1pc/[ll]|{j_*} \pair{S}_\D,}\]
	where $i_*,j_!$ are the inclusion functors. 
	Denote \[\D_1=S^{\perp_\D}, \D^{\leq 0}_1=\D_1\cap \D^{\leq 0}, \D^{\geq 0}=\D_1\cap \D^{\geq 0}, \B_1=\D_1\cap \B.\] Then $(\D^{\leq 0}_1,\D^{\geq 0}_1)$ is a bounded t-structure on $\D_1$ with heart $\B_1$.

We will use induction on the pair $(t,\sharp \SSS)$ to prove our assertion. As the first step of induction, we consider  arbitrary $t$ and $\sharp\SSS=1$. Then $\SSS=\{S\}$ and, up to a shift  of $\B$, the corresponding t-structure on $S^{\perp_\D}$ has heart $S^{\perp_{\A_t}}$. In particular, $\tau S^{[2]}\in \B$. Since we have  a triangle $\tau S^{[2]}\ra S\ra \tau S[1]\cra $, $S$ is  the desired object by Lemma~\ref{simple in A_t heart}. 
Now suppose $\sharp \SSS>1$. In particular, $t> 2$. 
By the induction hypothesis, there exist some simple $S'\in S^{\perp_{\A_t}}$ and some $l\in \Z$ such that $S'[l]$ is simple in $\B_1$ and is moreover $\D_1^{\leq 0}$-projective or $\D_1^{\geq 0}$-injective. 
Note that a simple object in $S^{\perp_{\A_t}}$ is isomorphic to $\tau S^{[2]}$ or to some simple object in $\A_t$ nonisomorphic to $\tau S, S$. 
If $S'\cong \tau S^{[2]}$ then we have $\tau S^{[2]}[l]\in \B$ and  $S$ is  the desired  object  by Lemma~\ref{simple in A_t heart}. It remains to consider the case when $S'$ is a simple object in $\A_t$ nonisomorphic to $\tau S$ or $S$. Up to a shift of $\B$, we can suppose $l=0$. Then $S'$ is either $\D_1^{\leq 0}$-projective or $\D_1^{\geq 0}$-injective. 

If $S'$ is $\D_1^{\leq 0}$-projective then $\SS_1 S'\in \D_1^{\geq 0}\subset \D^{\geq 0}$, where $\SS_1=i^!\SS i_*$ is the Serre functor of $\D_1=S^{\perp_\D}$. Easy computation shows that \[\SS_1 S' =\left\{\begin{array}{ll} \tau S'[1] & \quad \text{if}\,\, S'\ncong \tau^{-1}S;\\ \tau S^{[2]}[1] & \quad \text{if}\,\, S'\cong \tau^{-1}S.\end{array}\right.\] If $S'\ncong \tau^{-1}S$ then $\tau S'[1]\in \D^{\geq 0}$ and thus $S'$ is $\D^{\leq 0}$-projective. 
Moreover $S'$ is simple in $\B_1$ thus  simple in $\B$, whence $S'$ is the desired object. If $S'\cong \tau^{-1}S$ then $\tau S^{[2]}\in \D^{\geq 1}$. 
Suppose $j^*\B=\add\, S[n]$. 
Then $S\in\D^{\leq n}, \tau S[1]\in \D^{\geq n}$. If $n\geq 1$ then using the triangle $\tau S^{[2]}\ra S\ra \tau S[1]\cra $,  $\tau S[1]\in \D^{\geq n}$ and $\tau S^{[2]}\in \D^{\geq 1}$ imply  $S\in \D^{\geq 1}$. 
Then $S'\cong \tau^{-1}S$ is $\D^{\leq 0}$-projective. Now that $\tau^{-1}S$ is simple in $\B$,   $\tau^{-1}S$ is the desired. If $n\leq 0$ then $\tau S^{[2]}\in \D^{\geq 1}$ and $\tau S[1]\in \D^{\geq n}$ imply $S[n]\in \D^{\geq n}$, whereby yielding $S[n]\in \B$ since we already have $S[n]\in \D^{\leq 0}$. Now that $S[n]\in \B$ and $\tau S[n+1]\in\D^{\geq 0}$,  $S[n]$ is $\D^{\leq 0}$-projective. Moreover, we have $\tau S^{[2]}[n]\in  \D^{\geq 1}$ and thus $S[n]$ is simple in $\B$ by Lemma~\ref{simple in A_t heart}. Therefore $S$ is the desired. 

Similar arguments apply to the case when $S'$ is $\D_1^{\geq 0}$-injective.
The following are some sketchy arguments. Since $t>2$, $\tau^2S\ncong S$. We have \[\SS_1^{-1} S'=i^*\SS^{-1}i_* S'=\left\{\begin{array}{ll} \tau^{-1} S'[-1] & \text{if}\,\,S'\ncong \tau^2 S;\\ \tau S^{[2]}[-1] & \text{if}\,\,S'\cong \tau^2 S.\end{array}\right.\] Suppose $j^*\B=\add\, S[n]$. If $S'\ncong \tau^2 S$ then $\tau^{-1} S'$ is the desired. If  $S'\cong \tau^2 S$ then $\tau S$ is the desired when $n\leq -2$ and $S$ is the desired when $n> -2$.  We are done.
\end{proof}

We show that Assertion~\ref{ass} holds for a class of bounded t-structures on $\D^b(\X)$, where $\X$ is a weighted projective line of arbitrary type. 
\begin{lem}\label{concentrated equiv}
	Let $\X=\X(\udp,\udl)$ be a weighted projective line. Let $(\D^{\leq 0}, \D^{\geq 0})$ be a bounded t-structure on $\D=\D^b(\X)$ whose heart $\B$ satisfies $\{i\mid \vect\X[i]\cap \B\neq 0\}\subset\{j,j+1\}$. Then 
	Assertion~\ref{ass} holds under these additional assumptions.
\end{lem}
\begin{proof}
	We have only to show the sufficiency. Let $\SSS$ be the proper collection of simple sheaves asserted in Proposition~\ref{restrict t-str}. 	
	If $\SSS=\emptyset$ then up to a shift of $\B$ we have $\B=\F[1]*\T$ for some torsion pair $(\T,\F)$ in $\coh\X$.  
	By Lemma~\ref{cotilting torsion theory}, either $\T$ is a tilting torsion class or $\F$ is a cotilting torsion-free class. Then it follows  from Proposition~\ref{tilt-co-tilt equiv} that the inclusion $\B\monic \D^b(\X)$ extends to an exact equivalence $\D^b(\B)\overset{\sim}{\ra}\D^b(\X)$.  
	In particular, if the weight sequence $\udp$ is trivial then there is no exceptional simple sheaves and $\SSS=\emptyset$ and so the assertion also holds in this case. Now we  use induction on the weight sequence $\udp$ and consider a nontrivial weight sequence $\udp=(p_1,\dots,p_n)$. We suppose $\SSS\neq \emptyset$. 

	Take $\lambda\in \P^1$ such that $\SSS_\lambda=\SSS\cap \coh_\lambda\X\neq \emptyset$. By Lemma~\ref{restrict to coh0}, $(\D^{\leq 0}, \D^{\geq 0})$ restricts to a bounded t-structure $(\D^{\leq 0}_\lambda, \D^{\geq 0}_\lambda)$ on $\D^b(\coh_\lambda\X)$. Let $\B_\lambda=\D^b(\coh_\lambda\X)\cap \B$ be its heart. 
	Recall that $\coh_\lambda\X\simeq \A_{p_\lambda}$. By Lemma~\ref{simple simple Ext-proj}, for some exceptional simple sheaf  $S\in \SSS_\lambda$ and some $n\in \Z$, $S[n]$ is $\D_\lambda^{\leq 0}$-projective and either $S[n]$ or $\tau S[n+1]$ is simple in $\B_\lambda$.
	$S[n]\in \D^{\leq 0}, \tau S[n+1]\in \D^{\geq 0}$ imply that $S[n]$ is $\D^{\leq 0}$-projective. Then $(\D^{\leq 0}, \D^{\geq 0})$ is compatible with the recollement  
	\begin{equation}\label{der equiv re}
		\xymatrix{ \D^b(\X')\simeq \D^b(S^{\perp_\A}) \ar[rr]|{i_*}  & &\ar@/_1pc/[ll]|{i^*} \ar@/^1pc/[ll]|{i^!}\D=\D^b(\X) \ar[rr]|{j^*} & &\ar@/_1pc/[ll]|{j_!} \ar@/^1pc/[ll]|{j_*} \pair{S}_\D,}
	\end{equation}
	where $i_*, j_!$ are the inclusion functors, $\X'=\X(\udp',\udl)$ is a weighted projective line with weight sequence \[\udp'=(p_1,\dots, p_{i-1},p_i-1,p_{i+1},\dots, p_n)\] and the exact equivalence $\D^b(\X')\simeq \D^b(S^{\perp_\A})$ is induced by the equivalence $S^{\perp_\A}\simeq \coh\X'$ (see Theorem~\ref{simple perp}). 
	If the Serre functor $\SS=\tau(-)[1]$ is right t-exact then $S[n],\tau S[n+1]\in \B$. One easily shows  \[j_{!*}(S[n])=\im(\eta: S[n]\ra \tau S[n+1])=\left\{\begin{array}{ll} S[n] & \quad\text{if $S[n]$ is simple in $\B_\lambda$}\\ \tau S[n+1] & \quad\text{if $\tau S[n+1]$ is simple in $\B_\lambda$}\end{array}\right.,\] where $\eta: S[n]\ra \tau S[n+1]$ is any nonzero morphism.  Hence either $S[n]$ or $\tau S[n+1]$ is simple in $\B$.
 Then by Proposition~\ref{der equiv reduction}(1), the right t-exactness of the Serre functor $\SS$ of $\D^b(\X)$ implies the right t-exactness of the Serre functor $\SS_1$ of $\D^b(\X')$. 

 Let $\B_1$ be the heart of the corresponding t-structure on $\D^b(\X')$. Since the essential image of $\vect\X'[i]\cap \B_1$ under the sequence of functors $\D^b(\X')\simeq \D^b(S^{\perp_\A})\hookrightarrow \D^b(\X)$ is contained in $\vect\X[i]\cap \B$,  \[\{i\mid \vect\X[i]\cap \B\neq 0\}\subset \{j,j+1\}\quad\text{implies}\quad\{i\mid \vect\X'[i]\cap \B_1\neq 0\}\subset \{j,j+1\}.\] By the induction hypothesis, the right t-exactness of $\SS_1$ implies that the inclusion of $\B_1$ into $\D^b(\X')$ extends a derived equivalence $\D^b(\B_1)\simeq \D^b(\X')$. 
	Then by Proposition~\ref{der equiv reduction}(2), the inclusion $\B\monic \D^b(\X)$ extends to an exact equivalence $\D^b(\B)\simeq \D^b(\X)$.
\end{proof}

We eventually arrive at our proof of Assertion~\ref{ass} for $\D=\D^b(\X)$, where $\X$ is of domestic or tubular type. 
\begin{proof}[Proof of Theorem~\ref{der equiv}]

 We show the sufficiency. Assume that the Serre functor $\SS$ is right t-exact.
 We have shown in Lemma~\ref{concentrated equiv} that  if $\{i\mid \vect\X[i]\cap\B\neq 0\}\subset \{j,j+1\}$ then  Assertion~\ref{ass} holds. If $\X$ is of domestic or tubular type and $\B$ does not satisfy the condition even up to the action of $\aut\D^b(\X)$ then  $\B$ is of finite length by Proposition~\ref{not restrict to coh0}.  The remaining argument goes as in \cite[\S 4]{SR}. By Theorem~\ref{silting t-str}, $(\D^{\leq 0}, \D^{\geq 0})$ corresponds to a silting object $T$ in $\D^b(\X)$.
 In particular, we have an equivalence $F: \B\overset{\sim}{\ra}  \mod\,\End(T)$. If $\SS$ is right t-exact then $T$ is a tilting object by Lemma~\ref{silting tilting}, whose endomorphism algebra has finite global dimension by Proposition~\ref{silting property}. The composition \[\D^b(\B)\overset{\D^b(F)}{\lra} \D^b(\End (T))\overset{-\otimes^L T}{\lra} \D^b(\X)\] is an  exact equivalence which maps $\B$ into $\B$. Thus the inclusion $\B\monic \D^b(\X)$ extends to an exact equivalence $\D^b(\B)\simeq \D^b(\X)$. 
\end{proof}

\begin{rmk}
	We make a final remark on a potential approach to Conjecture~\ref{der equiv all case}, based on the validity of the following 
\begin{conj}\label{conj1}
Let $\X$ be a weighted projective line of arbitrary type. For any bounded t-structure $(\D^{\leq 0}, \D^{\geq 0})$ on $\D^b(\X)$, $\D^{\leq 0}$ contains  no nonzero Ext-projective iff it is a shift of the HRS-tilt with respect to some torsion pair $(\T,\F)$ in $\coh\X$ such that there is no nonzero sheaf $E\in \T$ with $\tau E\in \F$.
\end{conj}
The sufficiency obviously holds. The necessity holds in the domestic and tubular case by our description of bounded t-structures. 
	
The aforementioned potential approach is as follows. Let $(\D^{\leq 0}, \D^{\geq 0})$ be a bounded t-structure on $\D^b(\X)$ with heart $\B$. We can first try to show that Assertion~\ref{ass} holds when $\D^{\leq 0}$ contains no nonzero Ext-projective. For example, if  Conjecture~\ref{conj1} holds, then Assertion~\ref{ass} holds by Lemma~\ref{cotilting torsion theory} and Proposition~\ref{tilt-co-tilt equiv}. 
	Then  we consider the case when $\D^{\leq 0}$ contains a nonzero Ext-projective. Suppose all indecomposable Ext-projectives are torsion sheaves and suppose Conjecture~\ref{conj1} is true. Then the heart $\B$ satisfies $\{i\mid \vect\X[i]\cap \B\neq 0\}\subset \{j,j+1\}$ for some $j\in\Z$ and Assertion~\ref{ass} holds by Lemma~\ref{concentrated equiv}. 
	It remains to consider  the case when some indecomposable bundle $E$ is $\D^{\leq 0}$-projective (up to a shift of $\B$). On one hand, it's possible that our previous approach  still works, i.e., we can still apply Proposition~\ref{der equiv reduction} in some way. 
	On the other hand, since $E$ is exceptional, by Proposition~\ref{bundle perp},   $E^{\perp_{\coh\,\X}}\simeq \mod H$ for some hereditary algebra $H$. Stanley and van Roosmalen's result \cite{SR} may apply here. 

\end{rmk}

 \end{document}